\documentclass{article}
\usepackage{graphicx}
\usepackage[margin=2cm]{geometry} 
\usepackage[english]{babel}

\usepackage{multirow}
\usepackage{amsmath,amsthm,amsfonts,amssymb,mathtools}
\usepackage{mathrsfs}
\usepackage{bbm}
\usepackage{amsthm}
\usepackage{mathrsfs}
\usepackage[title]{appendix}
\usepackage{xcolor}
\usepackage[colorlinks,linkcolor=blue,citecolor=blue,urlcolor=blue]{hyperref}

\usepackage[
  backend=biber,
  natbib=true,
  giveninits=true,
  style=numeric,
  doi=true,
  url=false,
  isbn=false
]{biblatex}
\DeclareNameAlias{default}{family-given}

\usepackage{cleveref}
\usepackage{parskip}
\usepackage{tikz,pgfplots}
\usepackage{caption}
\usepackage{subcaption}

\usepackage{float}
\usepackage[shortlabels]{enumitem}
\usepackage{doi}
\setenumerate{label=(\roman*),itemsep=3pt,topsep=3pt}
\pgfplotsset{compat=1.18}

\pgfmathdeclarefunction{sinc}{1}{%
  \pgfmathparse{ (abs(#1) < 1e-8) ? 1 : sin(#1)/#1 }%
}
\pgfmathdeclarefunction{fejer}{2}{%
  \pgfmathparse{ #1 * ( (sinc(#1*#2/2)) / (sinc(#2/2)) )^2 }%
}

\pgfplotscreateplotcyclelist{fejercolors}{%
  {red!70!black,  thick},
  {blue!60!black,   thick},
  {purple!50!black,  thick},
}

\numberwithin{equation}{section}
\theoremstyle{plain}

\newtheorem{theorem}{Theorem}[section]
\newtheorem{lemma}[theorem]{Lemma}
\newtheorem{corollary}[theorem]{Corollary}
\newtheorem{proposition}[theorem]{Proposition}
\theoremstyle{definition}

\newtheorem{assumption}[theorem]{Assumption}
\newtheorem{remark}[theorem]{Remark}

\makeatother

\theoremstyle{remark}

\newcounter{casecount}
\newtheorem{case}[casecount]{Case}
\AtBeginEnvironment{proof}{\setcounter{casecount}{0}}

\newcounter{stepcount}
\newtheorem{step}[stepcount]{Step}
\AtBeginEnvironment{proof}{\setcounter{stepcount}{0}}

\newcommand{\gspzero}{g_{\mathrm{sp}}}
\newcommand{\gsplambda}{g_{\mathrm{sp}, \lambda}}

\newcommand{\gtemp}{g_{\mathrm{te}}}
\newcommand{\ratio}{\alpha}
\newcommand{\sarea}[1][d-1]{\sigma_{#1}}
\newcommand{\spconstexpt}{C_{\mathrm{sp}, \E}}
\newcommand{\spconstvart}{C_{\mathrm{sp}, \V}}
\newcommand{\tempconstexpt}{C_{\mathrm{te}, \E}}
\newcommand{\tempconstvart}{C_{\mathrm{te}, \V}}
\newcommand{\boxconstexpsp}{C_{\mathrm{box},  \mathrm{sp}, \E}}
\newcommand{\boxconstvarsp}{C_{\mathrm{box}, \mathrm{sp}, \V}}
\newcommand{\boxconstexptemp}{C_{\mathrm{box}, \mathrm{te}, \E}}
\newcommand{\boxconstvartemp}{C_{\mathrm{box}, \mathrm{te}, \V}}
\newcommand\smallo{
  \mathchoice
    {{\scriptstyle\mathcal{O}}}
    {{\scriptstyle\mathcal{O}}}
    {{\scriptscriptstyle\mathcal{O}}}
    {\scalebox{.7}{$\scriptscriptstyle\mathcal{O}$}}
  }
\newcommand{\E}{\mathbb{E}}
\newcommand{\V}{\mathbb{V}}
\newcommand{\ii}{{\mathrm{i}\mkern1mu}}
\newcommand{\fv}{\omega}

\date{}
\title{Statistical inference for the stochastic wave equation based on discrete observations}
\author{Anton Tiepner\footnote{anton.tiepner$\partial$fu-berlin$\bullet$de, Institute of Geological Sciences, Freie Universität Berlin}, Mathias Trabs\footnote{trabs$\partial$kit$\bullet$edu, Department of Mathematics, Karlsruhe Institute of Technology}, Eric Ziebell\footnote{ziebelle$\partial$hu-berlin$\bullet$de,Institut für Mathematik, Humboldt Universität zu Berlin}}

\begin{document}
\maketitle

\begin{abstract}
\noindent The wave speed of a stochastic wave equation driven by Riesz noise on the unbounded multidimensional spatial domain is estimated based on discrete measurements. Central limit theorems for second-order variations of the observations in space, time, and space-time are established. Under general assumptions on the spatial and temporal sampling frequencies, the resulting method-of-moments estimators are asymptotically normally distributed. The covariance structure of the discrete increments admits a closed-form representation involving two different Fejér-type kernels, enabling a precise analysis of the interplay between spatial and temporal contributions.

\smallskip

\noindent\textit{Keywords:} Discrete measurements, Stochastic wave equation, Central limit theorem, Second-order increments, Fejér kernels
\end{abstract}

\section{Introduction}
In this paper, the wave speed $\vartheta>0$ of the stochastic wave equation
\begin{equation}
\label{eq:SPDE}
\begin{cases}
\partial_{tt}^{2}u(t,x)=\vartheta \Delta u(t,x) + \dot{W}_{\beta}(t,x), \quad t \geq 0, \quad x \in \mathbb{R}^{d},\\
u(0,\cdot)=\partial_{t}u(0,\cdot)\equiv0,\\
\E[\dot{W}_{\beta}(t,x)\dot{W}_{\beta}(s,y)]=\delta_0(t-s)|x-y|^{-\beta}, \quad x,y \in \mathbb{R}^{d}, 
\end{cases}
\end{equation}
is estimated from discrete measurements in space and time. The Gaussian noise term $\dot{W}_{\beta}$ is white in time and coloured in space, and the Riesz kernel with fixed known parameter $\beta \in (0, 2 \land d)$ specifies the spatial regularity of the noise.
Spatially coloured noise is a standard assumption in the solution theory for the stochastic wave equation \cite{conusNonLinearStochasticWave2008, dalangExtendingMartingaleMeasure1999, walshIntroductionStochasticPartial1986, balanStochasticWaveEquation2010} and is also a natural modelling assumption for stochastic partial differential equations \cite{MENDEZ2010250, faugeras2015stochastic, newhall2025primer}. In particular, point evaluations of the solution to the SPDE \eqref{eq:SPDE} are well defined in all dimensions, and the regularity, propagation of singularities, and other analytic properties of the solution process have been analysed; see for instance \cite{clarkedelacerdaHittingTimesStochastic2014,dalangLderSobolevRegularitySolution2005,leeLocalNondeterminismExact2019,leePropagationSingularitiesStochastic2022a}.

Stochastic wave equations have been proposed in a range of application areas, including semiconductor laser line width theory \cite{wenzelSemiconductorLaserLinewidth2021} and the modelling of vibrating media and structures \cite{regaNonlinearVibrationsSuspended2004, khoshnevisanMinicourseStochasticPartial2009}, such as suspended or submerged cables, DNA strands, and shock waves. The system \eqref{eq:SPDE} may be viewed as an idealised high-dimensional stochastic wave equation with spatially coloured noise posed on a large bounded domain, where boundary effects are assumed to be negligible.

The time-space structure of stochastic partial differential equations (SPDEs) allows for diverse observation schemes in statistical analysis. The three major examples are spectral, local, and discrete observations. While the statistical literature for SPDEs is largely focused on parabolic equations, research on statistics for the stochastic wave equation remains limited. For the stochastic wave equation and extensions to higher-order elastic and damping operators, spectral measurements were considered by \citet*{liuParameterEstimationHyperbolic2010,liuEstimatingSpeedDamping2008} and \citet*{delgado-vencesStatisticalInferenceStochastic2023}. The local observation scheme, first developed for the non-parametric estimation of an unknown diffusivity function in the stochastic heat equation by \citet*{altmeyerNonparametricEstimationLinear2021}, has only recently been applied to the stochastic wave equation, see \citet*{ziebellNonparametricEstimationStochastic2024}. Similarly to the spectral regimes, the results of \citet*{ziebellNonparametricEstimationStochastic2024}, which were based on the theory of Riemann-Lebesgue operators, were extended to more general second-order stochastic Cauchy problems by \citet*{tiepnerParameterEstimationHyperbolic2025} using the concept of $M, N$-functions. 

Discrete observations of the one-dimensional stochastic heat equation in both space and time were first considered by \citet*{markussenLikelihoodInferenceDiscretely2003},
\citet*{cialencoNoteParameterEstimation2020} and \citet*{bibingerVolatilityEstimationStochastic2019} on bounded spatial domains, and by \citet*{chongHighfrequencyAnalysisParabolic2020} and \citet*{bibingerTrabs2019} in the unbounded case. \citet*{hildebrandtParameterEstimationSPDEs2019} investigate a variety of method-of-moments estimators based on different kinds of first-order space-time variations, allowing a precise statistical analysis of the interplay between the number of spatial and temporal observation points. Similar questions were considered by \citet*{cialencoNoteParameterEstimation2020} using elements of Malliavin calculus. Several extensions to the two-dimensional stochastic heat equation and a detailed statistical analysis on the unit cube were performed by \citet*{tonaki1, tonaki2}. The first framework in arbitrary dimensions, later developed by \citet*{bossertParameterEstimationSecondorder2024a}, builds on realised volatilities and allows the construction of an oracle estimator of the realised volatility within the underlying model.
A statistical analysis of higher-order variations of a discretely sampled semi-linear SPDE was achieved by \citet*{cialencoStatisticalAnalysisDiscretely2024}, highlighting the connections between the regularity of the process and the power variations considered. 

\begin{figure}[t]
    \centering
    \includegraphics[width=0.75\linewidth]{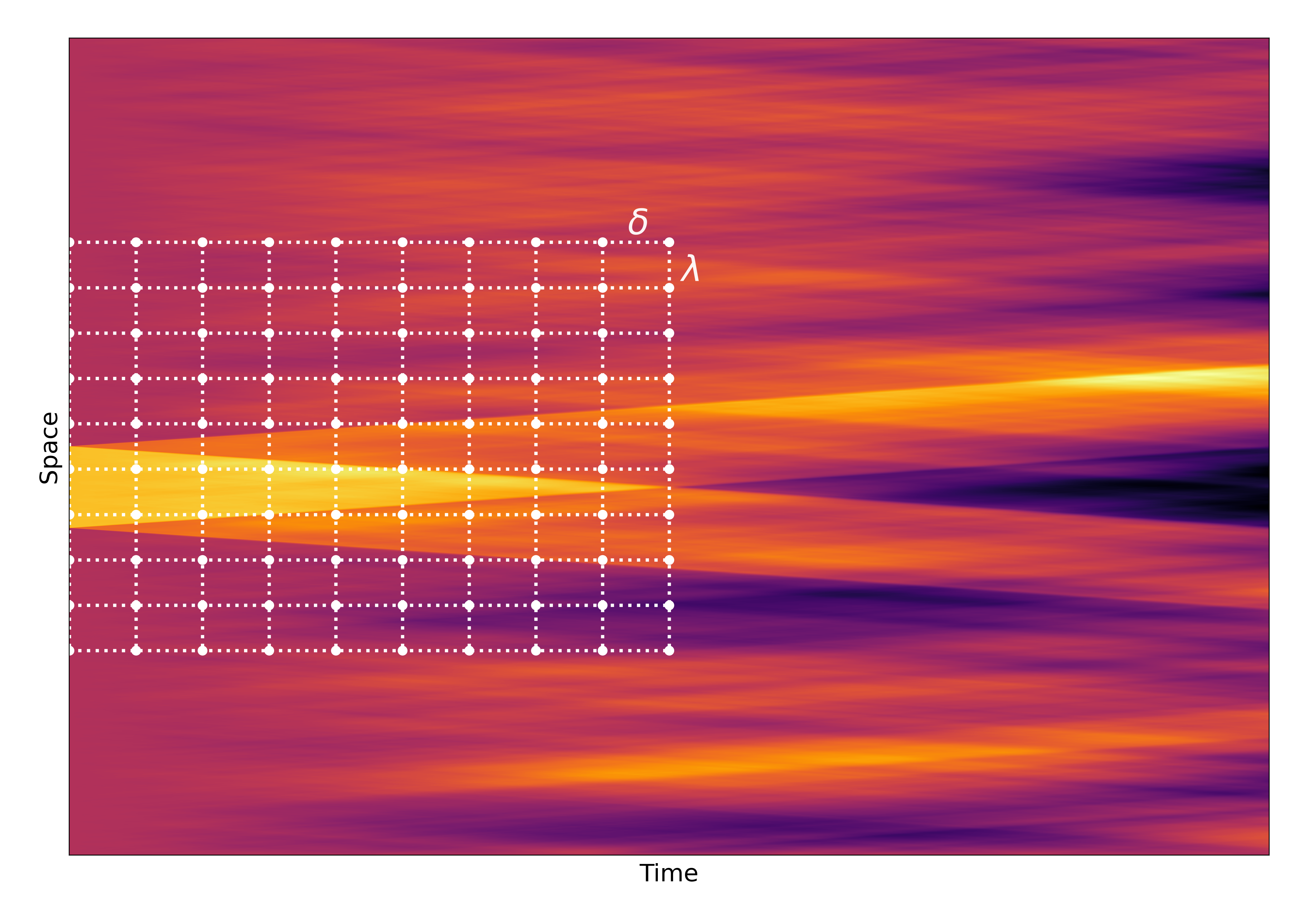}
    \caption{Illustration of the stochastic wave equation in $d=1$ with $\vartheta=\beta=0.5$, and non-zero initial condition together with the space-time observation grid with $n=10$ spatial and $m=10$ temporal observation points. The distance between two adjacent time points and spatial points is denoted by $\delta$ and $\lambda$, respectively. The simulation is based on a semi-implicit Euler scheme on a large bounded domain.}
    \label{fig:placeholder}
\end{figure}

For space-time white noise $\beta=d=1$, the wave speed of a stochastic wave equation was estimated by \citet*{assaadQuadraticVariationDrift2022a} using quadratic variations in time and space. Approximate likelihood ratio tests for the one-dimensional parabolic against hyperbolic alternatives based on the likelihood structure induced by discrete measurements and explicit method-of-moments estimators were analysed by \citet*{markussenLikelihoodInferenceDiscretely2003}. Other investigations of the quadratic variation using the Malliavin calculus include \citet*{khalilSpatialVariationSolution2018}, \citet*{khalilCorrelationStructureQuadratic2018} and \citet*{tudorHighOrderAsymptotic2019}. The Hurst parameter of a fractional-white stochastic wave equation was estimated by \citet*{khalilCorrelationStructureQuadratic2018} and \citet*{shevchenkoQuadraticVariationsFractionalwhite2023}, including a first exemplary CLT
based on box-increments, which combines measurements in both space and time.

We start our investigation by considering equidistant observations in space or time with the other variable remaining fixed. More specifically, we observe either $(u(t,x_k), k=0, \dots , n+1)$ with $x_k=\lambda k \rho$ and a direction $\rho \in \mathbb{R}^{d}\setminus \{0\}$ with $|\rho|=1$, or $(u(t_i, x), i=0, \dots, m+1)$ with $t_i=\delta i$ for $\delta,\lambda > 0$. Whilst analysing second-order variations in space and time, we discover that the variances of either variation are intrinsically related to  partial Fourier sums, Fejér kernels and Cesàro sums, see \citet{grafakosClassicalFourierAnalysis2014}. In particular, increments of the process translate into increments of the covariance function. However, the increments of the trigonometric functions, naturally involved in the covariance structure of the stochastic wave equation, can themselves be expressed through trigonometric functions, leading to Fejér-type kernels emerging within the variance. These insights constitute the theoretical backbone required to prove central limit theorems for variations of the process \eqref{eq:SPDE}, and can be summarised as follows. Denote by $\mathbf{V}_{\mathrm{sp}}$ and  $\mathbf{V}_{\mathrm{te}}$ the realised quadratic second-order variations of our observations in space and time, respectively. Under some general assumptions on the sampling frequencies $\lambda$ and $\delta$, we prove the following central limit theorems:
\begin{equation}
\label{eq:results_intro}
\begin{alignedat}{2}
\sqrt{n}\left(\frac{\lambda^{\beta-2}}{n}\mathbf{V}_{\mathrm{sp}}- \frac{t\spconstexpt}{\vartheta}\right)
  &\xrightarrow{d} N\left(0, \frac{t^{2}}{\vartheta^{2}}\spconstvart\right),
  \qquad && n\rightarrow\infty,\\
\sqrt{m}\left(\frac{\delta^{\beta-3}}{m^{2}}\mathbf{V}_{\mathrm{te}}  -  \frac{\tempconstexpt}{\vartheta^{\beta/2}}\right)
  &\xrightarrow{d} N\left(0, \frac{\tempconstvart}{\vartheta^{\beta}}\right),
  \qquad && m \rightarrow \infty,
\end{alignedat}
\end{equation}
for some positive, universal constants only depending on $\beta$ and $d$. 

Furthermore, we explore box-increments based on the space-time observations $(u(t_i,x_k),i=0,\dots,m+1,k=0,\dots,n+1)$. Both spatial and temporal components factorise through the corresponding Fejér kernels. The asymptotics of the associated second-order variation $\mathbf{V}_{\mathrm{sp,te}}$ depend on the hyperbolic sampling ratio $\ratio=\delta/\lambda$ determining the asymptotic regime. When $\ratio\rightarrow\infty$, the sampling frequency in space is higher and $\mathbf{V}_{\mathrm{sp,te}}$ is driven by the spatial increments, whereas the temporal behaviour dominates when $\ratio\rightarrow0$. In both cases we show a central limit theorem for $\mathbf{V}_{\mathrm{sp,te}}$:
\begin{equation}
\label{eq:results2_intro}
\begin{alignedat}{2}
\sqrt{nm}\left(\frac{\delta^{-1}\lambda^{\beta-2}}{nm^{2}}\mathbf{V}_{\mathrm{sp,te}}-\frac{\boxconstexpsp}{\vartheta}\right)
  &\xrightarrow{d} N\left(0, \frac{\boxconstvarsp}{\vartheta^{2}}\right),
  \qquad && \ratio \rightarrow \infty, \\
\sqrt{nm}\left(\frac{\delta^{\beta-3}}{nm^{2}}\mathbf{V}_{\mathrm{sp,te}}-\frac{\boxconstexptemp}{\vartheta^{\beta/2}}\right)
  &\xrightarrow{d} N\left(0, \frac{\boxconstvartemp}{\vartheta^{\beta}}\right),
  \qquad && \ratio \rightarrow 0,
\end{alignedat}
\end{equation}
The asymptotic normality of method-of-moments estimators based on the considered variations is obtained from the central limit theorems \eqref{eq:results_intro} and \eqref{eq:results2_intro} through the delta method. 

The paper is structured as follows. In \Cref{section:spde_model} the SPDE model is formally introduced. In \Cref{section:spatial_variation}, we treat the estimation problem for $\vartheta$ based on spatial variations of the process $(u(t,x), x \in \mathbb{R}^{d})$ along the line $(r \rho \colon r >0)$. In this chapter, we also encounter the Fejér kernel $\mathfrak{F}_{\mathrm{sp}}$ and introduce the proof strategy employed throughout this work. Following the same structure, we collect all statistical results from observations of the process $(u(t,x), t \geq 0)$ at a fixed spatial point $x \in \mathbb{R}^{d}$ in \Cref{section:temporal_variations}. In \Cref{section:space_time_variation} we combine our findings from the previous sections to analyse second-order increments in both space and time. Only proof ideas and outlines are given in the main part of the paper. Detailed proofs are deferred to \Cref{section:proofs}.

\paragraph*{Notation}
We abbreviate by $|x|=\sqrt{x_1^{2}+\dots +x_d^{2}}$ and $x \cdot y = x^{\top}y$ the usual norm and inner product on the euclidean space $\mathbb{R}^{d}$ for $d \in \mathbb{N}$. The imaginary number is written as $\ii$. Note that $A^{\top}$ is our notation for the transpose of the matrix $A$, and $f'$ will be used for an ordinary derivative of the function $f$. If $f$ is multivariate, the partial derivative in variable $x$ will be denoted by $\partial_{x} f$. Given that $(a_{m})_{m \in \mathbb{N}}$ and $(b_m)_{m \in \mathbb{N}}$ are two sequences of real numbers, we write $a_{m} \sim b_m$ if $a_m/b_m \rightarrow 1$ as $m \rightarrow \infty$ and $a_{m} \lesssim b_m$ if there exists a constant $c>0$ such that $|a_m| \leq c |b_m|$ for all $m \in \mathbb{N}$. Furthermore, we write $a_m \asymp b_m$ if both sequences have the same order of magnitude, i.e.\ $a_m \lesssim b_m$ and $b_m \lesssim a_m$. Abbreviate by $|A|$ the Lebesgue measure of a set $A$. \\
For $p \in [1, \infty]$ the spaces $L^{p}(\mathbb{R}^{d})$ and $L^{p}((0,\infty))$ abbreviate the usual $L^{p}$-spaces. For $f \in L^{1}(\mathbb{R}^{d})$ and $\xi\in\mathbb{R}^d$ we define the Fourier transform as $\mathcal{F}(f)(\xi)=\int_{\mathbb{R}^{d}}e^{\ii \xi \cdot \fv}f(\fv)\mathrm{d}\fv$. Similarly, we write $\mathcal{F}_{c}(f)(\xi)=\int_{\mathbb{R}^{d}}\cos(\xi\cdot\fv)f(\fv)\mathrm{d}\fv$ for the cosine transform. We further define for $\xi\in\mathbb{R}$ the transformations $\mathcal{F}^{+}(f)(\xi)=\int_{\mathbb{R}^{d}}f(\fv)e^{\ii |\fv| \xi}\mathrm{d}\fv$ and likewise $\mathcal{F}^{+}_{c}(f)(\xi)=\int_{\mathbb{R}^{d}}f(\fv)\cos(|\fv|\xi)\mathrm{d}\fv$. We define the unit sphere through $\mathbb{S}^{d-1}=\{x \in \mathbb{R}^{d} \colon |x|=1\}$ with $\sarea=|\mathbb{S}^{d-1}|$ abbreviating the $(d-1)$-dimensional surface area of the unit sphere. The (unnormalised) surface measure on $\mathbb{S}^{d-1}$ is abbreviated by $\sigma$. The Dirac-delta function centred at $x \in \mathbb{R}$ is abbreviated by $\delta_{x}$. For the notation associated with general increments of the functions, i.e.\ $\mathcal{I}$ and $\mathfrak{I}$, we refer to \Cref{section:general_results_on_increments}. We write $\mathrm{sinc}(x)=\sin(x)/x$ for the $\mathrm{sinc}$-function and
$\mathrm{tr}(A)$, $\Vert A \Vert_{2}$ and $\Vert A \Vert_{\infty}$ for the trace, spectral norm and maximum row norm of a matrix $A$, respectively. \\
Subscripts '$\mathrm{te}$' and '$\mathrm{sp}$' refer to temporal and spatial dependencies, and $m,n\in\mathbb{N}$ correspond to the number of temporal and spatial second-order increments with $\delta$ and $\lambda$ being their respective sampling frequencies. The constants in \eqref{eq:results_intro} and \eqref{eq:results2_intro} are explicitly given in \Cref{tab:asymptotic_constants}. We uniquely reserve $i,j\in\mathbb{Z}$ for temporal counting indices and $k,l\in\mathbb{Z}$ for spatial ones. General indices are denoted by $z,p,q\in\mathbb{Z}$.

\section{The SPDE model}
\label{section:spde_model}
In this paper, we study the linear stochastic wave equation
\eqref{eq:SPDE}, where $\vartheta>0$ is the unknown wave speed. The noise $\dot{W}_{\beta}$ is white in time and coloured in space. The solution of the stochastic wave equation admits a random field solution $(u(t,x), t \geq 0, x \in \mathbb{R}^{d})$ when the spatial covariance of the noise is given by the Riesz kernel, i.e.,  $\E[\dot{W}_{\beta}(t,x)\dot{W}_{\beta}(s,y)]=\delta_0(t-s)|x-y|^{-\beta}$ for $t,s \geq 0$ and $x,y \in \mathbb{R}^{d}$ with a fixed known parameter $\beta \in (0, 2\land d)$ specifying the spatial regularity, see \citet{dalangExtendingMartingaleMeasure1999}. If $\beta=d$, the covariance structure of the stochastic wave equation driven by Riesz noise would correspond to the case of space-time white noise. However, the stochastic wave equation driven by space-time white noise only has a pointwise solution when $\beta=d=1$. Classically, the solution to the stochastic wave equation \eqref{eq:SPDE} admits local Hölder regularity in space and time of order $1-\beta/2$, see \citet[Proposition 4.1]{dalangCriteriaHittingProbabilities2010}.
Note that all results are also valid for $d=1$ in the case of space-time white noise. 

Let $G$ be the Green's function associated with the wave equation on $\mathbb{R}^{d}$. Then, the extended martingale measure theory in \citet{dalangExtendingMartingaleMeasure1999} allows us to write the solution as the stochastic integral
\begin{equation*}
u(t,x)=\int_{0}^{t}\int_{\mathbb{R}^{d}}G_{t-s}(x-y)W_{\beta}(\mathrm{d}s,\mathrm{d}y), \quad t \geq 0, \quad x \in \mathbb{R}^{d},
\end{equation*}
where $W_{\beta}$ is the martingale measure induced by the noise $\dot{W}_{\beta}$. While the Green's function of the wave equation becomes increasingly irregular for larger dimensions and depends on the parity of $d \in \mathbb{N}$, see for instance \citet{evansPartialDifferentialEquations2010}, its Fourier transform always admits the representation
\begin{equation}
\label{eq:fourier_transform_greens_function}
\mathcal{F}(G_t)(\fv)=\frac{\sin(t \sqrt{\vartheta}|\fv|)}{\sqrt{\vartheta}|\fv|}=t\mathrm{sinc}(t\sqrt{\vartheta}|\fv|), \quad t \geq 0, \quad \fv \in \mathbb{R}^{d}.
\end{equation}
Itô's isometry from \citet[Theorem 2]{dalangExtendingMartingaleMeasure1999} allows us to exploit the representation \eqref{eq:fourier_transform_greens_function}, yielding an explicit characterisation of the covariance structure of the stochastic wave equation \eqref{eq:SPDE} which is essential throughout the paper.
\begin{proposition}
\label[proposition]{proposition:covariance_representation}
The covariance function of the stochastic wave equation \eqref{eq:SPDE} is characterised by
\begin{equation*}
\begin{aligned}
&\E[u(t,x)u(s,y)]\\
&\quad=\frac{1}{ (2\pi)^{d}\vartheta^{\beta/2}}\int_{0}^{t \land s} \int_{\mathbb{R}^{d}} \cos\left(\frac{(x-y)}{\sqrt{\vartheta}}\cdot \fv \right)\sin((t-r)|\fv|)\sin((s-r)|\fv|)|\fv|^{\beta-d-2}\mathrm{d}\fv \mathrm{d}r\\
&\quad= \frac{1}{ (2\pi)^{d}\vartheta^{\beta/2}} \int_{\mathbb{R}^{d}} \cos\left(\frac{(x-y)}{\sqrt{\vartheta}}\cdot \fv \right) \Phi_{\mathrm{te}, |\fv|}(t,s) |\fv|^{\beta-d-2}\mathrm{d}\fv
\end{aligned}
\end{equation*}
with
\begin{equation}
\label{eq:covariance_sine_integral}
\begin{aligned}
    \Phi_{\mathrm{te}, \xi}({t,s}) &\coloneqq\frac{\sin(|t-s|\xi)-\sin((t+s)\xi)}{4\xi}+\frac{1}{2}(t \land s)\cos((t-s)\xi)\\
    &=\frac{1}{2}\left((t \land s)\cos((t-s)\xi)-\frac{\cos((t \lor s)\xi)\sin((t \land s)\xi)}{\xi}\right),
\end{aligned}
\end{equation}    
for $t,s \geq 0$, $x,y \in \mathbb{R}^{d}$ and $\xi \in \mathbb{R}$.
\end{proposition}
\begin{remark}
\label[remark]{remark:geometric_structure}
\begin{enumerate}
\item[]
\item ($\beta=d=1$): For a formal proof of the representation in this special case, we refer to \citet[Lemma 1]{khalilCorrelationStructureQuadratic2018}. Moreover, there is a convenient geometric intuition for the covariance structure of the stochastic wave equation in this case. Let $\mathcal{B}([0,T] \times \mathbb{R})$ be the Borel $\sigma$-algebra. In the martingale-measure approach \cite{walshIntroductionStochasticPartial1986}, space-time white noise is usually interpreted as a centred Gaussian process $({W}(A), A \in \mathcal{B}([0,T]\times \mathbb{R}))$ with the covariance structure $\E[{W}(A){W}(B)]=|A \cap B|$ for $A, B \in \mathcal{B}([0,T]\times \mathbb{R})$. 
The stochastic wave equation is then given by
\begin{equation*}
u(t,x)=\frac{1}{ \sqrt{2\vartheta}}\int_{0}^{t}\int_{\mathbb{R}}\mathbbm{1}(|x-y| \leq \sqrt{\vartheta}(t-s)) W(\mathrm{d}s, \mathrm{d}y) = \frac{1}{ \sqrt{2\vartheta }}W(\mathcal{C}_{\vartheta}(t,x)), \quad  t >0,\quad  x \in \mathbb{R},
\end{equation*}
with the backward light cone $\mathcal{C}_{\vartheta}(t,x)=\{(s,y)\in (0,\infty) \times \mathbb{R}: 0 \leq s \leq t, |x-y|\leq \sqrt{\vartheta}(t-s)\}$. 
Thus, computing the covariance structure of the stochastic wave equation in $\beta=d=1$ amounts to calculating the area of intersection between two backward light cones:
\begin{equation*}
\E[u(t,x)u(s,y)]=\frac{1}{2 \vartheta}\E[W(\mathcal{C}_{\vartheta}(t,x))W(\mathcal{C}_{\vartheta}(s,y))]= \frac{1}{2 \vartheta} |\mathcal{C}_{\vartheta}(t,x)\cap \mathcal{C}_{\vartheta}(s,y)|.
\end{equation*}
There are two situations in which the covariance is non-zero: \Cref{figure:partial_overlap} visualizes the case of partial overlap, and the case where one light cone is contained within the other is shown in \Cref{figure:full_inclusion}. This leads to the covariance structure
\begin{equation*}
		\E[u(t,x) u(s,y)]= \frac{1}{16 \sqrt{\vartheta}}\mathbbm{1}_{\{|t-s| \leq \frac{|x-y|}{\sqrt{\vartheta}} < (t+s)\}} \Big(t+s-\frac{|x-y|}{\sqrt{\vartheta}}\Big)^{2}+ \frac{1}{4 \sqrt{\vartheta}}\mathbbm{1}_{\{\sqrt{\vartheta}|t-s|>|y-x|\}}(t \land s)^{2}.
\end{equation*}
The geometric structure is not as simple in higher dimensions. The intersection between backwards light cones is no longer itself a backwards light cone. Moreover, when $\beta \in (0, 2 \land d)$, the Riesz noise introduces long-range dependencies to the covariance structure, and even when the associated backwards light cones are disjoint, the covariance no longer becomes zero. 
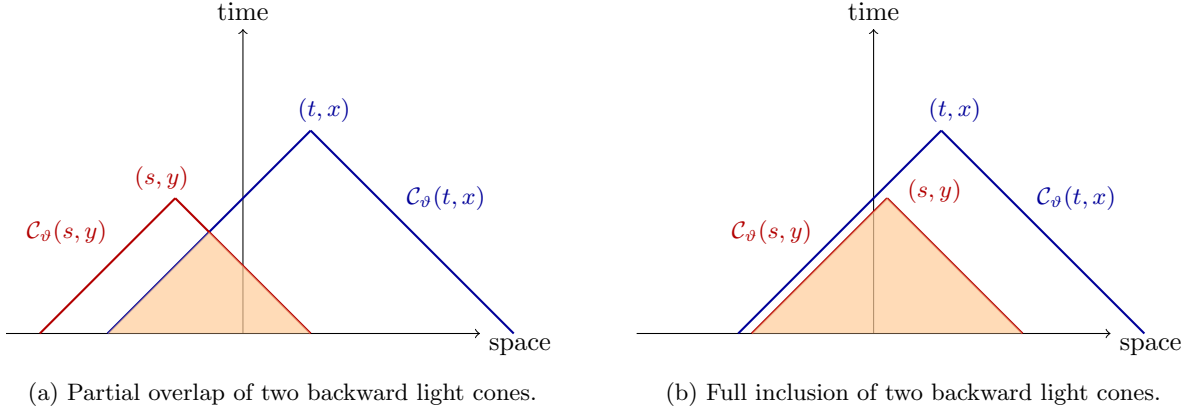
\begin{figure}[tb]
  \centering
  \begin{subfigure}[t]{0.45\textwidth}
    \centering
    \begin{tikzpicture}[scale=0.9]
      \draw[->] (-3.5,0) -- (3.5,0) node[right] at (3.5,-0.2) {space};
      \draw[->] (0,0) -- (0,4.5) node[above] {time};
      \draw[thick,blue!60!black] (1,3) -- (-2,0) -- cycle;
      \draw[thick,blue!60!black] (1,3) -- (4,0);
      \node[blue!60!black] at (1.2,3.3) {\small $(t,x)$};
      \draw[thick,red!70!black] (-1,2) -- (-3,0) -- cycle;
      \draw[thick,red!70!black] (-1,2) -- (1,0);
      \node[red!70!black] at (-1.2,2.3) {\small $(s,y)$};
      \fill[orange!40,opacity=0.7] (-2,0) -- (-0.5,1.5) -- (1,0) -- cycle;
      \node at (-2.6,1.5) {\small \textcolor{red!70!black}{$\mathcal{C}_{\vartheta}(s,y)$}};
      \node at (3,2) {\small \textcolor{blue!60!black}{ $\mathcal{C}_{\vartheta}(t,x)$}};
    \end{tikzpicture}
    \caption{Partial overlap of two backward light cones.}
    \label{figure:partial_overlap}
  \end{subfigure}
  \quad
  \begin{subfigure}[t]{0.45\textwidth}
    \centering
    \begin{tikzpicture}[scale=0.9]
      \draw[->] (-3.5,0) -- (3.5,0) node[right] at (3.5,-0.2) {space};
      \draw[->] (0,0) -- (0,4.5) node[above] {time};
      \draw[thick,blue!60!black] (1,3) -- (-2,0);
      \draw[thick,blue!60!black] (1,3) -- (4,0);
      \node[blue!60!black] at (1.2,3.3) {\small $(t,x)$};
      \draw[thick,red!70!black] (0.2,2) -- (-1.8,0);
      \draw[thick,red!70!black] (0.2,2) -- (2.2,0);
      \node[red!70!black] at (0.9,2.1) {\small $(s,y)$};
      \fill[orange!40,opacity=0.7] (-1.8,0) -- (0.2,2) -- (2.2,0) -- cycle;
      \node at (-1.5,1.5) {\small \textcolor{red!70!black}{$\mathcal{C}_{\vartheta}(s,y)$}};
      \node at (3,2) {\small \textcolor{blue!60!black}{ $\mathcal{C}_{\vartheta}(t,x)$}};
    \end{tikzpicture}
    \caption{Full inclusion of two backward light cones.}
    \label{figure:full_inclusion}
  \end{subfigure}

  \caption{(a) Partial overlap and (b) full inclusion of two backward light cones in \(\beta=d=1\). The orange corresponds to non-zero covariance.}
  \label{figure:combined_lightcones}
\end{figure}
\item \label{remark:scale_parameter} ($x=y$): If we set $x=y$ in \Cref{proposition:covariance_representation}, all integrands are radial. Thus, by switching to polar coordinates, the representation can be simplified considerably: There exists a constant $C_{\beta,d}$ such that for any fixed spatial point $x \in \mathbb{R}^{d}$, we have
	\begin{equation}
        \label{eq:concrete_equation_temporal_covariance_structure}
		\E[u(t,x)u(s,x)]= C_{\beta,d}\vartheta^{-\beta/2} \left( \frac{1}{2(3-\beta)}\left((t+s)^{3-\beta}-|t-s|^{3-\beta}\right)-( t \land s)|t-s|^{2-\beta}\right).
	\end{equation}
For a precise proof, we refer to \Cref{result:temporal_covariance_structure}. \Cref{eq:concrete_equation_temporal_covariance_structure} also shows that the covariance only depends on the unknown parameter through the scalar factor $\vartheta^{-\beta/2}$. Thus, if we add the index $\vartheta$ to the solution of the stochastic wave equation, all members of the family $(u_{\vartheta}(t,x), t \geq 0)$ share the same correlation structure up to a multiplicative factor. Equivalently, we have
\begin{equation*}
(u_{\vartheta}(t,x), t \geq 0) \overset{d}{\equiv} \vartheta^{-\beta/4}(u_{1}(t,x), t \geq 0),
\end{equation*}
see also \citet[Lemma 4.1]{assaadQuadraticVariationDrift2022a}. 
\end{enumerate}
\end{remark}

\section{Spatial variation}
\label{section:spatial_variation}
Suppose that $\rho \in \mathbb{S}^{d-1} \subset\mathbb{R}^{d}$ is a direction of observation and $t > 0$ is a fixed point in time. Let us assume we have access to discrete observations of the process $(u(t,x), x \in\mathbb{R}^{d})$ in space on the line segment $(r \rho, r \in (0, \infty))$. Thus, we observe $u(t, x_k)$ for $x_k=\lambda_n k\rho$ with mesh size $\lambda_n >0$ and $k=0, \dots,n+1$. We usually omit the index $n$ from $\lambda_n$ and simply write $\lambda=\lambda_n$, keeping in mind that $\lambda_n$ may be dependent on the number of spatial points. A visualisation of our spatial observation path is provided by \Cref{figure:visualisation_spatial_points}.
\begin{figure}[tb]
    \centering
    \begin{tikzpicture}[scale=2]

  \def\n{5} 
  \def\lambda{0.6} 
  \def\rx{0.8} 
  \def\ry{0.6} 

  \draw[->] (-0.5, 0) -- (3, 0);
  \draw[->] (0, -0.5) -- (0, 2);

  \draw[dashed] (0,0) circle (1);

  \draw[->, thick, blue!60!black] (0,0) -- ({\rx}, {\ry});
  \node[left] at ({0.9*\rx}, {0.9*\ry}) {\textcolor{blue!60!black}{\(\rho \in \mathbb{S}^1\)}};

  \draw[dashed] (0,0) -- ({\n*\lambda*\rx}, {\n*\lambda*\ry});

  \foreach \k in {0,...,\n} {
    \filldraw[red!70!black] ({\k*\lambda*\rx}, {\k*\lambda*\ry}) circle (0.03);
    
    \ifnum\k=1
      \node[below right] at ({\k*\lambda*\rx}, {\k*\lambda*\ry}) {\(x_1\)};
    \else\ifnum\k=3
      \node[below right] at ({\k*\lambda*\rx}, {\k*\lambda*\ry}) {\(x_k\)};
    \else\ifnum\k=\n
      \node[below right] at ({\k*\lambda*\rx}, {\k*\lambda*\ry}) {\(x_{n+1}\)};
    \fi\fi\fi
  }

  \node at (1.2,1.4) {\((r \rho,\ r > 0)\)};
\end{tikzpicture}
    \caption{Visualisation of the spatial observation points $(x_k)_{k=0}^{n+1}$ for a fixed $n \in \mathbb{N}$, $\lambda \in (0,1)$ in $d=2$.}   \label{figure:visualisation_spatial_points}
\end{figure}
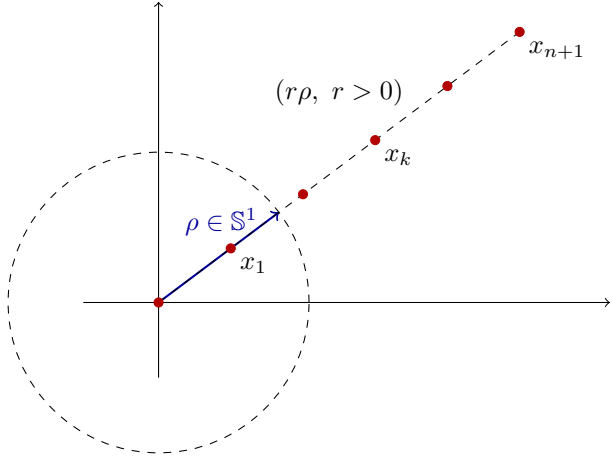
\par As the key statistic for our method-of-moments type approach, we consider the second-order spatial variation defined through
\begin{equation}
\label{eq:second_order_spatial_variation}
\mathbf{V}_{\mathrm{sp}} \coloneqq \sum_{k=1}^{n} (\mathbf{I}_{\mathrm{sp},k})^{2},
\end{equation}
with the second-order spatial increments associated with our observation vector given by
\begin{equation}
\label{eq:spatial_increments}
\begin{aligned}
\mathbf{I}_{\mathrm{sp},k} \coloneqq u(t, x_{k+1})+u(t, x_{k-1})-2u(t, x_k), \quad k=1, \dots,n.
\end{aligned}
\end{equation}
\begin{remark} 
\label[remark]{remark:intuition_second_order_increments}
\begin{enumerate}
\item[]
\item Even in the case of space-time white noise, i.e.\ assuming $\beta=d=1$, second-order spatial variations are natural. We can represent second-order increments as the difference between two first-order increments i.e.\ $\mathbf{I}_{\mathrm{sp},k}=(u(t,x_{k+1})-u(t,x_k))-(u(t,x_k)-u(t,x_{k-1}))$. The covariance structure of first-order spatial increments is given by
\begin{equation*}
\mathrm{Cov}(u(t,x_{k+1})-u(t,x_k), u(t,x_{l+1})-u(t,x_l))=\begin{cases}
 \frac{\lambda}{2 \vartheta}\Big(t-\frac{\lambda}{4 \sqrt{\vartheta}}\Big), \quad &k=l,\\
  - \frac{\lambda^{2}}{8 \vartheta^{3/2}}, \quad & k \neq l.
\end{cases}
\end{equation*}
see for instance \citet*{assaadQuadraticVariationDrift2022a}. Thus, the covariance matrix in this case is a double constant matrix, and the first-order increments remain correlated even when the spatial points involved in the increments are far apart. Since all the off-diagonals of the covariance matrix associated with the first-order increments are the same, the covariance of the second-order increments is a band matrix satisfying $\mathrm{Cov}(\mathbf{I}_{\mathrm{sp},k}, \mathbf{I}_{\mathrm{sp},l})=\frac{\lambda t}{\vartheta}$ for $k=l$, $\mathrm{Cov}(\mathbf{I}_{\mathrm{sp},k}, \mathbf{I}_{\mathrm{sp},l})= -\frac{\lambda t}{2\vartheta}$ for $|k-l|=1$ and $\mathrm{Cov}(\mathbf{I}_{\mathrm{sp},k}, \mathbf{I}_{\mathrm{sp},l})=0$ otherwise. 
\item Note that if $\beta=d=1$ and the spatial points are sufficiently far apart, it is possible that all the light cones from \Cref{figure:partial_overlap} at the observation points $x_0, \dots, x_{n+1}$ do not overlap, see also \citet{shevchenkoQuadraticVariationsFractionalwhite2023}. This is essentially an i.i.d. setting, which is another interesting consequence of the finite propagation speed of the wave equation. 
\end{enumerate}
\end{remark}
The covariance between increments \eqref{eq:spatial_increments} is a two-dimensional second-order increment of the covariance function from Proposition~\ref{proposition:covariance_representation}:
\begin{equation*}
\mathrm{Cov}( \mathbf{I}_{\mathrm{sp},k}, \mathbf{I}_{\mathrm{sp},l}) =\frac{t \lambda^{2-\beta}}{2(2\pi)^{d}\vartheta}\int_{\mathbb{R}^{d}} \mathcal{I}^{(2)}[\Phi_{\mathrm{sp}, \rho \cdot \fv}](k,l)  \left(1-\mathrm{sinc}(2t\sqrt{\vartheta}\lambda^{-1}|\fv|)\right)|\fv|^{\beta-d-2}\mathrm{d}\fv,
\end{equation*}
with the function
\begin{equation*}
    \Phi_{\mathrm{sp}, \xi}:\mathbb{R}^2\rightarrow\mathbb{R},\quad\Phi_{\mathrm{sp}, \xi}(x,y)\coloneqq \cos\left( {(x-y)} \xi \right),\quad x,y,\xi\in\mathbb{R},
\end{equation*} 
and the second-order discrete increment in two variables given by
\begin{equation}
\label{eq:second_order_incremental_operator_d_2}
\begin{aligned}
\mathcal{I}^{(2)}[h](q, z) &\coloneqq 4h(q, z) + h(q+1,z+1) + h(q+1, z-1) + h(q-1,z+1) + h(q-1, z-1) \\
& \qquad- 2\left(h(q,z+1) + h(q, z-1) + h(q-1,z)+h(q+1,z)\right)
\end{aligned}
\end{equation}
for a function $h:\mathbb{R}^2\rightarrow\mathbb{R}$ and $q,z\in\mathbb{Z}$. Using the fact that increments of trigonometric functions often admit closed form expressions, cf. \Cref{section:general_results_on_increments}, we obtain the following representation of the covariance between two different second-order spatial increments $\mathbf{I}_{\mathrm{sp},k}$ and $\mathbf{I}_{\mathrm{sp},l}$.
\begin{lemma}\label[lemma]{lemma:covariance_second_order_spatial_increments}
The covariance between two second-order spatial increments is given by 
\begin{equation*}
\mathrm{Cov}( \mathbf{I}_{\mathrm{sp},k}, \mathbf{I}_{\mathrm{sp},l})=\frac{8t\lambda^{2-\beta}}{(2\pi)^{d}\vartheta}\int_{\mathbb{R}^{d}}  \cos((k-l) {\rho \cdot \fv}) g_{\mathrm{sp},\lambda}(\fv)\mathrm{d}\fv\\
\end{equation*}
with
\begin{equation}
\label{eq:g_sp_lambda}
\gsplambda(\fv)\coloneqq (1-\mathrm{sinc}(2t\sqrt{\vartheta}\lambda^{-1}|\fv|))\sin^{4}(\rho \cdot \fv/2)|\fv|^{\beta-d-2}.
\end{equation}
For $\lambda=0$, we extend the definition naturally by $\gspzero(\fv) \coloneqq \sin^{4}(\rho \cdot \fv/2)|\fv|^{\beta-d-2}$. 
\end{lemma}
\begin{remark}
One of the consequences of the geometry of the stochastic wave equation is that explicit representations for the covariance of increments such as \Cref{lemma:covariance_second_order_spatial_increments} are available. While they are in principle also viable in the parabolic setting, cf. \citet[Proposition 2.1]{hildebrandtParameterEstimationSPDEs2019}, the resulting closed-form representation is not effective in the analysis, and increments are immediately approximated by derivatives up to an additional error which needs to be accounted for, see \citet[Proposition 3.5]{hildebrandtParameterEstimationSPDEs2019} and \citet{bibingerVolatilityEstimationStochastic2019}.
\end{remark}
\Cref{lemma:covariance_second_order_spatial_increments} shows that even in general dimensions where $\beta \in (0, 2 \land d)$ the covariance matrix associated with the second-order increments 
is still a Toeplitz matrix as described in \Cref{remark:intuition_second_order_increments}. Furthermore, we notice that the distance of the spatial points $\lambda$ and the unknown $\vartheta$ enter the covariance as a scalar factor and through the $\mathrm{sinc}$ part of the integral. Thus, as $\lambda \rightarrow 0$, the covariance should scale as $\lambda^{2-\beta}$ since the $\mathrm{sinc}$-term vanishes asymptotically, allowing us to characterise the expectation and variance of the spatial variation $\mathbf{V}_{\mathrm{sp}}$.
\begin{proposition}\label[proposition]{proposition:expectation_spatial_variation_t_fixed}
\label{eq:expectation_convergence_speed_result_spatial_fixed_t}
Assuming that $\lambda \rightarrow 0$ as $n \rightarrow \infty$, the expectation of the second-order spatial variation defined through \eqref{eq:second_order_spatial_variation} satisfies
\begin{equation*}
\frac{\lambda^{\beta-2}}{n}\E[\mathbf{V}_{\mathrm{sp}}] = \frac{t}{\vartheta}\spconstexpt+\mathcal{O}(\lambda^{2})
\end{equation*}
for the constant $\spconstexpt>0$ given in \Cref{tab:asymptotic_constants}.
\end{proposition}
\begin{remark}
\begin{enumerate}
\item[]
\item If we start the process \eqref{eq:SPDE} within the deterministic initial conditions $u(0, \cdot)=u_0$ and $\partial_t u(0,x)=v_0$ the solution of the stochastic wave equation involves the additional deterministic term $u_\mathrm{det}(t,x)\coloneqq(\partial_t G_t * u_0) (x) + (G_t * v_0)(x)$. We denote by $\mathbf{V}_{\mathrm{sp}, \mathrm{det}}$ the second-order spatial variation of $u_{\mathrm{det}}$ similarly to \eqref{eq:second_order_spatial_variation}, where $u$ is replaced by $u_{\mathrm{det}}$ in \eqref{eq:spatial_increments}. 
Let $\gamma_0$ and $\gamma_1$ abbreviate the Hölder regularity of $u_0$ and $v_0$, respectively. While the operator cosine (i.e. $\partial_t G_t * u_0$) preserves the regularity of the initial condition, the operator sine (i.e. $G_t * v_0$) provides one additional order of regularity. Let $\mathbf V_{\mathrm{sp}, \mathrm{det}}$ be the second-order variation of $u_{\mathrm{det}}$. If $u_{\mathrm{det}}$ is sufficiently regular, i.e. $\gamma=\mathrm{min}(\gamma_0, \gamma_1+1) > 1-\beta/2$, we obtain $n^{-1}\lambda^{\beta-2}\mathbb{E}[\mathbf{V}_{\mathrm{sp}, \mathrm{det}}]=O(\lambda^{2\gamma-(2-\beta)})$ and our statistical results remain true also for non-zero initial conditions.
\item Suppose instead of \eqref{eq:SPDE}, we would consider the non-linear stochastic wave equation $\partial_{tt}^{2}u(t,x)=\vartheta\Delta u(t,x)+F(u(t,x))+\dot{W}_{\beta}(t,x)$ with a Lipschitz non-linearity $F$. Then, with $u_{\mathrm{lin}}$ being the solution to \eqref{eq:SPDE} the random solution to the non-linear equation is given by
\begin{equation*}
u(t,x)= u_{\mathrm{lin}}(t,x)+ u_{\mathrm{rem}}(t,x), \quad u_{\mathrm{rem}}(t,x)\coloneqq\int_{0}^{t}\int_{\mathbb{R}^{d}}G_{t-s}(x-y)F(u(s,y))\mathrm{d}s \mathrm{d}y.
\end{equation*}
Thus, the spatial variation $\mathbf{V}_{\mathrm{sp}}$ may be decomposed into three parts, $\mathbf{V}_{\mathrm{sp},\mathrm{lin}}$ and  $\mathbf{V}_{\mathrm{sp},\mathrm{rem}}$, i.e. the second-order variations of $u_{\mathrm{lin}}$ and $u_{\mathrm{rem}}$, respectively, and a lower order cross-term. \citet[Theorem 7.6]{conusNonLinearStochasticWave2008} states that the solution to the non-linear stochastic wave equation is $\alpha$-Hölder continuous of order $\alpha < 1-\beta/2$. By the Lipschitz-continuity of $F$ the same also holds for $F(u(s,\cdot))$. Moreover, the Fourier transform of the Green's function in \eqref{eq:fourier_transform_greens_function} is essentially a Fourier multiplier of negative order one and damps high frequencies through the factor $|\fv|^{-1}$. Thus, the derivative along direction $\rho$ of the operator sine function induced by $G$ is still $\alpha$-Hölder continuous, implying a bound of the form $\mathbb{E}[|\mathbf{I}_{\mathrm{sp},k}[u_{\mathrm{rem}}]|^{2}] \lesssim \lambda^{2 +2\alpha} \mathbb{E}[\Vert \partial_\rho u_{\mathrm{rem}}(t, \cdot) \Vert^{2}_{C^{0,\alpha}}]<\infty$. Consequently, $n^{-1}\lambda^{\beta-2}\mathbb{E}[\mathbf{V}_{\mathrm{sp},\mathrm{rem}}] = O(\lambda^{2\alpha + \beta})$. A more precise analysis might be possible by representing $u_{\mathrm{rem}}$ directly using the Fourier transform of $G$ and incorporating the second-order increments as an additional Fourier multiplier as in \Cref{lemma:covariance_second_order_spatial_increments}. 

Overall, this shows that under a relatively mild assumption on the non-linearity $F$, the linear part of the solution appears to dominate. This is in line with an analogous behaviour for the non-linear stochastic heat equation, cf. \citet{HildebrandtTrabs2023}. A full investigation of the non-linear setting might thus be an interesting avenue for future research.
\end{enumerate} 
\end{remark}
Now that we have understood the asymptotic behaviour of the expectation of the statistic $\mathbf{V}_{\mathrm{sp}}$, we would also like to investigate its variance. 
\begin{proposition}\label[proposition]{proposition:asymptotic_scaling_for_the_variance}
Suppose that $\lambda \rightarrow 0$. Then, the variance of the second-order spatial variation satisfies 
\begin{equation*}
\frac{\lambda^{2\beta-4}}{n}\V(\mathbf{V}_{\mathrm{sp}}) \rightarrow \frac{t^{2}}{\vartheta^{2}}\spconstvart,
\end{equation*}
for the constant $\spconstvart>0$ given in \Cref{tab:asymptotic_constants}.
\end{proposition}
The proof of \Cref{proposition:asymptotic_scaling_for_the_variance} is rooted in classical Fourier analysis. Using Wicks' (Isserlis') theorem \cite[Theorem 1.28]{Janson_1997} and \Cref{lemma:covariance_second_order_spatial_increments}, we obtain the asymptotic
\begin{equation}
\label{eq:variance_asymptotic_1}
\frac{\lambda^{2\beta-4}}{n}\V(\mathbf{V}_{\mathrm{sp}}) \asymp \frac{1}{n}\sum_{k,l=1}^{n}\left(\int_{\mathbb{R}^{d}}\cos((k-l)\rho \cdot \fv)g_{\mathrm{sp},\lambda}(\fv)\mathrm{d}\fv\right)^{2}.
\end{equation}
The sum \eqref{eq:variance_asymptotic_1} admits a representation through the Fejér kernel, visualised in \Cref{figure:fejer},
\begin{equation}
\label{eq:introducing_the_fejér_kernel}
\mathfrak{F}_{\mathrm{sp}}(x) \coloneqq 
\frac{\sin^2\!\left(\tfrac{n x}{2}\right)}{n\,\sin^2\!\left(\tfrac{x}{2}\right)}
= \underbrace{\frac{1}{n}\sum_{k,l=1}^{n} \cos\!\big((k-l)x\big)}_{\text{cosine representation}}
= \underbrace{\sum_{|k|\leq n-1} w_{\mathrm{sp},k} e^{\ii kx}}_{\text{Fourier representation}}=\underbrace{\sum_{|k|\leq n-1} w_{\mathrm{sp},k} \cos(kx)}_{\text{simplification}},
\quad x \in \mathbb{R},
\end{equation}
with the symmetric weights $w_{\mathrm{sp}, k}\coloneqq1-|k|/n$, which itself has several useful representations summarised in \citet[Proposition 3.1.7]{grafakosClassicalFourierAnalysis2014}.
\begin{figure}[tb]
\centering
\begin{tikzpicture}[scale=0.8]
  \begin{axis}[
    width=12cm, height=8cm,
    domain=-12:12,
    samples=502,              
    trig format=rad,
    xmin=-12, xmax=12,
    ymin=0,   ymax=7,
    major grid style={line width=0.4pt},
    axis lines=left,           
    xtick={-10,-5,0,5,10},  
    ytick={0,2,4,6,8,10},
    xlabel={$x$},
    ylabel={},
    title={},
    legend style={draw=none, fill=none},
    legend pos=north east,
    cycle list name=fejercolors,
  ]

    \addplot {fejer( 2, x)}; \addlegendentry{$n=2$}
    \addplot {fejer( 4, x)}; \addlegendentry{$n=4$}
    \addplot {fejer( 6, x)}; \addlegendentry{$n=6$}
  \end{axis}
\end{tikzpicture}
\caption{Visualisation of the Fejér kernel $\mathfrak{F}_{\mathrm{sp}}$ defined through \eqref{eq:introducing_the_fejér_kernel} for different values of $n \in \mathbb{N}$.}
\label{figure:fejer}
\end{figure}
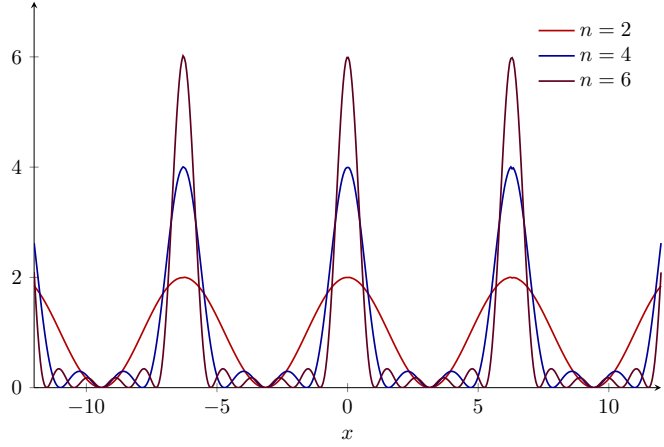
We may first expand the square and represent \eqref{eq:variance_asymptotic_1} as an integral in two different variables. Together with the sum-to-product formula for the cosine, and the cosine representation from \eqref{eq:introducing_the_fejér_kernel}, we have
\begin{equation}
\label{eq:variance_asymptotic_2}
\frac{\lambda^{2\beta-4}}{n}\V(\mathbf{V}_{\mathrm{sp}}) \asymp \int_{\mathbb{R}^{d}}\int_{\mathbb{R}^{d}} \big(\mathfrak{F}_{\mathrm{sp}}(\rho \cdot (\fv_1+\fv_2))+ \mathfrak{F}_{\mathrm{sp}}({\rho \cdot (\fv_1-\fv_2)})\big)\gsplambda(\fv_1)\gsplambda(\fv_2) \mathrm{d}\fv_1 \mathrm{d}\fv_2. 
\end{equation}
If we now express the Fejér kernel through the Fourier representation from \eqref{eq:introducing_the_fejér_kernel}, both Fejér kernels can be merged. In fact, since $\gsplambda$ is an even function, the double integral even factorises and we can fully represent the variance through the Fourier coefficients of the function $\gsplambda$:
\begin{equation}
\label{eq:partial_fourier_sum}
\frac{\lambda^{2\beta-4}}{n}\V(\mathbf{V}_{\mathrm{sp}}) \asymp \sum_{|k| \leq n-1}w_{\mathrm{sp},k} \left(\mathcal{F}_c(g_{\mathrm{sp},\lambda})(k \rho)\right)^{2} \rightarrow \sum_{k \in \mathbb{Z}} \left(\mathcal{F}_c(\gspzero)(k\rho)\right)^2 < \infty. 
\end{equation}

\begin{remark}
\begin{enumerate}
\item[]
\item Fejér and Dirichlet kernels are fundamentally related to the idea of reproducing functions through their Fourier coefficients, for instance, using partial Fourier sums or Cesàro sums such as \eqref{eq:partial_fourier_sum}. For an in-depth treatment of the subject and how to incorporate results on the decay of the Fourier transform, we refer to \citet[Chapter 3.2, Chapter 3.3.2]{grafakosClassicalFourierAnalysis2014}. While we follow that approach throughout this paper, the Fejér kernel can also be considered to converge to a counting measure (see \Cref{figure:fejer}):
\begin{equation*}
\mathfrak{F}_{\mathrm{sp}} \rightarrow 2\pi \mu_{2\pi \mathbb{Z}}=2\pi\sum_{k \in \mathbb{Z}}\delta_{2\pi k}, \quad n \rightarrow \infty,
\end{equation*}
and the limit of \eqref{eq:variance_asymptotic_2} can be represented as an integral over suitable hyper-planes. Both perspectives are linked through a higher-dimensional version of the Poisson summation principle. 
\item Even in the case $\beta=d=1$, a result identical to \Cref{proposition:asymptotic_scaling_for_the_variance} for first-order increments cannot be obtained as the number of non-trivial off-diagonal entries of the associated covariance matrix scales quadratically in $n$. In \eqref{eq:spatial_increments}, we have instead considered the incremental order two. In view of \Cref{lemma:covariance_second_order_spatial_increments}, the incremental order enters the covariance representation of the increments through the power of the sine term and the constant. The power of the sine term is crucial as it smoothens out the singularities of the Riesz colouring term $|\cdot|^{\beta-d-2}$ around the origin. We even require the derivative of $\gspzero$ to be integrable for the series in \eqref{eq:partial_fourier_sum} to converge, which is not satisfied given first-order increments.
\end{enumerate}
\end{remark}

\Cref{proposition:asymptotic_scaling_for_the_variance} already characterises the asymptotic behaviour of second-order spatial variation precisely. This is a crucial step in the verification of the Lyapunov condition for a central limit theorem: 
\begin{equation}
\label{eq:CLT_condition_maximum_row_norm_version}
\frac{\mathbf{V}_{\mathrm{sp}} - \E[\mathbf{V}_{\mathrm{sp}}]}{\sqrt{\V(\mathbf{V}_{\mathrm{sp}})}} \xrightarrow{d} N(0,1) \quad\text{ if } \quad \frac{\left(\max_{k=1, \dots,n}\sum_{l=1}^{n} |\mathrm{Cov}(\mathbf{I}_{\mathrm{sp},k}, \mathbf{I}_{\mathrm{sp},l})|\right)^{2}}{\V(\mathbf{V}_{\mathrm{sp}})} \rightarrow 0, \quad n \rightarrow \infty,
\end{equation}
see for instance \citet[Proposition 3.1]{hildebrandtParameterEstimationSPDEs2019}.
\begin{remark}
Let $\Sigma_n=(\mathrm{Cov}(\mathbf{I}_{\mathrm{sp},k}, \mathbf{I}_{\mathrm{sp},l}))_{k,l=1, \dots,n}$ be the covariance matrix associated with the second-order increments. Then, the vector of increments $(\mathbf{I}_{\mathrm{sp},1}, \dots, \mathbf{I}_{\mathrm{sp},n})$ is distributed according to $N(0, \Sigma_n)$ and its norm is the variation $\mathbf{V}_{\mathrm{sp}}$. In particular,  $\mathbf{V}_{\mathrm{sp}}$ can be rewritten in terms of the eigenvalues of the matrix $\Sigma_n$ and Lyapunov’s condition is satisfied provided that $\mathrm{tr}(\Sigma_n^{4})/\mathrm{tr}(\Sigma_n^{2})^{2}\rightarrow 0$. This condition is equivalent to $\Vert  \Sigma_n\Vert_{2}^{2}/\V(\mathbf{V}_{\mathrm{sp}}) \rightarrow 0$ as $n \rightarrow \infty$ and leads to \eqref{eq:CLT_condition_maximum_row_norm_version} if the spectral norm is bounded by the maximum absolute row norm. 
\end{remark}
By exploiting the Toeplitz structure of the covariance matrix in \Cref{lemma:covariance_second_order_spatial_increments}, the numerator in the condition \eqref{eq:CLT_condition_maximum_row_norm_version} can be controlled, and we obtain the central limit theorem. 
\begin{theorem}\label[theorem]{theorem:clt_spatial_fixed_t}
Suppose that $\lambda \rightarrow 0$ as $n\rightarrow\infty$, the second-order spatial variation satisfies 
\begin{equation*}
\sqrt{n}\left(\frac{\lambda^{\beta-2}}{n}\mathbf{V}_{\mathrm{sp}}- \E\left[\frac{\lambda^{\beta-2}}{n}\mathbf{V}_{\mathrm{sp}}\right]\right) \xrightarrow{d} N\left(0, \frac{t^{2}}{\vartheta^{2}}\spconstvart\right), \quad n \rightarrow \infty.
\end{equation*}
\end{theorem}
\begin{remark}
If we assume that $\lambda=1$ s.t. we leave the high-frequency regime, the expression for the expectation from \Cref{proposition:expectation_spatial_variation_t_fixed} is not easily inverted in $\vartheta$, making the construction of a method-of-moments type estimator difficult. However, the arguments leading to the representation of the variance in \eqref{eq:partial_fourier_sum} remain fully valid, only that the limiting constant will depend non-linearly on $\vartheta$. Thus, a central limit theorem of the form $n^{-1/2}(\mathbf{V}_{\mathrm{sp}} - \E[\mathbf{V}_{\mathrm{sp}}]) \xrightarrow{d} N(0, \tilde{C}_{\mathrm{sp}}(\beta,d,\vartheta))$ holds even when $\lambda$ is constant. 
\end{remark}
It remains to show that when the distance $\lambda$ between the spatial points $x_0, \dots, x_{n+1}$ shrinks sufficiently fast, we can replace $\E[n^{-1}\lambda^{\beta-2}\mathbf{V}_{\mathrm{sp}}]$ by its limit in \Cref{theorem:clt_spatial_fixed_t}. Since we have seen in \Cref{proposition:expectation_spatial_variation_t_fixed} that the bias is of order $\lambda$, it becomes clear that $\sqrt{n}\lambda^{2} \rightarrow 0$ is a sufficient condition which is summarised within the following corollary.
\begin{corollary}\label[corollary]{corollary:spatial_clt_fixed_time_expectation_swap}
Suppose that $\lambda= o(n^{-1/4})$. Then, the second-order spatial variation satisfies
\begin{equation*}
\sqrt{n}\left(\frac{\lambda^{\beta-2}}{n}\mathbf{V}_{\mathrm{sp}}- \frac{t}{\vartheta}\spconstexpt\right) \xrightarrow{d} N\left(0, \frac{t^{2}}{\vartheta^{2}}\spconstvart\right), \quad n \rightarrow \infty.
\end{equation*}
\end{corollary}
Based on \Cref{corollary:spatial_clt_fixed_time_expectation_swap}, we can now construct a method-of-moments type estimator for the unknown wave speed $\vartheta>0$:
\begin{equation}
\label{eq:estimator_spatial_variation_fixed_t}
\hat{\vartheta}_{\mathrm{sp},n} \coloneqq \frac{t \spconstexpt}{n^{-1}\lambda^{\beta-2}\mathbf{V}_{\mathrm{sp}}}. 
\end{equation}
The delta method applied to the estimator \eqref{eq:estimator_spatial_variation_fixed_t} yields the following central limit theorem. 
\begin{corollary}\label[corollary]{corollary:clt_estimator_spatial_fixed_time}
Suppose that $\lambda=o(n^{-1/4})$. Then, the method-of-moments estimator \eqref{eq:estimator_spatial_variation_fixed_t} admits the central limit theorem
\begin{equation*}
\sqrt{n}(\hat{\vartheta}_{\mathrm{sp},n}-\vartheta) \xrightarrow{d} N\left(0, \vartheta^{2}\frac{\spconstvart}{(\spconstexpt)^{2}} \right). 
\end{equation*}
\end{corollary}
\begin{remark}
\label[remark]{remark:confidence_intervals}
Based on the central limit theorem \Cref{corollary:clt_estimator_spatial_fixed_time}, we can now construct asymptotic confidence intervals with confidence level $1-\alpha$ by $[\hat{\vartheta}_{\mathrm{sp},n} -(\hat{\vartheta}_{\mathrm{sp},n}/\sqrt{n})c q_{1-\alpha/2},\hat{\vartheta}_{\mathrm{sp},n} + (\hat{\vartheta}_{\mathrm{sp},n}/\sqrt{n})c q_{1-\alpha/2}]$ where $c=\spconstvart^{1/2}/\spconstexpt$ and $q_{1-\alpha/2}$
is the $1-\alpha/2$ quantile of the standard normal distribution. 
\end{remark}
\section{Temporal variation}
\label{section:temporal_variations}
In contrast to \Cref{section:spatial_variation}, we start by fixing a spatial point $x \in \mathbb{R}^{d}$ and assume that we have access to observations of the process $(u(t,x), t \in \mathbb{R})$ at equidistant time points $t_i$ for $i=0, \dots, m+1$, i.e. we observe $u(t_i, x)$ for $t_i=\delta i$ and $i=0, \dots, m+1$. \Cref{figure:visualisation_of_temporal_observations} provides a visualisation of the observation scheme. As before, $\delta=\delta_m$ may depend on the number of time points.

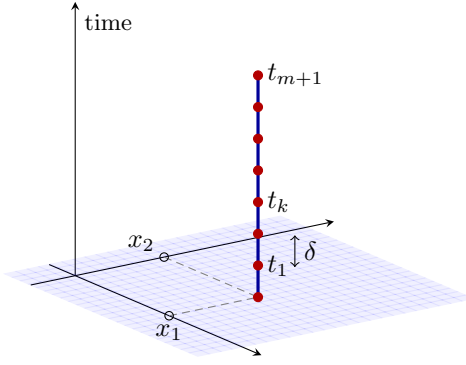
\begin{figure}[tb]
\centering
\begin{tikzpicture}

  \def\m{6}                          
  \pgfmathtruncatemacro{\mplusone}{\m+1}
  \def\d{0.7}                        
  \def\k{3}                          
  \def\xs{1.8} \def\ys{1.0}          
  \pgfmathsetmacro{\tmax}{\d*(\m+1)} 
  \pgfmathsetmacro{\xoff}{0.28}      
  \pgfmathsetmacro{\yoff}{0.25}

  \begin{axis}[
    view={55}{22},
    axis lines=center,                
    xlabel={}, ylabel={}, zlabel={time},
    xmin=-0.5, xmax=3.2,
    ymin=-0.5, ymax=2.6,
    zmin=0,   zmax=\tmax+0.6,
    ticks=none,
    tick label style={font=\small},
    label style={font=\small},
    enlargelimits=upper,
    clip=false
  ]

    \addplot3[
      surf, shader=flat, draw=none, opacity=0.06,
      domain=-0.5:3.2, y domain=-0.5:2.6
    ] {0};

    \addplot3[very thick, blue!60!black]
      coordinates {(\xs,\ys,0) (\xs,\ys,\tmax)};

    \pgfplotsinvokeforeach{0,...,\mplusone}{
      \pgfmathsetmacro{\ti}{\d*#1}
      \addplot3[only marks, mark=*, mark size=1.7pt, red!70!black]
        coordinates {(\xs,\ys,\ti)};
    }

    \node[anchor=west] at (axis cs:\xs,\ys,\d) {$t_1$};
    \node[anchor=west] at (axis cs:\xs,\ys,\k*\d) {$t_k$};
    \node[anchor=west] at (axis cs:\xs,\ys,\tmax) {$t_{m+1}$};

    \coordinate (P)  at (axis cs:\xs,\ys,0);   
    \coordinate (Px) at (axis cs:\xs,0,0);     
    \coordinate (Py) at (axis cs:0,\ys,0);     

    \draw[densely dashed, gray] (Px) -- (P);   
    \draw[densely dashed, gray] (Py) -- (P);   

    \addplot3[only marks, mark=o, mark size=1.6pt]
      coordinates {(\xs,0,0) (0,\ys,0)};
    \node[below] at (axis cs:\xs,0,0) {$x_1$};
    \node[left]  at (axis cs:0,\ys,0.3) {$x_2$};

    \addplot3[<->] coordinates {(\xs+\xoff,\ys+\yoff,\d) (\xs+\xoff,\ys+\yoff,2*\d)};
    \node[anchor=west] at (axis cs:\xs+\xoff,\ys+\yoff,1.5*\d) {$\delta$};

  \end{axis}
\end{tikzpicture}
\caption{Visualisation of the time discrete observations $t_0, t_1, \dots, t_{m+1}$ with distance $t_{i+1}-t_{i}=\delta$ in $d=2$ at the fixed spatial point $x=(x_1, x_2)$.}
\label{figure:visualisation_of_temporal_observations}
\end{figure}

Similarly, to \eqref{eq:second_order_spatial_variation}, we follow a method-of-moment type approach and introduce the second-order temporal variation
\begin{equation}
\label{eq:second_order_temporal_variation}
\mathbf{V}_{\mathrm{te}} \coloneqq \sum_{i=1}^{m} \mathbf{I}_{\mathrm{te},i}^{2},
\end{equation}
with the second-order temporal increments
\begin{equation}
\label{eq:temporal_increments}
 \mathbf{I}_{\mathrm{te},i} \coloneqq u(t_{i+1},x)+u(t_{i-1},x)-2u(t_{i}, x), \quad i=1, \dots, m.
\end{equation}
\begin{remark}
\label[remark]{remark:MLE}
Consider the case of space-time white noise $\beta=d=1$. Then, in view of \Cref{figure:full_inclusion}, the process $(u(t,x), t \geq 0)$ has the time-scaled covariance structure of a Brownian motion
\begin{equation*}
\E[u(t_i, x)u(t_j, x)]=\frac{1}{4 \sqrt{\vartheta}} (t_i \land t_j)^{2}=\frac{\delta^{2}}{4\sqrt{\vartheta}} (i^{2} \land j^{2}), \quad i,j=1,\dots,m,
\end{equation*}
In particular, the vector $U_{m}=(u(t_1,x), \dots, u(t_m, x))$ is distributed according to a multivariate Gaussian $U_m \sim N(0, {\delta^{2}}/(4\sqrt{\vartheta})A_m )$ with $A_m = [i^{2}\land j^{2}]_{i,j=1,\dots,m}$. Thus, the log-likelihood $l$ associated with the vector $U_{m}$ satisfies $\partial_{\vartheta}l(\vartheta| U_m)
={m}/(4\vartheta)-{Q_m}/(\delta^{2}\sqrt{\vartheta})$ with $Q_m=U_m^{\top}A_m^{-1}U_m$, leading to the maximum likelihood estimator $\hat{\vartheta}_m^{\mathrm{MLE}}= \left({\delta^{2}m}/(4Q_m)\right)^{2}$. We observe that computing the MLE involves computing the inverse of the matrix $A_m$. According to \citet[8.11]{seberMatrixHandbookStatisticians2007}, the inverse is a tridiagonal matrix $A_m^{-1}=\mathrm{trid}[b^{(m-1)}, a^{(m)}, b^{(m-1)}]$ with $a^{(m)} \in \mathbb{R}^{m}$ and $b^{(m-1)} \in \mathbb{R}^{m-1}$ given by $a^{(m)}_{j} = {1}/(2j-1)+{1}/(2j+1)$ for $j=1, \dots, m$ and $b^{(m-1)}_j = -{1}/(2j+1)$ for $j=1, \dots,m-1$, respectively. Thus, $Q_m$ is given by
\begin{equation*}
Q_m = u(t_1, x)^{2} + \sum_{j=1}^{m-1} \frac{(u(t_{j+1}, x)-u(t_j,x))^{2}}{2j+1} + \frac{u(t_m,x)^{2}}{2m+1},
\end{equation*}
and the MLE involves a weighted quadratic variation. Notice that
\begin{equation*}
\V(u(t_{j+1}, x)-u(t_{j}, x))=\frac{1}{4\sqrt{\vartheta}}(t_{j+1}^{2} - t_{j}^{2})= \frac{\delta^{2}}{4\sqrt{\vartheta}}[(j+1)^{2}-j^{2}]=\frac{\delta^{2}}{4\sqrt{\vartheta}}[2j+1]. 
\end{equation*}
Thus, increments with higher variance are weighted less then increments with smaller variance by the maximum likelihood estimator. 
\end{remark}
Our next aim is to obtain a representation of the covariance between two different second-order temporal increments $\mathbf{I}_{\mathrm{te},i}$ and $\mathbf{I}_{\mathrm{te},j}$. While for temporal variations it is possible to work with precise evaluations of the explicit covariance formula \eqref{eq:concrete_equation_temporal_covariance_structure} directly, we unified our approach throughout this paper and follow the Fourier methodology outlined in \Cref{section:spatial_variation}. 

The covariance between two different second-order temporal variations is characterised by the two-dimensional second-order increment of the function $(i,j) \mapsto \Phi_{\mathrm{te}, |\fv|}(\delta i,\delta j)$ from \eqref{eq:covariance_sine_integral}, which is computed explicitly in the following result.
\begin{lemma}\label[lemma]{lemma:covariance_second_order_temporal_increments}
The covariance between two different second-order temporal increments is given by
\begin{equation*}
\begin{aligned}
&\mathrm{Cov}(\mathbf{I}_{\mathrm{te},i}, \mathbf{I}_{\mathrm{te},j})
= \delta^{3-\beta}\left((i \land j) \left(\frac{8}{ (2\pi)^{d}\vartheta^{\beta/2}} \right)\int_{\mathbb{R}^{d}}  \cos((i-j) |\fv|) \gtemp(\fv)\mathrm{d}\fv + R(i,j)\right)
\end{aligned}
\end{equation*}
with
\begin{equation}
\label{eq:g_temp_d}
\gtemp(\fv)\coloneqq\sin^4(|\fv|/2)|\fv|^{\beta-d-2}, 
\end{equation}
and an explicit remainder term given by \eqref{eq:time_remainders}, which is uniformly bounded by a constant independent of $i$ and $j$. 
\end{lemma}
\begin{remark}
Compared with \Cref{lemma:covariance_second_order_spatial_increments} the covariance structure of the time increments derived in \Cref{lemma:covariance_second_order_temporal_increments} involves additional explicit remainder terms, depending on the distance between the indices $i$ and $j$. These terms emerge, since the temporal process $(u(t,x), t \geq 0)$ is non-stationary and its covariance involves the mapping $\Phi_{\mathrm{te, \xi}}$ defined through \eqref{eq:covariance_sine_integral}, which resembles the covariance of the standard harmonic oscillator. By carefully taking account the different cases emerging from the absolute value and the minimum in the representation \eqref{eq:covariance_sine_integral}, increments of the mapping $\Phi_{\mathrm{te, \xi}}$ could still be computed explicitly due to its highly trigonometric nature. 
\end{remark}
Up the remainder terms both representations for the covariance in \Cref{lemma:covariance_second_order_spatial_increments} and \Cref{lemma:covariance_second_order_temporal_increments} are structurally quite similar. They differ in their scaling in $\lambda$ and $\delta$, respectively. In fact, the entire temporal covariance structure is completely independent of spatial point $x \in \mathbb{R}^{d}$ chosen, while in \eqref{lemma:covariance_second_order_spatial_increments} $\gsplambda$ is still time-dependent.  Moreover, in contrast to the spatial covariance structure, $\vartheta$ is a scale parameter of the Gaussian process $(u(t,x), t \geq 0)$ such that the covariance structure depends on $\vartheta$ only through a scalar factor. This fact is particularly important for the expected value of the temporal variation $\mathbf{V}_{\mathrm{te}}$ considered in the following result, because the expression for the expectation can always be reversed in $\vartheta$.
\begin{proposition}\label[proposition]{proposition:expectation_temp_fixed_space}
The expected value of the second-order temporal variation $\mathbf{V}_{\mathrm{te}}$ admits the asymptotic behaviour:
\begin{equation*}
\frac{\delta^{\beta-3}}{m^{2}}\E[\mathbf{V}_{\mathrm{te}}] =  \frac{\tempconstexpt}{\vartheta^{\beta/2}}+\mathcal{O}(m^{-1})
\end{equation*}
for the constant $\tempconstexpt>0$ given in \Cref{tab:asymptotic_constants}. 
\end{proposition}

For the spatial variations, the variance of $\mathbf{V}_{\mathrm{sp}}$ leads naturally to the cosine representation of the Fejér kernel \eqref{eq:introducing_the_fejér_kernel}. Assuming that we can ignore the remainder terms from \Cref{lemma:covariance_second_order_temporal_increments}, the same line of argumentation leads to the representation
\begin{equation*}
\frac{\delta^{2\beta-6}}{m^{3}}\V(\mathbf{V}_{\mathrm{te}}) \asymp \frac{1}{m^{3}}\sum_{i,j=1}^{m}\left((i \land j)\int_{\mathbb{R}^{d}}  \cos((i-j) |\fv|) \gtemp(\fv)\mathrm{d}\fv\right)^{2}.
\end{equation*}
Due to the additional factor $(i \land j)$, an expansion of the square as two integrals leads to a mapping 
\begin{equation}
\label{eq:fejer_time_kernel}
\mathfrak{F}_{\mathrm{te}}(x)\coloneqq \frac{1}{m^{3}}\sum_{i,j=1}^{m} (i \land j)^{2} \cos((i-j)x)={\sum_{|j| \leq m-1} {w}_{\mathrm{te},j} \cos(jx)}, \quad 
w_{\mathrm{te},j}\coloneqq \frac{1}{m^{3}}\sum_{i=1}^{m-|j|} i^{2},
\end{equation}
taking the role of the Fejér kernel, see \Cref{lemma:fourier_representation_time_fejer_kernel} for a proof. In particular, we obtain the representation
\begin{equation*}
\frac{\delta^{2\beta-6}}{m^{3}}\V(\mathbf{V}_{\mathrm{te}}) \asymp \int_{\mathbb{R}^{d}}\int_{\mathbb{R}^{d}}\left(\mathfrak{F}_{\mathrm{te}}(|\fv_1|+|\fv_2|)+\mathfrak{F}_{\mathrm{te}}(|\fv_1|-|\fv_2|)\right)\gtemp(\fv_1)\gtemp(\fv_2)\mathrm{d}\fv_1 \mathrm{d}\fv_2,
\end{equation*}
which allows us to find a suitable representation as a partial Fourier sum
\begin{equation}
\label{eq:variance_heuristic_temp}
\frac{\delta^{2\beta-6}}{m^{3}}\V(\mathbf{V}_{\mathrm{te}}) \asymp  \sum_{|j| \leq m-1} w_{\mathrm{te},j} \left(\mathcal{F}^{+}_{c}(\gtemp)(j)\right)^{2}.
\end{equation}
\begin{remark}
In principle, since all functions in \eqref{proposition:covariance_representation} are radial if a fixed spatial point is considered, one can also pass to polar coordinates immediately, which for instance leads to the representation \eqref{eq:concrete_equation_temporal_covariance_structure}. In that case the term $\mathcal{F}_c^{+}(g_{\mathrm{te}})(j)$ in \eqref{eq:variance_heuristic_temp} would correspond exactly to the cosine transform of the mapping $r\mapsto\sin^{4}(r/2)r^{\beta-3}$ on $[0,\infty)$.
\end{remark}
\begin{proposition}\label[proposition]{proposition:variance_temporal_second_order_variation_fixed_space}
The variance of the second-order temporal variation has the asymptotic behaviour
\begin{equation*}
\frac{\delta^{2\beta-6}}{m^{3}}\V(\mathbf{V}_{\mathrm{te}}) \rightarrow \frac{\tempconstvart}{\vartheta^{\beta}}, \quad m \rightarrow \infty,
\end{equation*}
for the constant $\tempconstvart>0$ given in \Cref{tab:asymptotic_constants}.
\end{proposition}
\begin{remark}
\label[remark]{remark:suboptimal_variance_first_order_increments}
First-order increments appear naturally within the MLE in the case of space-time white noise, see \Cref{remark:MLE}. However, the parameter $\beta$ influences the overall regularity of the process $(u(t,x), t \geq 0)$. As a result, a statement of the of \Cref{proposition:variance_temporal_second_order_variation_fixed_space} is in general false when considering first-order increments. Given the first-order variation 
$$\mathbf{V}_{\mathrm{te}}^{(1)}\coloneqq\sum_{i=1}^{m}(u(t_{i+1},x)-u(t_{i-1},x))^{2},$$ one can verify, for instance, using \eqref{eq:concrete_equation_temporal_covariance_structure} that
\begin{equation*}
\V(\mathbf{V}_{\mathrm{te}}^{(1)}) \asymp \begin{cases}
m^{3}\delta^{6-2\beta}, & \beta\in (1/2,2),\\
m^{3}\log(m)\delta^{5}, &\beta = 1/2,\\
m^{4-2\beta}\delta^{6-2\beta}, & \beta \in (0, 1/2).
\end{cases}
\end{equation*}
Smaller values of $\beta$ correspond to more spatial colouring, and the process $(u(t,x), t \geq 0)$ will have more regularity. At some critical point, i.e.\ at $\beta=1/2$, a phase transition occurs and the first-order (quadratic) variation degenerates. This critical point exactly corresponds to the well-known phase transition encountered for the Hurst coefficient at $H=3/4$, see for instance \citet{COHEN2006187} and the references therein. 
\end{remark}
Analogously to \eqref{eq:CLT_condition_maximum_row_norm_version}, we prove a central limit theorem using the Lyapunov condition 
\begin{equation}
\label{eq:CLT_condition_maximum_row_norm_version_temporal}
 \frac{\left(\max_{i=1, \dots,m}\sum_{j=1}^{m} |\mathrm{Cov}(\mathbf{I}_{\mathrm{te},i}, \mathbf{I}_{\mathrm{te},j})|\right)^{2}}{\V(\mathbf{V}_{\mathrm{te}})} \rightarrow 0, \quad m \rightarrow \infty.
\end{equation}
\begin{theorem}\label[theorem]{theorem:clt_temporal_space_fixed}
The second-order temporal variation \eqref{eq:second_order_temporal_variation} satisfies
\begin{equation*}
\sqrt{m}\left(\frac{\delta^{\beta-3}}{m^{2}}\mathbf{V}_{\mathrm{te}}  -  \frac{\tempconstexpt}{\vartheta^{\beta/2}}\right) \xrightarrow{d} N\left(0, \frac{\tempconstvart}{\vartheta^{\beta}}\right), \quad m \rightarrow \infty.
\end{equation*}
\end{theorem}
\begin{remark}
\label[remark]{remark:bias_in_the_purely_temporal_situation}
Note that in contrast to \Cref{corollary:spatial_clt_fixed_time_expectation_swap}, we do not need to impose any additional convergence behaviour for $\delta$. Indeed, the remainder term in \Cref{proposition:expectation_temp_fixed_space} already converges to zero even when multiplied by the additional factor $\sqrt{m}$. In fact, \Cref{theorem:clt_temporal_space_fixed} covers both high-frequency and low-frequency regimes by default. This is natural in view of \Cref{remark:geometric_structure} \ref{remark:scale_parameter} as the canonical statistic in the model is independent of $\delta$ after whitening. In essence, the estimation of $\vartheta$ is uneffected by the sampling regime and only depends on the sample size $m$.
\end{remark}
Based on \Cref{proposition:expectation_temp_fixed_space}, we define the method-of-moments estimator
\begin{equation}
\label{eq:mom_estimator_temporal_space_fixed}
\hat{\vartheta}_{\mathrm{te},m} \coloneqq \left(\frac{\tempconstexpt}{m^{-2}\delta^{\beta-3}\mathbf{V}_{\mathrm{te}}}\right)^{2/\beta}
\end{equation}
which satisfies the following CLT by applying the delta method in \Cref{theorem:clt_temporal_space_fixed}.
\begin{corollary}\label[corollary]{corollary:clt_method_of_moments_estimator_temporal_space_fixed}
The method-of-moments estimator \eqref{eq:mom_estimator_temporal_space_fixed} satisfies
\begin{equation*}
\sqrt{m}(\hat{\vartheta}_{\mathrm{te},m}- \vartheta) \xrightarrow{d} N\left(0, \vartheta^{2} \frac{4 \tempconstvart}{\beta^{2}\tempconstexpt^{2}}\right), \quad m \rightarrow\infty.
\end{equation*}
\end{corollary}
While the asymptotic variance in \Cref{theorem:clt_temporal_space_fixed} and \Cref{corollary:clt_method_of_moments_estimator_temporal_space_fixed} remains independent of the spatial point $x$, it does depend on the spatial regularity $\beta$, the unknown wave speed $\vartheta$, the constant $\tempconstexpt$ and $\tempconstvart$. As described in \Cref{remark:confidence_intervals}, confidence intervals can be constructed based on \Cref{corollary:clt_method_of_moments_estimator_temporal_space_fixed}, via the plug-in principle.

\begin{remark}
\begin{enumerate}
\item[]

\item 
In the case of temporal observations at a fixed spatial observation point $x \in \mathbb{R}^{d}$, it is not difficult to see that the convergence rate $\sqrt{m}$ is optimal. Since $\vartheta$ is a scale parameter of the Gaussian process $(u(t,x), t \geq 0)$ the covariance matrix of the observation vector $(u(t_0,x), \dots, u(t_{m+1},x))$ depends on the unknown parameter only through the scalar $\vartheta^{-\beta/2}$, see \Cref{remark:geometric_structure} \ref{remark:scale_parameter}. For two different probability measures, $P$ and $Q$, we define the (squared) Hellinger distance $H^{2}(P,Q)={\frac {1}{2}}\int ({\sqrt {P(dx)}}-{\sqrt {Q(dx)}})^{2}$. The Hellinger distance between the laws $P_{\vartheta_0,m}$ and $P_{\vartheta_1,m}$ of the time-discrete observations $(u(t_i,x), i =0,\dots,m+1)$ for two different wave speeds $\vartheta_0$ and $\vartheta_1$ is given by
\begin{equation*}
    H^{2}(P_{\vartheta_0,m}, P_{\vartheta_1,m})= 1- \left(1-\frac{(\vartheta_0^{\beta/2}- \vartheta_1^{\beta/2})^{2}}{(\vartheta_0^{\beta/2}+\vartheta_1^{\beta/2})^{2}}\right)^{(m+2)/4}.
\end{equation*}
In particular, the optimal rate of convergence for estimating the wave speed $\vartheta$ is $m^{-1/2}$.
As suggested by \Cref{remark:suboptimal_variance_first_order_increments}, the method-of-moments estimator based on first-order increments does not, in general, achieve the optimal rate of convergence whenever $\beta \in (0, 1/2)$. For instance, in the boundary case $\beta=1/2$ it attains the convergence rate $m^{-1/2}\mathrm{log}(m)^{2}$ and a log-factor is lost. In contrast, whenever $\beta \in (1/2, 2 \land d)$, first-order increments would be sufficient and the optimal rate of convergence can also be obtained.
\item While \Cref{theorem:clt_temporal_space_fixed} shows that the method-of-moments estimator $\hat{\vartheta}_{\mathrm{te},m}$ achieves the optimal rate of convergence, it does not achieve minimal asymptotic variance. In general settings with a complicated covariance structure, it may not be easy to disentangle the non-linear dependence of the covariance matrix from the unknown parameter $\vartheta$ and to determine the MLE as in \Cref{remark:MLE}. Which is why the method-of-moments approach forms a robust alternative. However, in this particular case, since $\vartheta$ is a scale parameter as discussed in \Cref{remark:geometric_structure} \ref{remark:scale_parameter}, the MLE is computable and given by 
\begin{equation*}
\hat{\vartheta}_{\mathrm{MLE},m} \coloneqq \left(\frac{1}{\frac{1}{m+2}|\Sigma_m^{-1/2}U_m|^{2}}\right)^{2/\beta},
\end{equation*}
where $U_m=(u(t_0, x), \dots, u(t_{m+1},x))$ and $\Sigma_m$ is the covariance matrix of the vector $U_m$ in the special case $\vartheta=1$. While it may, in general, not be clear how to obtain a precise analytic representation for $\Sigma_m^{-1/2}$, it can at least be approximated numerically. For a more detailed discussion of efficiency considerations in the one-dimensional case, we refer to \citet{markussenLikelihoodInferenceDiscretely2003}.
\end{enumerate}
\end{remark}

\section{Space-time variation}
\label{section:space_time_variation}
In this section, we combine our insights from \Cref{section:spatial_variation} and \Cref{section:temporal_variations}, and consider observations of the form $u(t_i, x_k)$ with $t_i=\delta i$ for $i=0, \dots, m+1$ and $x_k=\lambda k\rho$ for $k=0, \dots, n+1$ and $\rho \in \mathbb{S}^{d-1}$. Thus, we are considering equidistant observations of the process \eqref{eq:SPDE} on a space-time grid $\{(t_i, x_k) \colon i=0, \dots, m+1, k=0, \dots, n+1\}$, visualised in \Cref{figure:visualisation_space_time_grid}.

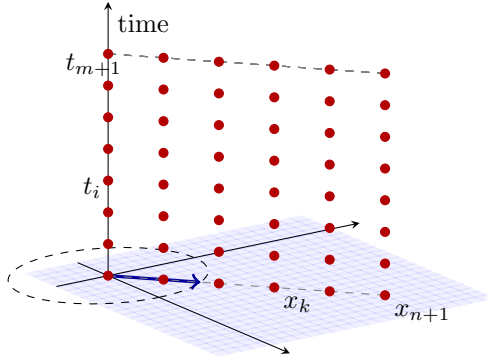
\begin{figure}
    \centering
    \begin{tikzpicture}

  \def\n{5}                          
  \def\m{6}                          
  \pgfmathtruncatemacro{\mplusone}{\m+1}
  \def\lambda{0.6}                   
  \def\d{0.7}                        
  \def\rx{0.8}
  \def\ry{0.6}

  \pgfmathsetmacro{\L}{\n*\lambda}             
  \pgfmathsetmacro{\tmax}{\d*(\m+1)}           
  \pgfmathsetmacro{\xoff}{0.28}                
  \pgfmathsetmacro{\yoff}{0.22}
  \def\klabel{3}                                
   
  \begin{axis}[
    view={55}{22},
    axis lines=center,
    xlabel={}, ylabel={}, zlabel={time},
    xmin=-0.6, xmax=3.2,
    ymin=-0.6, ymax=2.6,
    zmin=0,   zmax=\tmax+0.6,
    ticks=none,
    enlargelimits=upper,
    clip=false
  ]

    \addplot3[
      surf, shader=flat, draw=none, opacity=0.06,
      domain=-0.6:3.2, y domain=-0.6:2.6
    ] {0};

  \draw[dashed] (0,0) circle (1);

    \draw[->, very thick, blue!60!black]
      (axis cs:0,0,0) -- (axis cs:\rx,\ry,0);

    \draw[dashed, gray]
      (axis cs:0,0,0) -- (axis cs:\L*\rx,\L*\ry,0);

    \pgfplotsinvokeforeach{0,...,\mplusone}{
      \pgfmathsetmacro{\ti}{\d*#1}

      \draw[dashed, gray]
        (axis cs:0,0,\ti) -- (axis cs:\L*\rx,\L*\ry,\ti);

      \foreach \kk in {0,...,\n}{
        \addplot3[only marks, mark=*, mark size=1.7pt, red!70!black]
          coordinates {(\kk*\lambda*\rx, \kk*\lambda*\ry, \ti)};
      }
    }

    \node[anchor=west] at (axis cs:0,-0.4,\klabel*\d) {$t_i$};
    \node[anchor=west] at (axis cs:0,-0.6,\tmax) {$t_{m+1}$};

    \node[below right] at (axis cs:\klabel*\lambda*\rx,\klabel*\lambda*\ry,0) {$x_k$};
    \node[below right] at (axis cs:\n*\lambda*\rx,\n*\lambda*\ry,0) {$x_{n+1}$};




  \end{axis}
\end{tikzpicture}
    \caption{Visualisation of the space-time grid $(t_i, x_k)$ in $d=2$ for $i=0, \dots, m+1$ and $k=0, \dots, n+1$. The blue vector symbolising the direction of the spatial measurements is $\rho \in \mathbb{S}^1$.}
    \label{figure:visualisation_space_time_grid}
\end{figure}

We extend the definition of the spatial and temporal increments \eqref{eq:spatial_increments} and \eqref{eq:temporal_increments}, by including the fixed variable in the notation. More specifically, we write $\mathbf{I}_{\mathrm{sp},k}(t)$ for the second-order spatial increments at time $t$ and $\mathbf{I}_{\mathrm{te},i}(x)$ for the second-order temporal increment at point $x \in \mathbb{R}^{d}$. We can now take increments of the increments in the remaining variable and define space-time box-increments of the form
\begin{equation}
\label{eq:space_time_increments}
\begin{aligned}
\mathbf{I}_{\mathrm{sp},\mathrm{te},i,k}\coloneqq\mathbf{I}_{\mathrm{sp},k}(t_{i+1})+\mathbf{I}_{\mathrm{sp},k}(t_{i-1})-2\mathbf{I}_{\mathrm{sp},k}(t_i)=\mathbf{I}_{\mathrm{te},i}(x_{k+1})+\mathbf{I}_{\mathrm{te},i}(x_{k-1})-2\mathbf{I}_{\mathrm{te},i}(x_k).
\end{aligned}
\end{equation}
The associated space-time variation is then given by
\begin{equation}
\label{eq:space_time_variation}
\mathbf{V}_{\mathrm{sp}, \mathrm{te}}\coloneqq \sum_{i=1}^{m}\sum_{k=1}^{n} \mathbf{I}_{\mathrm{sp},\mathrm{te},i,k}^{2}.
\end{equation}
We would like to derive the covariance between two box increments $\mathbf{I}_{\mathrm{sp},\mathrm{te},i,k}$ and $\mathbf{I}_{\mathrm{sp},\mathrm{te},j,l}$. The strategy from \Cref{lemma:covariance_second_order_temporal_increments} and \Cref{lemma:covariance_second_order_spatial_increments} is still applicable, noting that we obtain a second-order increment of the four-dimensional covariance function. Notice that the covariance function from \Cref{proposition:covariance_representation} is particularly simple in the sense that there are no actual cross-dependencies between the temporal and the spatial factors. Thus, the second-order increment in four dimensions can be decomposed into the product of two two-dimensional second-order increments evaluated in the variables $k,l$ and $i,j$, respectively. Based on this idea, the following result provides a full characterisation of the covariance structure of the box increments. 
\begin{lemma}\label[lemma]{lemma:covariance_space_time_increments}
The space-time increments \eqref{eq:space_time_increments} admit the covariance structure
\begin{equation*}
\begin{aligned}
&\mathrm{Cov}(\mathbf{I}_{\mathrm{sp},\mathrm{te},i,k}, \mathbf{I}_{\mathrm{sp},\mathrm{te},j,l}) \\
&= \delta \lambda^{2-\beta} \Bigg(\frac{128 (i \land j) }{(2\pi)^{d}\vartheta} \int_{\mathbb{R}^{d}} \cos\left((k-l) {\rho \cdot \fv}\right)\sin^{4}(\ratio\sqrt{\vartheta} |\fv|/2)\cos((i-j)\ratio\sqrt{\vartheta} |\fv|)\gspzero(\fv)\mathrm{d}\fv\\
&\qquad+R_{\mathrm{sp}}(i,j,k,l)\Bigg)\\
&= \delta^{3-\beta} \Bigg(\frac{128 (i \land j) }{(2\pi)^{d}\vartheta^{\beta/2}} \int_{\mathbb{R}^{d}} \cos\left((\ratio\sqrt{\vartheta})^{-1}(k-l) {\rho \cdot \fv}\right)\sin^{4}((\ratio\sqrt{\vartheta})^{-1} \rho \cdot \fv/2)\cos((i-j) |\fv|)\gtemp(\fv)\mathrm{d}\fv\\
&\qquad+R_{\mathrm{te}}(i,j,k,l)\Bigg),
\end{aligned}
\end{equation*}
with $$\ratio = \delta/\lambda$$ and remainder terms, defined in \eqref{eq:spatial_remainder_box_inc} and \eqref{eq:temp_remainder_box_inc}, which remain uniformly bounded up to a constant independently of $i,j,k,l$ and $\alpha$, in the regimes $\alpha \rightarrow \infty$ and $\alpha \rightarrow 0$, respectively. 
\end{lemma}

When $\alpha\rightarrow\infty$, the spatial resolution is much finer than the temporal one and the asymptotics of the space-time variation will be dominated by the spatial behaviour. Vice versa, $\alpha\rightarrow0$ corresponds to a sampling regime driven by the temporal frequency and $\mathbf{V}_{\mathrm{sp}, \mathrm{te}}$ essentially acts as $\mathbf{V}_{\mathrm{te}}$ averaged through space.
Up to the scaling factors, the entire covariance structure revealed in \Cref{lemma:covariance_space_time_increments} depends on the ratio between the spatial and temporal resolution levels. Neglecting the remainder terms, the variance of a second-order space-time increment satisfies
\begin{equation}
\label{eq:appearence_RL_lemma}
\begin{aligned}
\delta^{-1}\lambda^{\beta-2}\frac{\V(\mathbf{I}_{\mathrm{sp},\mathrm{te},i,k})}{i} &\asymp \frac{1}{\vartheta}\left(\int_{\mathbb{R}^{d}} \sin^{4}(\ratio \sqrt{\vartheta}|\fv|/2)\gspzero(\fv)\mathrm{d}\fv\right),\\
\delta^{\beta-3} \frac{\V(\mathbf{I}_{\mathrm{sp},\mathrm{te},i,k})}{i} &\asymp \frac{1}{\vartheta^{\beta/2}}\left(\int_{\mathbb{R}^{d}}\sin^{4}((\ratio \sqrt{\vartheta})^{-1}\rho \cdot \fv/2) \gtemp(\fv)\mathrm{d}\fv\right). 
\end{aligned}
\end{equation}
The $\mathrm{sin}^{4}$ is a bounded function, while $\gspzero$ and $\gtemp$ are integrable. The generalised Riemann-Lebesgue Lemma, see for instance \citet{kahaneGeneralizationsRiemannLebesgueCantorLebesgue1980}, has proven to be an important tool in statistics for the stochastic wave equation, see \citet{ziebellNonparametricEstimationStochastic2024}. In fact, if we assume that $\ratio \rightarrow \infty$ or $\ratio \rightarrow 0$, the oscillatory part of the corresponding integrals in \eqref{eq:appearence_RL_lemma} will converge to the mean of the function $\sin^{4}$, and we observe 
\begin{equation}
\label{eq:Riemann_lebesgue_illustration}
\begin{aligned}
\frac{1}{\vartheta}\left(\int_{\mathbb{R}^{d}} \sin^{4}(\ratio \sqrt{\vartheta}|\fv|/2)\gspzero(\fv)\mathrm{d}\fv\right) &\rightarrow \frac{1}{\vartheta} \left(\frac{3}{8}\int_{\mathbb{R}^{d}}\gspzero(\fv)\mathrm{d}\fv\right), \quad \ratio \rightarrow \infty,\\
 \frac{1}{\vartheta^{\beta/2}}\left(\int_{\mathbb{R}^{d}}\sin^{4}((\ratio \sqrt{\vartheta})^{-1}\rho \cdot \fv/2) \gtemp(\fv)\mathrm{d}\fv\right) & \rightarrow \frac{1}{\vartheta^{\beta/2}}\left(\frac{3}{8} \int_{\mathbb{R}^{d}}\gtemp(\fv)\mathrm{d}\fv\right), \quad \ratio \rightarrow 0. 
\end{aligned}
\end{equation}
Notice that in both cases the limits in \eqref{eq:Riemann_lebesgue_illustration} do not depend on the factor $\sqrt{\vartheta}$ anymore, allowing us to prove the following useful asymptotic characterisation of the expectation of \eqref{eq:space_time_variation}.
\begin{proposition}\label[proposition]{proposition:box_expectation}
The expectation of the space-time variation \eqref{eq:space_time_variation} satisfies
\begin{equation*}
\begin{aligned}
\frac{\delta^{-1}\lambda^{\beta-2}}{nm^{2}}\E[\mathbf{V}_{\mathrm{sp}, \mathrm{te}}]&=\frac{\boxconstexpsp}{\vartheta}+\mathcal{O}\left(m^{-1}+\ratio^{-2}\right), \quad \alpha \rightarrow \infty,\\
\frac{\delta^{\beta-3}}{nm^{2}}\E[\mathbf{V}_{\mathrm{sp}, \mathrm{te}}]&=\frac{\boxconstexptemp}{\vartheta^{\beta/2}}+\mathcal{O}\left(m^{-1}+\ratio^2\right), \quad \alpha \rightarrow 0,
\end{aligned}
\end{equation*}
for constants $\boxconstexpsp,\boxconstexptemp>0$ given in \Cref{tab:asymptotic_constants}.
\end{proposition}
The remainder in \Cref{proposition:box_expectation} is composed of two parts. The contribution $m^{-1}$ is inherited from the temporal variations similar to \Cref{proposition:expectation_temp_fixed_space} as the rescaled variance of an increment increases with growing time, cf. \eqref{eq:appearence_RL_lemma}. Depending on the asymptotic regime, the second part is an error incurred due to the approximation in \eqref{eq:Riemann_lebesgue_illustration} via the Riemann-Lebesgue lemma.

As in the previous section, we will eventually need to make sure that the bias converges to zero asymptotically. The expected rate in the central limit theorem will be exactly $\sqrt{nm}$ leading to the following assumption. 
\begin{assumption}
\label[assumption]{assumption:bias_convergence}
We assume that
\begin{equation}
\label{eq:bias_assumption_n_frac_m}
\frac{n}{m} \rightarrow 0
\end{equation}
and 
\begin{equation}
\label{eq:bias_assumption_ratio}
\begin{cases}
\ratio/(nm)^{1/4} \rightarrow \infty,  & \ratio \rightarrow \infty,\\
\ratio(nm)^{1/4} \rightarrow 0, &\ratio \rightarrow 0.
\end{cases}
\end{equation}
\end{assumption}
We have already seen in \Cref{lemma:covariance_space_time_increments} that the spatial and temporal components of the covariance structure of the increments factorise. As a consequence, when studying the variance of the space-time variation $\mathbf{V}_{\mathrm{sp}, \mathrm{te}}$, we can combine techniques from \Cref{section:spatial_variation} and \Cref{section:temporal_variations}. Both the argument for $\ratio \rightarrow \infty$ and $\ratio \rightarrow 0$ are structurally very similar up to some technical differences in handling the term $\rho \cdot \fv$ compared to $|\fv|$. Here we describe the arguments in the spatial asymptotic $\ratio \rightarrow \infty$. Both the Fejér kernels $\mathfrak{F}_{\mathrm{sp}}$ and $\mathfrak{F}_{\mathrm{te}}$ emerge in the rescaled variance $\delta^{-2}\lambda^{2\beta-4}n^{-1}m^{-3}\V(\mathbf{V}_{\mathrm{sp}, \mathrm{te}})$ asymptotically behaving as
\begin{equation}
\label{eq:double_fejer_representation_variance_box_increments}
\begin{aligned}
\int_{\mathbb{R}^{d}}\int_{\mathbb{R}^{d}}  \mathfrak{F}_{\mathrm{sp}}(\fv_1, \fv_2)\mathfrak{F}_{\mathrm{te}}(\fv_1, \fv_2)  \sin^{4}(\ratio \sqrt{\vartheta}|\fv_1|/2)\gspzero(\fv_1)\sin^{4}(\ratio \sqrt{\vartheta}|\fv_2|/2)\gspzero(\fv_2)\mathrm{d}\fv_1\mathrm{d}\fv_2\\
\end{aligned}
\end{equation}
and
\begin{equation*}
\begin{aligned}
\mathfrak{F}_{\mathrm{sp}}(\fv_1, \fv_2) &\coloneqq \left(\mathfrak{F}_{\mathrm{sp}}(\rho \cdot (\fv_1 + \fv_2)) + \mathfrak{F}_{\mathrm{sp}}(\rho \cdot (\fv_1 - \fv_2)) \right),\\
\mathfrak{F}_{\mathrm{te}}(\fv_1, \fv_2) &\coloneqq \left(\mathfrak{F}_{\mathrm{te}}(\ratio \sqrt{\vartheta} (|\fv_1|+|\fv_2|)) +\mathfrak{F}_{\mathrm{te}}(\ratio \sqrt{\vartheta} (|\fv_1|-|\fv_2|)) \right).
\end{aligned}
\end{equation*}
The entire expression \eqref{eq:double_fejer_representation_variance_box_increments} factorises and can be reduced to the following double sum of Fourier coefficients
\begin{equation}
\label{eq:variance_asymptotic_heuristic_boxincrements}
\frac{\delta^{-2}\lambda^{2\beta-4}}{nm^{3}}\V(\mathbf{V}_{\mathrm{sp}, \mathrm{te}}) \asymp \sum_{\substack{|j|\leq m-1\\|k|\leq n-1}}w_{\mathrm{sp},k}w_{\mathrm{te},j}\left(\int_{\mathbb{R}^{d}} \cos(k \rho \cdot \fv)\cos(j \ratio \sqrt{\vartheta} |\fv|) \sin^{4}(\ratio \sqrt{\vartheta}|\fv|/2)\gspzero(\fv)\mathrm{d}\fv\right)^{2}. 
\end{equation}
As we can express 
\begin{equation*}
    \cos(j\ratio\sqrt{\vartheta}|\fv|)\sin^4(\ratio\sqrt{\vartheta}|\fv|/2)=\sum_{z\in\{0,\pm1,\pm2\}}c_z \cos((j+z)\ratio\sqrt{\vartheta}|\fv|)
\end{equation*}
for some factors $c_z$ to be defined later, \eqref{eq:variance_asymptotic_heuristic_boxincrements} consists of integrals containing high frequent cosine waves which vanish by the Riemann-Lebesgue Lemma, whenever the decay in spatial direction, i.e.\ the variable $k$, is sufficiently fast. In that case, solely terms where $\cos((j+z)\ratio\sqrt{\vartheta}|\fv|)=1$, i.e.\ $|j|\leq 2$ and $z=-j$, constitute to the asymptotics of \eqref{eq:variance_asymptotic_heuristic_boxincrements}. To ensure a sufficient decay in spatial and temporal direction, respectively, we impose the following additional requirement.

\begin{assumption}
\label[assumption]{assumption:helper}
    If $d=1$, we assume 
    \begin{equation}
    \label{eq:bias_assumption_ratio_helper}
    \begin{cases}
    \frac{\sqrt{\vartheta}\ratio}{2}\geq n,  & \ratio \rightarrow \infty,\\
    \frac{1}{2\sqrt{\vartheta}\ratio}\geq m, &\ratio \rightarrow 0.
    \end{cases}
    \end{equation}
\end{assumption}

The proof of the following result gives a rigorous meaning to \eqref{eq:variance_asymptotic_heuristic_boxincrements} and derives the precise asymptotic for the variance of the space-time variation \eqref{eq:space_time_variation}. 
 \begin{proposition}\label[proposition]{proposition:variance_space_time_variation} Grant \Cref{assumption:bias_convergence} and \Cref{assumption:helper}. Then, we have
 \begin{equation*}
 \begin{aligned}
 \frac{\delta^{-2}\lambda^{2\beta-4}}{nm^{3}}\V(\mathbf{V}_{\mathrm{sp}, \mathrm{te}}) &\rightarrow \frac{\boxconstvarsp}{\vartheta^{2}}, \quad \ratio \rightarrow \infty,\\
 \frac{\delta^{2\beta-6}}{nm^{3}}\V(\mathbf{V}_{\mathrm{sp}, \mathrm{te}}) &\rightarrow \frac{\boxconstvartemp}{\vartheta^{\beta}}, \quad \ratio \rightarrow 0,
 \end{aligned}
 \end{equation*}
 for constants $\boxconstvarsp,\boxconstvartemp>0$ given in \Cref{tab:asymptotic_constants}.
 \end{proposition}

 \begin{remark}
\label[remark]{rmk:assumption_weakening}
Note that \Cref{assumption:bias_convergence} will be required eventually for the asymptotic analysis of the bias. Neither \Cref{assumption:bias_convergence} nor \Cref{assumption:helper} are required at all for \Cref{proposition:variance_space_time_variation} when $d \geq 3$ and can be weakened for $d=2$, which is revealed by the discussion following \Cref{lemma:C_k_j_f_a_j}. Nevertheless \Cref{assumption:bias_convergence} and \Cref{assumption:helper} simplify the proof considerably and were chosen to increase the readability of the manuscript. The dimensional dependence of the more precise arguments stems from the fact that in higher dimensions it is actually possible to separate the Fourier decay of the coefficients of \eqref{eq:variance_asymptotic_heuristic_boxincrements} along the directions in which $j$ and $k$ act. This is fundamentally impossible when $d=1$, and $d=2$ constitutes an edge case where the arguments may be applied, but only yield  the high-dimensional decay, partially.
\end{remark}
The required condition for a central limit theorem for the space-time variation \eqref{eq:space_time_variation} now involves both the spatial and the temporal components, amounting to
\begin{equation}
\label{eq:lyapunov_condition_box_increments}
\frac{\left( \max_{j=1,\dots,m}\max_{l=1,\dots,n} \sum_{k=1}^{n}\sum_{i=1}^{m}|\mathrm{Cov}(\mathbf{I}_{\mathrm{sp},\mathrm{te},i,k},\mathbf{I}_{\mathrm{sp},\mathrm{te},j,l})|\right)^{2}}{\V(\mathbf{V}_{\mathrm{sp}, \mathrm{te}})} \rightarrow 0.
\end{equation}

\begin{theorem}\label[theorem]{theorem:clt_box_increments}  Grant \Cref{assumption:bias_convergence} and \Cref{assumption:helper}. Then the space-time variation \eqref{eq:space_time_variation} satisfies
\begin{equation*}
\begin{aligned}
{\sqrt{nm}\left(\frac{\delta^{-1}\lambda^{\beta-2}}{nm^{2}}\mathbf{V}_{\mathrm{sp}, \mathrm{te}}-\frac{C_{\mathrm{box,sp,\E}}}{\vartheta}\right)} &\xrightarrow{d} N\left(0, \frac{\boxconstvarsp}{\vartheta^{2}}\right), \quad \ratio \rightarrow \infty,\\
{\sqrt{nm}\left(\frac{\delta^{\beta-3}}{nm^{2}}\mathbf{V}_{\mathrm{sp}, \mathrm{te}}-\frac{C_{\mathrm{box,temp,\E}}}{\vartheta^{\beta/2}}\right)} &\xrightarrow{d} N\left(0, \frac{\boxconstvartemp}{\vartheta^{\beta}}\right), \quad \ratio \rightarrow 0.
\end{aligned}
\end{equation*}
\end{theorem}
An appropriate rescaling of the variations is essential to achieve asymptotic convergence. In the case of space-time white noise (i.e. $\beta=d=1$) and choices $\delta=m^{-1}$ and $\lambda=n^{-1}$, the quadratic variations $\mathbf{V}_{\mathrm{sp}}$ and $\mathbf{V}_{\mathrm{te}}$, based solely on spatial and temporal increments, respectively, converge to their asymptotic limits without further normalisation; cf. \Cref{corollary:clt_estimator_spatial_fixed_time} and \Cref{theorem:clt_temporal_space_fixed}; see also \citet[Proposition 4.1 and Proposition 5.1]{assaadQuadraticVariationDrift2022a}. 

 When $\beta \in (0, 2 \land d)$, the local Hölder regularity of the process in both space and time is $1-\beta/2$. Thus, both the variance of a spatial increment in \Cref{lemma:covariance_second_order_spatial_increments} and of temporal increment \Cref{lemma:covariance_second_order_temporal_increments} have a scalar contribution of size $\lambda^{2-\beta}$ and $\delta^{2-\beta}$, respectively. However, the purely spatial process is stationary, while the temporal process is non-stationary and incurs an additional increase in variance of order $\delta i$. This increased variance, must be compensated for in the rescaling of the temporal variation compared to the spatial one. 

The regimes $\ratio \rightarrow \infty$ and $\ratio \rightarrow 0$ essentially state that we observe at a higher spatial or temporal observation frequency, respectively. As a consequence the rescaling of the space-time variation $\mathbf{V}_{\mathrm{sp},\mathrm{te}}$ in \Cref{theorem:clt_box_increments} needs to account for the dominating spatial or temporal resolution. In fact, the ratio between both normalisations
\begin{equation*}
\frac{\delta^{3-\beta}}{\delta^{-1}\lambda^{\beta-2}}=\left(\frac{\delta}{\lambda}\right)^{2-\beta}=\ratio^{2-\beta}
\end{equation*}
is scaled exactly in proportion to the ratio $\ratio$ determining the asymptotic regime. 
\begin{remark}
Since the renormalisation is rooted in the regularity of the solutions, the correct rescaling for the stochastic heat equation in \citet{hildebrandtParameterEstimationSPDEs2019} differs. While the stochastic wave equation is Hölder continuous of order $1-\beta/2$ in both space and time, the corresponding orders for the stochastic heat equation are $1-\beta/2$ in space and $1/2-\beta/4$ in time, see \citet{dalangHittingProbabilitiesSystems2013}, leading to different normalisations in \citet[Proposition 3.5 and Theorem 3.7]{hildebrandtParameterEstimationSPDEs2019} compared to the hyperbolic setting. Furthermore, in the parabolic setting, the corresponding sampling regime scales with $r=\lambda/\sqrt{\delta}=\alpha^{-1}\sqrt{\delta}$ compared to the situation of the stochastic wave equation, where both time and space contribute equally. 
\end{remark}

As a consequence, the dependence of the asymptotic variance on the wave speed $\vartheta$ in \Cref{proposition:variance_space_time_variation} is as in \Cref{section:spatial_variation,section:temporal_variations}. The resulting method of moments estimators are given by
\begin{equation*}
\hat{\vartheta}_{\mathrm{box},\mathrm{sp}, m,n} \coloneqq \frac{\boxconstexpsp}{n^{-1}m^{-2}\delta^{-1}\lambda^{\beta-2}\mathbf{V}_{\mathrm{sp}, \mathrm{te}}}, \quad \hat{\vartheta}_{\mathrm{box},\mathrm{te}, m,n} \coloneqq \left(\frac{\boxconstexptemp}{n^{-1}m^{-2}\delta^{\beta-3}\mathbf{V}_{\mathrm{sp}, \mathrm{te}}}\right)^{2/\beta},
\end{equation*}
and have the following asymptotic behaviour. 
\begin{corollary}
\label[corollary]{corollary:clt_estimators_box}
Grant \Cref{assumption:bias_convergence} and \Cref{assumption:helper}. Then, we have
\begin{equation*}
\begin{aligned}
&\sqrt{nm}(\hat{\vartheta}_{\mathrm{box},\mathrm{sp},m,n} - \vartheta) \xrightarrow{d} N\left(0, \vartheta^{2}\frac{\boxconstvarsp}{(\boxconstexpsp)^{2}}\right), \quad \ratio \rightarrow \infty,\\
&\sqrt{nm}(\hat{\vartheta}_{\mathrm{box},\mathrm{te},m,n} - \vartheta) \xrightarrow{d} N\left(0, \frac{4 \boxconstexptemp}{\beta^{2}\boxconstexptemp^{2}}\right), \quad \ratio \rightarrow 0.
\end{aligned}
\end{equation*}
\end{corollary}
While the scaling regimes are very different in the parabolic setting, the actual rate of convergence in \Cref{corollary:clt_estimators_box} mirrors exactly the results from \citet{hildebrandtParameterEstimationSPDEs2019}.

\section{Proofs and auxiliary results}
\label{section:proofs}

In the following four subsections we prove the results presented in the \Cref{section:spde_model,section:spatial_variation,section:temporal_variations,section:space_time_variation}, in numerical order. Our analysis of the variations in space and time is based on the covariance between second-order increments. As such, an in-depths understanding of the incremental structure is necessary, and both general and explicit results are collected in \Cref{section:general_results_on_increments}. Furthermore, the regularity of the functions $\gspzero$ and $\gtemp$ constituting to the covariance of increments is discussed in \Cref{sec: ax_results_g}.

The next table summarizes the asymptotic constants emerging in the expected values and variances of the second-order variations from \Cref{section:spatial_variation,section:temporal_variations,section:space_time_variation}.
\begin{table}[h!]
    \centering
    \scalebox{1.1}{
    \begin{tabular}{l|l|l}
    \multicolumn{1}{c|}{Variations in}&\multicolumn{1}{c|}{Expected value}&\multicolumn{1}{c}{Variance}\\ \hline
         Space&$\spconstexpt\coloneqq\frac{8}{(2\pi)^{d}}\lVert\gspzero\rVert_{L^1(\mathbb{R}^d)}$&$\spconstvart\coloneqq\frac{128}{(2\pi)^{2d}}\sum_{k \in \mathbb{Z}}\left(\mathcal{F}_c(\gspzero)(k\rho)\right)^{2}$\\
         Time&$\tempconstexpt\coloneqq\frac{4}{ (2\pi)^{d}}\lVert\gtemp\rVert_{L^1(\mathbb{R}^{d})}$&$\tempconstvart\coloneqq\frac{128}{3(2\pi)^{2d}}\sum_{j \in \mathbb{Z}} \left(\mathcal{F}_{c}^{+}(\gtemp)(j)\right)^{2}$\\
         Space-time, $\alpha\rightarrow\infty$&$\boxconstexpsp\coloneqq\frac{24}{(2\pi)^{d}}\lVert\gspzero\rVert_{L^1(\mathbb{R}^d)}$&$\boxconstvarsp\coloneqq\frac{2^{8}35}{3(2\pi)^{2d}}\sum_{k\in\mathbb{Z}}\left(\mathcal{F}_c(\gspzero)(k\rho)\right)^2$\\
         Space-time, $\alpha\rightarrow0$&$\boxconstexptemp\coloneqq\frac{24}{ (2\pi)^{d}}\lVert\gtemp\rVert_{L^1(\mathbb{R}^{d})}$&$\boxconstvartemp\coloneqq\frac{2^{8}35}{3(2\pi)^{2d}}\sum_{j\in\mathbb{Z}}\left(\mathcal{F}_c^+(\gtemp)(j)\right)^2$
    \end{tabular}}
    \caption{Asymptotic limits of expectation and variance of the rescaled variations in space, time and space-time defined in \Cref{section:spatial_variation,section:temporal_variations,section:space_time_variation} with sampling frequency ratio $\alpha=\delta/\lambda$. The functions $\gspzero$ and $\gtemp$ are introduced in Lemma~\ref{lemma:covariance_second_order_spatial_increments} and Lemma~\ref{lemma:covariance_second_order_temporal_increments}, respectively.}
    \label{tab:asymptotic_constants}
\end{table}

\subsection{Proofs for \Cref{section:spde_model}}

\begin{proof}[Proof of \Cref{proposition:covariance_representation}]
Itô's isometry from \citet{dalangExtendingMartingaleMeasure1999} gives the representation
\begin{equation*}
\begin{aligned}
\E[u(t,x)u(s,y)]=\frac{1}{(2\pi)^{d}}\int_{0}^{t \land s} \int_{\mathbb{R}^{d}} \mathcal{F}(G_{t-r}(x-\cdot))(\fv)\overline{\mathcal{F}(G_{s-r}(y-\cdot))}(\fv) |\fv|^{\beta-d}\mathrm{d}\fv \mathrm{d}r.
\end{aligned}
\end{equation*}
Using the shift property of the Fourier transform yields
\begin{equation*}
\E[u(t,x)u(s,y)]=\frac{1}{(2\pi)^{d}}\int_{0}^{t \land s} \int_{\mathbb{R}^{d}} e^{\ii(x-y) \cdot \fv }\overline{\mathcal{F}(G_{t-r}(\cdot))}(\fv){\mathcal{F}(G_{s-r}(\cdot))}(\fv) |\fv|^{\beta-d}\mathrm{d}\fv \mathrm{d}r.
\end{equation*}
By taking the real part, we obtain
\begin{equation*}
\begin{aligned}
\E[u(t,x)u(s,y)]&= \frac{1}{(2\pi)^{d}\vartheta}\int_{0}^{t \land s} \int_{\mathbb{R}^{d}} \cos((x-y)  \cdot \fv) \sin((t-r)|\fv|\sqrt{\vartheta}) \sin((s-r)|\fv|\sqrt{\vartheta}) |\fv|^{\beta-d-2}\mathrm{d}\fv \mathrm{d}r\\
&= \frac{1}{(2\pi)^{d}\vartheta^{\beta/2}}\int_{0}^{t \land s} \int_{\mathbb{R}^{d}} \cos\left(\frac{(x-y)}{\sqrt{\vartheta}}  \cdot \fv \right) \sin((t-r)|\fv|) \sin((s-r)|\fv|) |\fv|^{\beta-d-2}\mathrm{d}\fv \mathrm{d}r,
\end{aligned}
\end{equation*}
which is exactly the desired representation. The second characterisation follows immediately with Fubini's theorem. It remains to calculate the integral.
Using the product-to-sum formula $\sin(x)\sin(y)=\frac{1}{2}(\cos(x-y)-\cos(x+y))$, we obtain
\begin{equation*}
\begin{aligned}
\Phi_{\mathrm{te}, \xi}(t,s)&=\int_{0}^{t \land s}\sin((t-r)\xi)\sin((s-r)\xi)\mathrm{d}r\\
&= \frac{1}{2}\int_{0}^{t \land s}\cos((t-s)\xi)\mathrm{d}r - \frac{1}{2} \int_{0}^{t \land s}\cos((t+s-2r)\xi)\mathrm{d}r\\
&= \frac{1}{2} (t \land s)\cos((t-s)\xi) + \frac{1}{2}\int_{|t-s|\xi}^{(t+s)\xi} \cos(r) \frac{\mathrm{d}r}{-2\xi}\\
&=\frac{\sin(|t-s|\xi)-\sin((t+s)\xi)}{4\xi}+\frac{1}{2}(t \land s)\cos((t-s)\xi). 
\end{aligned}
\end{equation*}
Another application of the sum-to-product identity $\sin(x)-\sin(y)=2\cos((x+y)/2)\sin((x-y)/2)$ yields together with $\sin(-x)=-\sin(x)$:
\begin{equation*}
\begin{aligned}
\frac{\sin(|t-s|\xi)-\sin((t+s)\xi)}{4\xi}&=\frac{\cos((|t-s|+(t+s))\xi/2)\sin((|t-s|-(t+s))\xi/2)}{2 \xi}\\
&=-\frac{\cos((t \lor s)\xi)\sin((t \land s)\xi)}{2 \xi} 
\end{aligned}
\end{equation*}
and thus
\begin{equation*}
\begin{aligned}
\Phi_{\mathrm{te}, \xi}(t,s)
&=\frac{1}{2}\left((t \land s)\cos((t-s)\xi)-\frac{\cos((t \lor s)\xi)\sin((t \land s)\xi)}{\xi}\right).
\end{aligned}\qedhere
\end{equation*}
\end{proof}

\begin{proposition}	\label[proposition]{result:temporal_covariance_structure}
	There exists a constant $C_{\beta,d}>0$ such that for any fixed point in space $x \in \mathbb{R}^{d}$, we have
	\begin{equation}
		\label{eq:covariance_function_fixed_point_in_space}
		\E[u(t,x)u(s,x)]= C_{\beta,d}\vartheta^{-\beta/2} \left( \frac{1}{2(3-\beta)}\left((t+s)^{3-\beta}-|t-s|^{3-\beta}\right)-( t \land s)|t-s|^{2-\beta}\right).
	\end{equation}
\end{proposition}
\begin{proof}
	Consider the representation for the covariance function from \Cref{proposition:covariance_representation}. For a single fixed spatial point, the integrand within the integral representation of the covariance is radial. In polar coordinates we obtain
	\begin{equation*}
		\begin{aligned}
			 & \E[u(t,x)u(s,x)]                                                                                                                                                                         \\
			 & = \frac{1}{(2\pi)^{d} \vartheta}\int_{0}^{t \land s}\int_{\mathbb{R}^{d}} {\sin((t-r)\sqrt{\vartheta}|\fv|)\sin((s-r)\sqrt{\vartheta}|\fv|)}|\fv|^{\beta-d-2} \mathrm{d}\fv \mathrm{d}r \\
			 & = \frac{\sarea}{(2\pi)^{d} \vartheta^{\beta/2}}\int_{0}^{t \land s}\int_{0}^{\infty} {\sin((t-r)\fv)\sin((s-r)\fv)}\fv^{\beta-3} \mathrm{d}\fv \mathrm{d}r,
		\end{aligned}
	\end{equation*}
	where we have used a change of variables to remove the dependence of the integral on the wave speed $\vartheta$.
	Using again that
	$\sin((t-r)\fv)\sin((s-r)\fv)  =\frac{1}{2}\left(\cos((t-s)\fv)-\cos((t+s-2r)\fv)\right)$,
	we obtain for the inner integral
	\begin{equation*}
		\begin{aligned}
			 & \int_{0}^{\infty} {\sin((t-r)\fv)\sin((s-r)\fv)}\fv^{\beta-3} \mathrm{d}\fv                                                                                                   \\
			 & =\left[ \sin((t-r)\fv)\sin((s-r)\fv) \frac{\fv^{\beta-2}}{\beta-2}\right]_{0}^{\infty}                                                                                        \\
			 & \qquad-\frac{1}{2}\int_{0}^{\infty} \left({-(t-s) \sin((t-s)\fv) + (t+s-2r) \sin((t+s-2r)\fv)} \right)\frac{\fv^{\beta-2}}{\beta-2}\mathrm{d}\fv                                  \\
			 & = \frac{1}{2(\beta-2)}\left( (t-s)\int_{0}^{\infty}\sin((t-s)\fv)\fv^{\beta-2}\mathrm{d}\fv - (t+s-2r)\int_{0}^{\infty}\sin((t+s-2r)\fv)\fv^{\beta-2} \mathrm{d}\fv\right).
		\end{aligned}
	\end{equation*}
	Setting $c_\beta \coloneqq \int_{0}^{\infty}\sin(\fv)\fv^{\beta-2}d\fv>0$, we observe
	\begin{equation*}
		\begin{aligned}
			\int_{0}^{\infty} {\sin((t-r)\fv)\sin((s-r)\fv)}\fv^{\beta-3} \mathrm{d}\fv
			 & = \frac{c_\beta}{2(\beta-2)} \left( |t-s|^{2-\beta}-(t+s-2r)^{2-\beta}\right).
		\end{aligned}
	\end{equation*}
	Note that if $d \geq 2$ the integral defining $c_\beta$ is only absolutely convergent for $\beta \in (0,1)$, but it remains conditionally convergent for $\beta \in [1,2)$ and is thus finite for all $\beta \in (0, 2 \land d)$. 
    This yields for $s \leq t$:
	\begin{equation*}
		\begin{aligned}
			 & \E[u(t,x)u(s,x)]                                                                                                                                                                        \\
			 & =\frac{\sarea}{(2\pi)^{d} \vartheta^{\beta/2}}\int_{0}^{t \land s}\frac{c_\beta}{2(\beta-2)} \left(|t-s|^{2-\beta}-(t+s-2r)^{2-\beta}\right) \mathrm{d}r                     \\
			 & =\frac{\sarea c_\beta}{(2\pi)^{d} \vartheta^{\beta/2}2(\beta-2)} \left( ( t \land s)|t-s|^{2-\beta}-\int_{0}^{t \land s}(t+s-2r)^{2-\beta}
			\mathrm{d}r\right)                                                                                                                                                                                 \\
			 & =\frac{\sarea c_\beta}{(2\pi)^{d} \vartheta^{\beta/2}2(\beta-2)} \left( ( t \land s)|t-s|^{2-\beta}-\frac{1}{2}\int_{|t-s|}^{t+s}r^{2-\beta}
			\mathrm{d}r\right)                                                                                                                                                                                 \\
			 & =\frac{\sarea c_\beta}{(2\pi)^{d} \vartheta^{\beta/2}2(\beta-2)} \left( ( t \land s)|t-s|^{2-\beta}-\frac{1}{2(3-\beta)}\left((t+s)^{3-\beta}-|t-s|^{3-\beta}\right)\right)\\
          &= C_{\beta,d}\vartheta^{-\beta/2} \left( \frac{1}{2(3-\beta)}\left((t+s)^{3-\beta}-|t-s|^{3-\beta}\right)-( t \land s)|t-s|^{2-\beta}\right).
		\end{aligned}
	\end{equation*}
The constant is given by $C_{\beta, d} \coloneqq \frac{\sarea c_\beta}{(2\pi)^{d} 2(2-\beta)}$.
\end{proof}

\subsection{Proofs for \Cref{section:spatial_variation}}
We begin by finding a suitable expression for the spatial covariance function of the stochastic wave equation. 
\begin{lemma}\label[lemma]{result:spatial_covariance_function}
Given a fixed time $t \geq 0$, the covariance function of the stochastic wave equation at two different spatial locations is given by
\begin{equation*}
\E[u(t,x_k)u(t,x_l)]=\frac{t\lambda^{2-\beta}}{2(2\pi)^{d}\vartheta}\int_{\mathbb{R}^{d}} \cos\left( {(k-l)} \rho \cdot \fv \right) \left(1 - \mathrm{sinc}(2t\sqrt{\vartheta}\lambda^{-1}|\fv|)\right)|\fv|^{\beta-d-2}\mathrm{d}\fv.
\end{equation*}
\end{lemma}
\begin{proof}
This is an immediate consequence of \Cref{proposition:covariance_representation} by setting $t=s$, rescaling in $\sqrt{\vartheta}$ as well as $\lambda$ and using $\cos(x)\sin(x)=\sin(2x)/2$.
\end{proof}

\begin{proof}[Proof of \Cref{lemma:covariance_second_order_spatial_increments}]
In this proof, we are going to use the notation from \Cref{section:general_results_on_increments}. Let us denote by $c(k,l)=\E[u(t,x_k)u(t,x_l)]$ the spatial covariance function of the stochastic wave equation evaluated at our observation points $x_k$ for $k,l=0, \dots, n+1$. Then, the covariance of the second-order spatial increments can be represented as a second-order increment of the covariance function $\E[\mathbf{I}_{\mathrm{sp},k}\mathbf{I}_{\mathrm{sp},l}]=\mathcal{I}^{(2)}[c](k,l)$. Thus, pulling the incremental operator into the representation of the covariance function, \Cref{result:spatial_covariance_function} together with \Cref{lemma:f_+_-_increments} yield
\begin{equation}
\label{eq:cov_rep_second_order_sp_variation_proof_increments_pulled_in}
\begin{aligned}
\mathrm{Cov}( \mathbf{I}_{\mathrm{sp},k}, \mathbf{I}_{\mathrm{sp},l}) &=\frac{t \lambda^{2-\beta}}{2(2\pi)^{d}\vartheta}\int_{\mathbb{R}^{d}} \mathcal{I}^{(2)}[\Phi_{\mathrm{sp}, \rho \cdot \fv}](k,l)  \left(1-\mathrm{sinc}(2t\sqrt{\vartheta}\lambda^{-1}|\fv|)\right)|\fv|^{\beta-d-2}\mathrm{d}\fv\\
&=\frac{t\lambda^{2-\beta}}{2(2\pi)^{d}\vartheta}\int_{\mathbb{R}^{d}}  \mathfrak{I}^{(4)}[\cos(\cdot (\rho^{\top} \fv))](k-l) \left(1-\mathrm{sinc}(2t\sqrt{\vartheta}\lambda^{-1}|\fv|)\right)|\fv|^{\beta-d-2}\mathrm{d}\fv,
\end{aligned}
\end{equation}
with $\Phi_{\mathrm{sp}, \xi}(k,l)=\cos\left( {(k-l)} \xi \right)$. We have already computed the increments of the cosine in \Cref{result:higher_order_increments_trig}, yielding
\begin{equation*}
\begin{aligned}
\mathfrak{I}^{(4)}[\cos(\cdot (\rho^{\top} \fv))](k-l) &= 16\sin^{4}(\rho \cdot \fv/2)\cos\left((k-l) {\rho \cdot \fv}\right),
\end{aligned}
\end{equation*}
and the result follows by plugging the representation for $\mathfrak{I}^{(4)}[\Phi_{\mathrm{sp}, \rho \cdot \fv}](k-l)$ back into \eqref{eq:cov_rep_second_order_sp_variation_proof_increments_pulled_in}.
\end{proof}

\begin{proof}[Proof of \Cref{proposition:expectation_spatial_variation_t_fixed}]
In view of the Toeplitz structure of the spatial increments described by \Cref{lemma:covariance_second_order_spatial_increments}, their expectation satisfies
\begin{align*}
\E[\mathbf{V}_{\mathrm{sp}}]=\sum_{k=1}^{n}\V(\mathbf{I}_{\mathrm{sp},k})&=n\lambda^{2-\beta}\frac{8t}{(2\pi)^{d}\vartheta} \int_{\mathbb{R}^{d}} \sin^{4}(\rho \cdot \fv/2 )(1-\mathrm{sinc}(2t\sqrt{\vartheta}\lambda^{-1}|\fv|))|\fv|^{\beta-d-2}\mathrm{d}\fv\\
&=n\lambda^{2-\beta}\frac{8t}{(2\pi)^{d}\vartheta}\int_{\mathbb{R}^{d}} (1-\mathrm{sinc}(2t\sqrt{\vartheta}\lambda^{-1}|\fv|))\gspzero(\fv)\mathrm{d}\fv.
\end{align*}
With $\spconstexpt$ from \Cref{tab:asymptotic_constants} we obtain
\begin{equation*}
\begin{aligned}
\frac{\lambda^{\beta-2}}{n}\E[\mathbf{V}_{\mathrm{sp}}] - \frac{t}{\vartheta}\spconstexpt &= \frac{8t}{(2\pi)^{d}\vartheta} \int_{\mathbb{R}^{d}} \mathrm{sinc}(2t\sqrt{\vartheta}\lambda^{-1}|\fv|)\gspzero(\fv)\mathrm{d}\fv\\
&= \lambda  \frac{8t}{(2\pi)^{d}\vartheta}\int_{\mathbb{R}^{d}} \frac{\sin(2t\sqrt{\vartheta}\lambda^{-1}|\fv|)}{2t \sqrt{\vartheta}}\gspzero(\fv)/|\fv|\mathrm{d}\fv\\
&= \smallo(\lambda^{2}),
\end{aligned}
\end{equation*}
where have used the Riemann-Lebesgue Lemma in the last step and the fact that $\nabla (\gspzero(\fv)/|\fv|) \in L^{1}(\mathbb{R}^{d})$ for $\beta \in (0, 2 \land d)$ by \Cref{result:regularity_g}. \qedhere
\end{proof}

\begin{proof}[Proof of \Cref{proposition:asymptotic_scaling_for_the_variance}]
\begin{step}[Details for the derivation of \eqref{eq:partial_fourier_sum}]
Using Wicks' theorem (Isserlis' theorem) and \Cref{lemma:covariance_second_order_spatial_increments}, we have
\begin{equation*}
\begin{aligned}
&\frac{\lambda^{2\beta-4}}{n}\V(\mathbf{V}_{\mathrm{sp}}) = \frac{2 \lambda^{2\beta-4}}{n}\sum_{k,l=1}^{n}\mathrm{Cov}(\mathbf{I}_{\mathrm{sp},k}, \mathbf{I}_{\mathrm{sp},l})^{2}\\
&= \frac{t^{2}}{\vartheta^{2}}\frac{128}{(2\pi)^{2d}}\frac{1}{n}\sum_{k,l=1}^{n}\left(\int_{\mathbb{R}^{d}}\cos((k-l)\rho \cdot \fv)\gsplambda(\fv)\mathrm{d}\fv\right)^{2}\\
&=\frac{t^{2}}{\vartheta^{2}}\frac{64}{(2\pi)^{2d}} \frac{1}{n}\sum_{k,l=1}^{n} \int_{\mathbb{R}^{d}}\int_{\mathbb{R}^{d}} \left(\cos\left({(k-l)}\rho \cdot (\fv_1+\fv_2)\right)+ \cos\left({(k-l)}\rho \cdot (\fv_1-\fv_2)\right) \right)\gsplambda(\fv_1)\gsplambda(\fv_2) \mathrm{d}\fv_1 \mathrm{d}\fv_2\\
&= \frac{t^{2}}{\vartheta^{2}}\frac{128}{(2\pi)^{2d}}\int_{\mathbb{R}^{d}}\int_{\mathbb{R}^{d}} \frac{1}{2}\big(\mathfrak{F}_{\mathrm{sp}}(\rho \cdot (\fv_1+\fv_2))+ \mathfrak{F}_{\mathrm{sp}}({\rho \cdot (\fv_1-\fv_2)})\big)\gsplambda(\fv_1)\gsplambda(\fv_2) \mathrm{d}\fv_1 \mathrm{d}\fv_2. 
\end{aligned}
\end{equation*}
By plugging in the representations for the Fejér kernel \eqref{eq:introducing_the_fejér_kernel}, we can see that the $\fv_1$ and $\fv_2$ dependencies factorise and we obtain
\begin{align}
&\frac{1}{2}\big(\mathfrak{F}_{\mathrm{sp}}(\rho \cdot (\fv_1+\fv_2))+ \mathfrak{F}_{\mathrm{sp}}({\rho \cdot (\fv_1-\fv_2)})\big)\nonumber\\
&=\frac{1}{2}\left(\sum_{|k|\leq n-1}w_{\mathrm{sp},k}e^{\ii k\,\rho \cdot (\fv_1 +\fv_2)}+ \sum_{|k|\leq n-1}w_{\mathrm{sp},k}e^{\ii k\, \rho \cdot (\fv_1 -\fv_2)}\right)\nonumber\\
&= \sum_{|k|\leq n-1}w_{\mathrm{sp},k}\frac{1}{2}\left(e^{\ii k \,\rho \cdot (\fv_1 +\fv_2)}+ e^{\ii k\,\rho \cdot (\fv_1 -\fv_2)}\right)\label{eq:fejer_factorisation}\\
&=  \sum_{|k|\leq n-1}w_{\mathrm{sp},k} e^{\ii k\,\rho \cdot \fv_1}\cos(k \rho \cdot\fv_2)\nonumber\\
&=  \sum_{|k|\leq n-1}w_{\mathrm{sp},k} \cos(k\rho \cdot \fv_1)\cos(k \rho \cdot\fv_2),\nonumber
\end{align}
where the last line follows by symmetry of the weights and the cosine as well as the antisymmetry of the sine. Plugging \eqref{eq:fejer_factorisation} back into our representation of the variance yields
\begin{align*}
&\frac{\lambda^{2\beta-4}}{n}\V(\mathbf{V}_{\mathrm{sp}}) \\
&= \frac{t^{2}}{\vartheta^{2}}\frac{128}{(2\pi)^{2d}}\sum_{|k|\leq n-1}w_{\mathrm{sp},k}\int_{\mathbb{R}^{d}}\int_{\mathbb{R}^{d}} \cos(k\,\rho \cdot \fv_1)\cos(k \rho \cdot\fv_2)\gsplambda(\fv_1)\gsplambda(\fv_2) \mathrm{d}\fv_1 \mathrm{d}\fv_2\\
&= \frac{t^{2}}{\vartheta^{2}}\frac{128}{(2\pi)^{2d}}\sum_{|k|\leq n-1}w_{\mathrm{sp},k}\left(\int_{\mathbb{R}^{d}} \cos(k\,\rho \cdot \fv_1)\gsplambda(\fv_1)\mathrm{d}\fv_1 \right) \left(\int_{\mathbb{R}^{d}}\cos(k \rho \cdot \fv_2)\gsplambda(\fv_2) \mathrm{d}\fv_2\right)\\
&=\frac{t^{2}}{\vartheta^{2}}\frac{128}{(2\pi)^{2d}}\sum_{|k|\leq n-1}w_{\mathrm{sp},k}|\mathcal{F}_c(\gsplambda)(k\rho)|^{2}.
\end{align*}
\end{step}
\begin{step}[Determining the limit]
By dominated convergence, we obtain as $\lambda \rightarrow 0$ the limit 
\begin{equation*}
\mathcal{F}_c(\gsplambda)(k \rho) \rightarrow \mathcal{F}_c(\gspzero)(k \rho), \quad \lambda \rightarrow 0,
\end{equation*}
where we have used \Cref{result:uniform_integrability_f_lambda} for the uniform integrability of $\gsplambda$, i.e.\ of the series $((g_{\mathrm{sp}, \lambda_n})_{n \in \mathbb{N}}$. \Cref{result:uniform_integrability_f_lambda} also shows that the derivative remains uniformly integrable for $\lambda>0$, such that independently of $\lambda$ we have $|\mathcal{F}_c(\gsplambda)(k \rho)| \lesssim |k|^{-1}$, see also \citet[Theorem 3.3.9]{grafakosClassicalFourierAnalysis2014}. Thus, we have the uniform decay $|\mathcal{F}_c(\gsplambda)(k\rho)|^{2} \lesssim |k|^{-2}$. Recall $\spconstvart$ from \Cref{tab:asymptotic_constants}. Another application of the dominated convergence theorem concludes the result: 
\begin{equation*}
 \frac{\lambda^{2\beta-4}}{n}\V(\mathbf{V}_{\mathrm{sp}})\rightarrow \frac{t^{2}}{\vartheta^{2}}\frac{128}{(2\pi)^{2d}}\sum_{k \in \mathbb{Z}}|\mathcal{F}_c(\gspzero)(k\rho)|^{2}=\frac{t^2}{\vartheta^2}\spconstvart < \infty, \quad n \rightarrow \infty, \quad\lambda \rightarrow 0. 
\end{equation*}
Note that we may extend the finite sum to an infinite sum by extending the weights through zero for all the remaining values.\qedhere
\end{step}
\end{proof}

\begin{lemma}\label[lemma]{eq:numerator_control_spatial_t_fixed}
The numerator in \eqref{eq:CLT_condition_maximum_row_norm_version} satisfies
\begin{equation*}
\left(\max_{k=1, \dots, n}\sum_{l=1}^{n}|\mathrm{Cov}( \mathbf{I}_{\mathrm{sp},k}, \mathbf{I}_{\mathrm{sp},l})| \right)^{2} = \mathcal{O}\left( \lambda^{4-2\beta}\log(n)^{2}\right).
\end{equation*}
\end{lemma}
\begin{proof}
By \Cref{lemma:covariance_second_order_spatial_increments}, the covariance matrix of the vector of increments is a Toeplitz matrix and its entries may be abbreviated by $r_{k-l} \coloneqq \mathrm{Cov}( \mathbf{I}_{\mathrm{sp},k}, \mathbf{I}_{\mathrm{sp},l})$.
Each entry only depends on the difference $k-l$ of indices. In particular, if we sum over all the possible configurations, we have
\begin{equation*}
\max_{k=1, \dots, n}\sum_{l=1}^{n}|r_{k-l}| \leq 2\sum_{|k| \leq n-1} |r_{k}|.
\end{equation*}
We further observe by integration by parts and \Cref{result:uniform_integrability_f_lambda}:
\begin{equation*}
\begin{aligned}
|r_k|&=\lambda^{2-\beta}\frac{8t}{(2\pi)^{d}\vartheta}\left|\int_{\mathbb{R}^{d}}  \cos(k {\rho \cdot \fv}) g_{\mathrm{sp},\lambda}(\fv)\mathrm{d}\fv\right|\lesssim \lambda^{2-\beta} \frac{1}{|k|}, \quad 1 \leq |k| \leq n-1. 
\end{aligned}
\end{equation*}
As a consequence, whenever $k \neq 0$, it holds that $\sum_{1 \leq |k| \leq n-1} |r_k| \lesssim \lambda^{2-\beta} \sum_{k=1}^{n}\frac{1}{k}$. The result follows from the fact that the $n$-th harmonic number grows like $\mathrm{log}(n)$.
\end{proof}

\begin{proof}[Proof of \Cref{theorem:clt_spatial_fixed_t}]
By \Cref{eq:numerator_control_spatial_t_fixed} and \Cref{proposition:asymptotic_scaling_for_the_variance}, we observe 
\begin{equation*}
 \frac{\left(\max_{k=1, \dots,n}\sum_{l=1}^{n} |\mathrm{Cov}(\mathbf{I}_{\mathrm{sp},k}, \mathbf{I}_{\mathrm{sp},l})|\right)^{2}}{\V(\mathbf{V}_{\mathrm{sp}})} =\mathcal{O}\left(\frac{\lambda^{4-2\beta}\log(n)^{2}}{n \lambda^{4-2\beta}}\right)=\smallo(1), \quad n \rightarrow \infty.
\end{equation*}
Therefore, by \citet[Proposition 3.1]{hildebrandtParameterEstimationSPDEs2019}, we have
\begin{equation*}
\frac{\mathbf{V}_{\mathrm{sp}} - \E[\mathbf{V}_{\mathrm{sp}}]}{\sqrt{\V(\mathbf{V}_{\mathrm{sp}})}} \xrightarrow{d} N(0,1), \quad n \rightarrow \infty. 
\end{equation*}
Consequently, rescaling and Slutsky's theorem yields by \Cref{proposition:asymptotic_scaling_for_the_variance}
\begin{equation*}
\sqrt{n}\left(\frac{\lambda^{\beta-2}}{n}\mathbf{V}_{\mathrm{sp}}- \E\left[\frac{\lambda^{\beta-2}}{n}\mathbf{V}_{\mathrm{sp}}\right]\right) \xrightarrow{d} N\left(0, \frac{t^{2}}{\vartheta^{2}}\spconstvart\right), \quad n \rightarrow \infty. \qedhere
\end{equation*}
\end{proof}
\begin{proof}[Proof of \Cref{corollary:spatial_clt_fixed_time_expectation_swap}]
We decompose 
\begin{equation*}
\begin{aligned}
&\sqrt{n}\left(\frac{\lambda^{\beta-2}}{n}\mathbf{V}_{\mathrm{sp}}- \frac{t}{\vartheta}\spconstexpt\right)  \\
&= \sqrt{n}\left(\frac{\lambda^{\beta-2}}{n}\mathbf{V}_{\mathrm{sp}}- \E\left[\frac{\lambda^{\beta-2}}{n}\mathbf{V}_{\mathrm{sp}}\right]\right) + \sqrt{n}\left(\frac{\lambda^{\beta-2}}{n}\E[\mathbf{V}_{\mathrm{sp}}] - \frac{t}{\vartheta}\spconstexpt\right).\\
\end{aligned}
\end{equation*}
The result follows by combining \Cref{eq:expectation_convergence_speed_result_spatial_fixed_t} and \Cref{theorem:clt_spatial_fixed_t} and another application of Slutsky's theorem.
\end{proof}

\begin{proof}[Proof of \Cref{corollary:clt_estimator_spatial_fixed_time}]
We apply the delta method with $f(x)=t \spconstexpt/x$. The derivative of $f$ is given by $f'(x)=-t \spconstexpt/x^{2}$ and we obtain with $f'(t\spconstexpt/\vartheta))=-\vartheta^{2}/(t\spconstexpt )$
\begin{equation*}
\sqrt{n}(\hat{\vartheta}_{\mathrm{sp},n}-\vartheta) \xrightarrow{d} N\left(0, \frac{t^{2}}{\vartheta^{2}}\spconstvart f'(t \spconstexpt/\vartheta)^{2}\right)=N\left(0, \frac{\vartheta^{2}\spconstvart}{(\spconstexpt)^{2}} \right). \qedhere
\end{equation*}
\end{proof}

\subsection{Proofs for \Cref{section:temporal_variations}}
\begin{proof}[Proof of \Cref{lemma:covariance_second_order_temporal_increments}]
The notation for increments of functions used throughout this proof is based on \Cref{section:general_results_on_increments}. By \Cref{proposition:covariance_representation} the covariance admits the representation
\begin{equation*}
\begin{aligned}
\E[u(t_i,x)u(t_j,x)]
&= \frac{\delta^{3-\beta}}{ (2\pi)^{d}\vartheta^{\beta/2}} \int_{\mathbb{R}^{d}} \Phi_{\mathrm{te}, |\fv|}(i,j) |\fv|^{\beta-d-2}\mathrm{d}\fv, 
\end{aligned}
\end{equation*}
with
\begin{equation*}
\begin{aligned}
\Phi_{\mathrm{te}, \xi}(t,s) &= \frac{1}{2}\left((t \land s)\cos((t-s)\xi)-\frac{\cos((t \lor s)\xi)\sin((t \land s)\xi)}{\xi}\right)\\
&=\frac{1}{2}\left((t \land s)\cos((t-s)\xi)+\frac{\sin(|t-s|\xi)-\sin((t+s)\xi)}{2\xi}\right)\\
&=\frac{1}{2}\left( f_{\xi,1}(t,s) + \frac{f_{\xi,2}(t,s)-f_{\xi,3}(t,s)}{2\xi}\right),
\end{aligned}
\end{equation*}
where $f_{\xi,1}(t,s)=(t \land s) \cos((t-s)\xi)$, $f_{\xi,2}(t,s)=\sin(|t-s|\xi)$, $f_{\xi,3}(t,s)=\sin((t+s)\xi)$ for $\xi\in\mathbb R$.
The covariance between two temporal increments is given by
\begin{equation}
\label{eq:general_form_temp_cov_f_1_2_3}
\begin{aligned}
\mathrm{Cov}(\mathbf{I}_{\mathrm{te},i}, \mathbf{I}_{\mathrm{te},j}) &= \frac{\delta^{3-\beta}}{ (2\pi)^{d}\vartheta^{\beta/2}} \int_{\mathbb{R}^{d}} \mathcal{I}^{(2)}[\Phi_{\mathrm{te}, |\fv|}](i,j) |\fv|^{\beta-d-2}\mathrm{d}\fv\\
&=\frac{\delta^{3-\beta}}{ (2\pi)^{d}\vartheta^{\beta/2}} \int_{\mathbb{R}^{d}} \frac{1}{2}\left( \mathcal{I}^{(2)}[f_{|\fv|,1}](i,j) + \frac{\mathcal{I}^{(2)}[f_{|\fv|,2}](i,j)-\mathcal{I}^{(2)}[f_{|\fv|,3}](i,j)}{2|\fv|}\right) |\fv|^{\beta-d-2}\mathrm{d}\fv.
\end{aligned}
\end{equation}
By \Cref{result:higher_order_increments_trig}, we obtain a closed form representation of odd and even increments of the cosine
\begin{equation}
\label{eq:trig_expression_cos}
\begin{aligned}
\mathfrak{I}^{(4)}[\cos(\cdot \xi)](|i-j|)&=2^{4}\sin^{4}(\xi/2)\cos\left(|i-j|\xi + 2\pi\right)=2^{4}\sin^{4}(\xi/2)\cos\left(|i-j|\xi\right),\\
\mathfrak{I}^{(3)}[\cos(\cdot \xi)](|i-j|)&=-2^{3}\sin(\xi)\sin^{2}(\xi/2)\sin(|i-j|\xi + \pi)=\operatorname{sign}(i-j)2^{3}\sin(\xi)\sin^{2}(\xi/2)\sin((i-j)\xi)\\
&=2^{3}\sin(\xi)\sin^{2}(\xi/2)\sin(|i-j|\xi).
\end{aligned}
\end{equation}
Clearly, since the cosine is symmetric, the mapping $x \mapsto \cos(x \xi)$ is also symmetric and we may apply \Cref{lemma:minimum_increment_lemma} with $f_{\xi,1}$, yielding with \eqref{eq:trig_expression_cos}:
\begin{equation*}
\begin{aligned}
\mathcal{I}^{(2)}[f_{\xi,1}](i,j)&=
\begin{cases}
(i\land j)\mathfrak{I}^{(4)}[\cos(\cdot \xi)](|i-j|) -\mathfrak{I}^{(3)}[\cos(\cdot \xi)](|i-j|)+ 4\cos(\xi)-2\cos(2\xi), &i=j,\\
(i \land j)\mathfrak{I}^{(4)}[\cos(\cdot \xi)](|i-j|)-\mathfrak{I}^{(3)}[\cos(\cdot \xi)](|i-j|)-\cos(\xi), &|i-j|=1,\\
(i\land j) \mathfrak{I}^{(4)}[\cos(\cdot \xi)](|i-j|) - \mathfrak{I}^{(3)}[\cos(\cdot \xi)](|i-j|), &|i-j|\geq2
\end{cases}\\
&=\begin{cases}
(i\land j)2^{4}\sin^{4}(\xi/2)\cos(|i-j|\xi)+ 4\cos(\xi)-2\cos(2\xi), &i=j,\\
(i \land j)2^{4}\sin^{4}(\xi/2)\cos(|i-j|\xi) - 2^{3}\sin^2(\xi)\sin^{2}(\xi/2)-\cos(\xi), &|i-j|=1,\\
(i\land j) 2^{4}\sin^{4}(\xi/2)\cos\left(|i-j|\xi\right) - 2^{3}\sin(\xi)\sin^{2}(\xi/2)\sin(|i-j|\xi), &|i-j|\geq2.
\end{cases}
\end{aligned}
\end{equation*}

By \Cref{result:increment_sine_absolute}, \Cref{result:higher_order_increments_trig} and \Cref{lemma:f_+_-_increments}, we further have
\begin{equation*}
\begin{aligned}
\mathcal{I}^{(2)}[f_{\xi,2}](i,j)&=\begin{cases}
-8\sin(\xi)+2\sin(2\xi), & i=j,\\
7\sin(\xi)+\sin(3\xi)-4\sin(2\xi), & |i-j|=1,\\
2^{4}\sin^{4}(\xi/2)\sin(|i-j|\xi), & |i-j| \geq 2,
\end{cases}
\end{aligned}
\end{equation*}
and
\begin{equation*}
\begin{aligned}
\mathcal{I}^{(2)}[f_{\xi,3}](i,j) &= \mathfrak{I}^{(4)}[\sin(\cdot \xi)](i+j) = 2^{4}\sin^{4}(\xi/2)\sin\left((i+j)\xi\right). 
\end{aligned}
\end{equation*}
Plugging the general representation for the second-order increments back into \eqref{eq:general_form_temp_cov_f_1_2_3} yields
\begin{equation*}
\begin{aligned}
&\mathrm{Cov}(\mathbf{I}_{\mathrm{te},i}, \mathbf{I}_{\mathrm{te},j})\\
&= \frac{8\delta^{3-\beta}}{ (2\pi)^{d}\vartheta^{\beta/2}}(i \land j)\int_{\mathbb{R}^{d}} \sin^{4}(|\fv|/2) \cos((i-j) |\fv|)|\fv|^{\beta-d-2}\mathrm{d}\fv + \delta^{3-\beta}R(i,j)
\end{aligned}
\end{equation*}
with
\begin{equation}
\label{eq:time_remainders}
\begin{aligned}
R(i,j)&\coloneqq R^{(1)}(i,j)+R^{(2)}(i,j)+R^{(3)}(i,j),\\
R^{(1)}(i,j) &\coloneqq \frac{4}{ (2\pi)^{d}\vartheta^{\beta/2}}\int_{\mathbb{R}^{d}} \sin^{4}(|\fv|/2) r_{i,j}^{(1)}(|\fv|)|\fv|^{\beta-d-3} \mathrm{d}\fv,\\
R^{(2)}(i,j) &\coloneqq \frac{1}{2 (2\pi)^{d}\vartheta^{\beta/2}}\int_{\mathbb{R}^{d}} r_{i,j}^{(2)}(|\fv|)|\fv|^{\beta-d-2}\mathrm{d}\fv,\\
R^{(3)}(i,j) &\coloneqq \frac{1}{2 (2\pi)^{d}\vartheta^{\beta/2}}\int_{\mathbb{R}^{d}} r_{i,j}^{(3)}(|\fv|)|\fv|^{\beta-d-2}\mathrm{d}\fv,
\end{aligned}
\end{equation}
where
\begin{equation}
\label{eq:temporal_remainder_terms}
\begin{aligned}
r_{i,j}^{(1)}(\xi) &\coloneqq 
\begin{cases}
-\sin(2i\xi), &i=j,\\
-\sin((i+j)\xi),& |i-j|=1,\\
\sin(|i-j|\xi)-\sin((i+j)\xi),& |i-j| \geq 2,
\end{cases}
\\
r_{i,j}^{(2)}(\xi)&\coloneqq \begin{cases}
4\cos(\xi)-2\cos(2\xi)-4\mathrm{sinc}(\xi)+2\mathrm{sinc}(2\xi), &i=j,\\
-\cos(\xi)+\frac{7}{2}\mathrm{sinc}(\xi)+\frac{3}{2}\mathrm{sinc}(3\xi)-4\mathrm{sinc}(2\xi), & |i-j|=1,\\
0, &|i-j|\geq 2,
\end{cases}\\
r_{i,j}^{(3)}(\xi) &\coloneqq \begin{cases}
0, & i =j,\\
-8\sin^{2}(\xi)\sin^{2}(\xi/2), &|i-j|=1,\\
-8\sin(\xi)\sin^{2}(\xi/2)\sin(|i-j|\xi), &|i-j|\geq 2.
\end{cases}
\end{aligned}
\end{equation}
The boundedness of the remainder term $R(i,j)$ is shown in \Cref{lemma:remainder_bounds_temp}.
\end{proof}

\begin{lemma}\label[lemma]{lemma:remainder_bounds_temp}
    The reminder terms $R(i,j)$, $i,j\leq m,$ given in \eqref{eq:time_remainders} satisfy
    \begin{equation*}
        \sup_{i,j\leq m}|R(i,j)|\lesssim c
    \end{equation*}
    for some universal constant $c$.
\end{lemma}
\begin{proof}
Due to the decomposition \eqref{eq:time_remainders} it suffices to show that $R^{(z)}(i,j)$, $z\in\{1,2,3\},$ $i,j\leq m,$ are uniformly bounded.
\begin{case}[Remainder $R^{(1)}$]
    As $|r_{i,j}^{(1)}(r)|\leq 2$ and $|\sin(r/2)|\leq (1\land r)$ for $r\geq0$, it immediately follows by splitting the integral
    \begin{equation*}
        \begin{aligned}
            \sup_{i,j\leq m}\left|R^{(1)}(i,j)\right| &= \frac{4}{ (2\pi)^{d}\vartheta^{\beta/2}}\left|\int_{\mathbb{R}^{d}} \sin^{4}(|\fv|/2) r_{i,j}^{(1)}(|\fv|)|\fv|^{\beta-d-3} \mathrm{d}\fv\right|\\
            &\leq \frac{8\sarea}{ (2\pi)^{d}\vartheta^{\beta/2}}\int_0^\infty\sin^4(r/2)r^{\beta-4}\mathrm{d}r\\
            &\lesssim \int_0^1r^{\beta}\mathrm{d}r+\int_1^\infty r^{\beta-4}\mathrm{d}r
            <\infty.
        \end{aligned}
    \end{equation*}
\end{case}
\begin{case}[Remainder $R^{(2)}$]
    
Consider the case $i=j$. Then
    \begin{equation*}
    \begin{aligned}
        R^{(2)}(i,j)&=\frac{1}{2 (2\pi)^{d}\vartheta^{\beta/2}}\int_{\mathbb{R}^{d}} r_{i,j}^{(2)}(|\fv|)|\fv|^{\beta-d-2}\mathrm{d}\fv\\
        &=\frac{1}{2 (2\pi)^{d}\vartheta^{\beta/2}}\int_{\mathbb{R}^{d}} \left(4(\cos(|\fv|)-\operatorname{sinc}(|\fv|))-2(\cos(2|\fv|)-\operatorname{sinc}(2|\fv|))\right)|\fv|^{\beta-d-2}\mathrm{d}\fv\\
        &=\frac{\sarea}{2 (2\pi)^{d}\vartheta^{\beta/2}}\int_{0}^\infty \left(4(\cos(r)-\operatorname{sinc}(r))-2(\cos(2r)-\operatorname{sinc}(2r))\right)r^{\beta-3}\mathrm{d}r.\\
    \end{aligned}
    \end{equation*}
    As in the case of $R^{(1)}$, by splitting the integral it is sufficient to show integrability around $0$ due to the rapid decay of $r^{\beta-3}$. We will show that 
    \begin{equation}
    \label{eq: diff_cos_sinc}
        |\cos(r)-\operatorname{sinc}(r)|\lesssim r^2,\quad r\in[0,1],
    \end{equation}
    implying 
    $$\int_{0}^1|\cos(r)-\operatorname{sinc}(r)|r^{\beta-3}\mathrm{d}r<\infty.$$
    By the series representation of the sine and cosine function, it holds for $r\leq1$
    \begin{equation*}
        \begin{aligned}
            |\cos(r)-\operatorname{sinc}(r)|&=\left|\sum_{n=0}^\infty\frac{(-1)^n}{(2n)!}r^{2n}-\sum_{n=0}^\infty\frac{(-1)^n}{(2n+1)!}r^{2n}\right|
            =\left|\sum_{n=0}^\infty\frac{(-1)^n2n}{(2n+1)!}r^{2n}\right| \\
            &\leq \frac{r^2}{3}+r^2\sum_{n=2}^\infty\frac{1}{(2n)!}\leq \frac{r^2}{3}+r^2\sum_{n=2}^\infty\frac{1}{2^n}< r^2,
        \end{aligned}
    \end{equation*}
    proving \eqref{eq: diff_cos_sinc} and thus $|\cos(2r)-\operatorname{sinc}(2r)|\lesssim r^2$, too. Likewise, the same arguments apply for the case $|i-j|=1$.
\end{case}
\begin{case}[Remainder $R^{(3)}$]
    Since $|r^{(3)}_{i,j}(r)|\leq 8(1\land r^3)$ for $r\geq0$, we obtain similarly to the case $R^{(1)}$ the upper bound
    $$\sup_{i,j\leq m}\left|R^{(3)}(i,j)\right|\lesssim 1$$
    by splitting the integral into two parts and exploiting the regularity of the sine around the origin. \qedhere
\end{case}
\end{proof}
\begin{proof}[Proof of \Cref{proposition:expectation_temp_fixed_space}]

Recall $\tempconstexpt$ from \Cref{tab:asymptotic_constants}. By \Cref{lemma:covariance_second_order_temporal_increments} and \Cref{lemma:remainder_bounds_temp}, we have
\begin{align*}
\delta^{\beta-3}m^{-2}\E[\mathbf{V}_{\mathrm{te},m}]&= \frac{1}{m^{2}}\sum_{i=1}^{m}i \left(\frac{8}{ (2\pi)^{d}\vartheta^{\beta/2}} \right)\int_{\mathbb{R}^{d}} \gtemp(\fv)\mathrm{d}\fv + \frac{1}{m^{2}}\sum_{i=1}^{m} R(i,i)\\
&= \vartheta^{-\beta/2}\tempconstexpt + \frac{1}{m}\vartheta^{-\beta/2} \tempconstexpt + \mathcal{O}\left(\frac{1}{m} \right).\tag*{\qedhere}
\end{align*}
\end{proof}

\begin{lemma}\label[lemma]{lemma:fourier_representation_time_fejer_kernel}
The kernel defined through \eqref{eq:fejer_time_kernel}
admits the representation
\begin{equation*}
\mathfrak{F}_{\mathrm{te}}(x)=\sum_{|j|\leq m-1} w_{\mathrm{te},j}e^{\ii jx}={\sum_{|j| \leq m-1} {w}_{\mathrm{te},j} \cos(jx)},
\end{equation*}
with the weights 
\begin{equation*}
w_{\mathrm{te},j}= \frac{1}{m^{3}}\sum_{i=1}^{m-|j|} i^{2},
\end{equation*}
satisfying $0 \leq w_{\mathrm{te},j}\leq 1$ and $w_{\mathrm{te},j}\rightarrow 1/3$ as $m \rightarrow \infty$ for all $j$.
\end{lemma}
\begin{proof}
Note that since both the cosine and $(i \land j)^{2}$ are symmetric, we immediately obtain the identity
\begin{equation*}
\mathfrak{F}_{\mathrm{te}}(x)=\frac{1}{m^{3}}\sum_{i=1}^{m}\sum_{j=1}^{m}(i \land j)^{2}\cos((i-j)x)=\frac{1}{m^{3}}\sum_{i=1}^{m}\sum_{j=1}^{m}(i \land j)^{2}e^{\ii (i-j)x}.
\end{equation*}
Setting $i-j=j'$, we obtain by the symmetry of the weights $w_{\mathrm{te},j}$
\begin{equation*}
\begin{aligned}
    \mathfrak{F}_{\mathrm{te}}(x)&=\sum_{|j'| \leq m-1} \left(\frac{1}{m^{3}}\sum_{i,j=1,\dots,m \colon i-j=j'}^{m}(i \land j)^{2} \right)e^{\ii j' x}\\
    &=\sum_{|j'| \leq m-1} \left(\frac{1}{m^{3}}\sum_{j=1}^{m-|j'|} j^{2} \right)e^{\ii j' x}=\sum_{|j'|\leq m-1} w_{\mathrm{te},j'}e^{\ii j'x}\\
    &={\sum_{|j| \leq m-1} {w}_{\mathrm{te},j} \cos(jx)}.
\end{aligned}
\end{equation*}
Note that $0 \leq w_{\mathrm{te},j} \leq 1$. The limit follows immediately by setting $\tilde{m}=m-|j'|$ and observing
\begin{equation*}
w_{\mathrm{te},j}=\frac{1}{m^{3}}\sum_{j=1}^{m-|j'|} j^{2}= \frac{\tilde{m}(\tilde{m}+1)(2\tilde{m}+1)}{6m^{3}} \rightarrow \frac{1}{3}, \quad m \rightarrow \infty,
\end{equation*}
where we have used that $\tilde{m}^{3}/m^{3} \rightarrow 1$. 
\end{proof}
\begin{proof}[Proof of \Cref{proposition:variance_temporal_second_order_variation_fixed_space}]
    We observe with Wick's theorem that
\begin{align}
\delta^{2\beta-6}m^{-3}\V(\mathbf{V}_{\mathrm{te}})&= \delta^{2\beta-6}m^{-3}2\sum_{i,j=1}^{m} \mathrm{Cov}(\mathbf{I}_{\mathrm{te},i}, \mathbf{I}_{\mathrm{te},j})^{2}\label{eq: tempvarcalculation}\\
&=  2m^{-3}\sum_{i,j=1}^{m}\left((i \land j) \left(\frac{8}{ (2\pi)^{d}\vartheta^{\beta/2}} \right)\int_{\mathbb{R}^{d}}\cos(|\fv| (i-j)) \gtemp(\fv)\mathrm{d}\fv + R(i,j)\right)^{2}.\nonumber
\end{align}
Recall $\tempconstvart$ from \Cref{tab:asymptotic_constants}. In the first step, we will show the convergence of the leading order expression
\begin{equation}
\label{eq: firststeptempvar}
    2\left(\frac{8}{ (2\pi)^{d}\vartheta^{\beta/2}} \right)^2m^{-3}\sum_{i,j=1}^{m} \left((i \land j)\int_{\mathbb{R}^{d}}\cos(|\fv| (i-j)) \gtemp(\fv)\mathrm{d}\fv\right)^{2}\rightarrow\frac{\tempconstvart}{\vartheta^\beta},\quad m\rightarrow\infty.
\end{equation}
In the second step, we verify that the the remainder terms $R(i,j)$, $i,j\leq m,$ in \eqref{eq: tempvarcalculation} and their crossterms with the leading order expression are neglectable, proving the assertion.
\begin{step}[Determining the limit]
    \begin{equation*}
    \begin{aligned}
   &2\left(\frac{8}{ (2\pi)^{d}\vartheta^{\beta/2}} \right)^{2}\frac{1}{m^{3}}\sum_{i,j=1}^{m}\left((i \land j)\int_{\mathbb{R}^{d}}\cos(|\fv| (i-j))\gtemp(\fv)\mathrm{d}r\right)^{2}\\
    &=2\left(\frac{8}{ (2\pi)^{d}\vartheta^{\beta/2}} \right)^{2} \int_{\mathbb{R}^{d}}\int_{\mathbb{R}^{d}} \frac{1}{m^{3}}\sum_{i,j=1}^{m}(i \land j)^{2}\cos(|\fv_1|(i-j))\cos(|\fv_2|(i-j))\gtemp(\fv_1)\gtemp(\fv_2)\mathrm{d}\fv_1 \mathrm{d}\fv_2 \\
    &=2\left(\frac{8}{ (2\pi)^{d}\vartheta^{\beta/2}} \right)^{2} \int_{\mathbb{R}^{d}}\int_{\mathbb{R}^{d}} \frac{1}{2}\frac{1}{m^{3}}\sum_{i,j=1}^{m}(i \land j)^{2}[\cos((|\fv_1|+|\fv_2|)(i-j))+\cos((|\fv_1|-|\fv_2|)(i-j))]\\
    &\hspace{60mm}\cdot\gtemp(\fv_1)\gtemp(\fv_2)\mathrm{d}\fv_1 \mathrm{d}\fv_2 \\
    &=2\left(\frac{8}{ (2\pi)^{d}\vartheta^{\beta/2}} \right)^{2} \int_{\mathbb{R}^{d}}\int_{\mathbb{R}^{d}} \frac{1}{2}[\mathfrak{F}_{\mathrm{te}}(|\fv_1|+|\fv_2|)+\mathfrak{F}_{\mathrm{te}}(|\fv_1|-|\fv_2|)]\gtemp(\fv_1)\gtemp(\fv_2)\mathrm{d}\fv_1 \mathrm{d}\fv_2.
    \end{aligned}
    \end{equation*}
    Next, we can plug in the Fourier representation from \Cref{lemma:fourier_representation_time_fejer_kernel} and obtain 
    \begin{equation}
    \label{eq:time_fejer_factorisation}
    \begin{aligned}
     \frac{1}{2}[\mathfrak{F}_{\mathrm{te}}(|\fv_1|+|\fv_2|)+\mathfrak{F}_{\mathrm{te}}(|\fv_1|-|\fv_2|)] &= \sum_{|j| \leq m-1} w_{\mathrm{te},j}\frac{1}{2}[e^{\ii j(|\fv_1|+|\fv_2|)}+e^{\ii j(|\fv_1|-|\fv_2|)}] \\
    &= \sum_{|j| \leq m-1} w_{\mathrm{te},j} e^{\ii j|\fv_1|}\cos(j|\fv_2|)\\
    &=\sum_{|j| \leq m-1} w_{\mathrm{te},j} \cos(j|\fv_1|)\cos(j|\fv_2|),
    \end{aligned}
    \end{equation}
    where the last equality follows by symmetry of the weights $w_{\mathrm{te}}$ and the cosine as well as the antisymmetry of the sine function. Plugging \eqref{eq:time_fejer_factorisation} back into the penultimate display yields
    \begin{equation*}
    \begin{aligned}
        &2\left(\frac{8}{ (2\pi)^{d}\vartheta^{\beta/2}} \right)^{2}\frac{1}{m^{3}}\sum_{i,j=1}^{m}\left((i \land j)\int_{\mathbb{R}^{d}}\cos(|\fv| (i-j))\gtemp(\fv)\mathrm{d}\fv\right)^{2}\\
        &=2\left(\frac{8}{ (2\pi)^{d}\vartheta^{\beta/2}} \right)^{2}\sum_{|j| \leq m-1}w_{\mathrm{te},j} |\mathcal{F}^+_c(\gtemp)(j)|^{2}.
    \end{aligned}
    \end{equation*}
    By \Cref{result:regularity_g} it is clear that both $\gtemp$ and its derivative live in $L^{1}(\mathbb{R}^{d})$. Since $\gtemp$ is radial, we have
    \begin{equation*}
    \mathcal{F}_c^{+}(\gtemp)(j)=\int_{\mathbb{R}^{d}} \gtemp(\fv)\cos(|\fv|j) \mathrm{d}\fv = \int_{0}^{\infty}\cos(rj)\sin^{4}(r)r^{\beta-3}\mathrm{d}r.
    \end{equation*}
    As consequence, we know that $|\mathcal{F}_c^{+}(\gtemp)(j)| \lesssim \frac{ 1}{|j|}$. Furthermore, the weights satisfy $0 \leq w_{\mathrm{te},j}\leq 1$ with the pointwise limit $w_{\mathrm{te},j} \rightarrow 1/3$. We may extend the weights by zero, i.e.\ $w_{\mathrm{te},j}=0$ whenever $|j| > m-1$ and the dominated convergence theorem yields
    \begin{equation*}
    \sum_{|j| \leq m-1}w_{\mathrm{te},j} |\mathcal{F}_c^{+}(\gtemp)(j)|^{2}= \sum_{j \in \mathbb{Z}}w_{\mathrm{te},j} |\mathcal{F}_c^{+}(\gtemp)(j)|^{2} \rightarrow \frac{1}{3}\sum_{j \in \mathbb{Z}} |\mathcal{F}_c^{+}(\gtemp)(j)|^{2}< \infty.
    \end{equation*}
    Thus, the leading order term satisfies
    \begin{equation*}
        2\left(\frac{8}{ (2\pi)^{d}\vartheta^{\beta/2}} \right)^{2}\frac{1}{m^{3}}\sum_{i,j=1}^{m}\left((i \land j)\int_{\mathbb{R}^{d}}\cos( |\fv| (i-j))\gtemp(\fv)\mathrm{d}\fv\right)^{2}\rightarrow\frac{\tempconstvart}{\vartheta^\beta},\quad m\rightarrow\infty.
    \end{equation*}
\end{step}
\begin{step}[Elimination of the remainder terms]
    By \Cref{lemma:remainder_bounds_temp} the remainders $R(i,j),$ $i,j\leq m,$ are of order $\mathcal{O}(1)$ independently of $i$ and $j$. By the Cauchy-Schwarz inequality for double sums, i.e.\ 
    \begin{equation}
    \label{eq:cauchy-schwarz}
        \left|\sum_{i,j=1}^{m}a_{i,j}b_{i,j}\right| \lesssim\sqrt{\sum_{i,j=1}^{m}a_{i,j}^{2} \sum_{i,j=1}^{m}b_{i,j}^{2}},
    \end{equation}
    this yields together with the first step
    \begin{equation*}
        \begin{aligned}
            &\delta^{2\beta-6}m^{-3}\V(\mathbf{V}_{\mathrm{te}})\\
            &=2m^{-3}\sum_{i,j=1}^{m}\left((i \land j) \left(\frac{8}{ (2\pi)^{d}\vartheta^{\beta/2}} \right)\int_{\mathbb{R}^{d}}\cos(|\fv| (i-j)) \gtemp(\fv)\mathrm{d}\fv + R(i,j)\right)^{2}\\
            &=2m^{-3}\sum_{i,j=1}^{m}\left((i \land j) \left(\frac{8}{ (2\pi)^{d}\vartheta^{\beta/2}} \right)\int_{\mathbb{R}^{d}}\cos(|\fv| (i-j)) \gtemp(\fv)\mathrm{d}\fv\right)^2 +\mathcal{O}(m^{-1/2})
        \end{aligned}
    \end{equation*}
    since     $\sum_{i,j=1}^mR(i,j)^{2}=\mathcal{O}(m^2).$ \qedhere
\end{step}
\end{proof}

\begin{lemma}
\label[lemma]{result:numerator_temporal_variation_fixed_space}
The numerator in \eqref{eq:CLT_condition_maximum_row_norm_version_temporal} satisfies
\begin{equation*}
\left(\max_{i=1, \dots,m}\sum_{j=1}^{m} |\mathrm{Cov}(\mathbf{I}_{\mathrm{te},i}, \mathbf{I}_{\mathrm{te},j})|\right)^{2} = \delta^{6-2\beta}\mathcal{O}( m^{2}\left(1 + \log(m)\right)^{2}). 
\end{equation*}
\end{lemma}
\begin{proof}
By \Cref{lemma:covariance_second_order_temporal_increments} and \Cref{lemma:remainder_bounds_temp}, we have
\begin{equation*}
\begin{aligned}
&|\mathrm{Cov}(\mathbf{I}_{\mathrm{te},i}, \mathbf{I}_{\mathrm{te},j})| \\
&=  \delta^{3-\beta} \left|(i \land j) \left(\frac{8}{ (2\pi)^{d}\vartheta^{\beta/2}} \right)\int_{\mathbb{R}^{d}}\cos(|\fv| (i-j)) \gtemp(\fv)\mathrm{d}\fv + R(i,j)\right|\\
&\lesssim \delta^{3-\beta}\left( \sup_{i,j\leq m}|R(i,j)| + (i \land j) \left|\int_{\mathbb{R}^{d}}\cos(|\fv| (i-j)) \gtemp(\fv)\mathrm{d}\fv\right|\right)\\
&\lesssim \delta^{3-\beta}\left(1 + (i \land j)\left|\int_{\mathbb{R}^{d}}\cos(|\fv| (i-j)) \gtemp(\fv)\mathrm{d}\fv\right|\right).
\end{aligned}
\end{equation*}
Consequently,
\begin{equation*}
\begin{aligned}
\max_{i=1, \dots,m}\sum_{j=1}^{m} |\mathrm{Cov}(\mathbf{I}_{\mathrm{te},i}, \mathbf{I}_{\mathrm{te},j})| &\lesssim \delta^{3-\beta}\left(m + \max_{i=1, \dots, m}\sum_{j=1}^{m} (i \land j)\left|\int_{\mathbb{R}^{d}}\cos(|\fv| (i-j)) \gtemp(\fv)\mathrm{d}\fv \right|\right)\\
&\lesssim \delta^{3-\beta}\left(m + m\sum_{1 \leq |j| \leq m-1} \left|\int_{\mathbb{R}^{d}}\cos(|\fv| j) \gtemp(\fv)\mathrm{d}\fv \right|\right)\\
&\lesssim \delta^{3-\beta}\left(m + m\sum_{1 \leq |j| \leq m-1} \frac{1}{|j|}\right)\\
&= \delta^{3-\beta}\mathcal{O}(m(1+ \log(m)))
\end{aligned}
\end{equation*}
where we have used the fact that by \Cref{result:regularity_g} both $\gtemp$ and its derivative are in $L^{1}(\mathbb{R}^{d})$. The case $i-j=0$, contributes to the term of order $m$.
\qedhere
\end{proof}

\begin{proof}[Proof of \Cref{theorem:clt_temporal_space_fixed}]
By \Cref{result:numerator_temporal_variation_fixed_space} and \Cref{proposition:variance_temporal_second_order_variation_fixed_space}, we have
\begin{equation*}
\frac{\left(\max_{i=1, \dots,m}\sum_{j=1}^{m} |\mathrm{Cov}(\mathbf{I}_{\mathrm{te},i}, \mathbf{I}_{\mathrm{te},j})|\right)^{2}}{\V(\mathbf{V}_{\mathrm{te},m})} = \mathcal{O}\left(\frac{\delta^{6-2\beta}(m+m\log(m))^{2}}{\delta^{6-2\beta}m^{3}} \right) \rightarrow 0, \quad m \rightarrow \infty.
\end{equation*}
Thus, by \citet[Proposition 3.1]{hildebrandtParameterEstimationSPDEs2019} we obtain
\begin{equation*}
\frac{\mathbf{V}_{\mathrm{te},m} - \E[\mathbf{V}_{\mathrm{te},m}]}{\sqrt{\V(\mathbf{V}_{\mathrm{te},m})}} \xrightarrow{d} N(0,1), \quad m \rightarrow \infty.
\end{equation*}
Together with \Cref{proposition:variance_temporal_second_order_variation_fixed_space} the appropriate rescaling from \Cref{proposition:expectation_temp_fixed_space} yields
\begin{equation*}
\sqrt{m}\left(\frac{\delta^{\beta-3}}{m^{2}}\mathbf{V}_{\mathrm{te},m}  - \E\left[\frac{\delta^{\beta-3}}{m^{2}}\mathbf{V}_{\mathrm{te},m} \right]\right) \xrightarrow{d} N\left(0, \frac{\tempconstvart}{\vartheta^{\beta}}\right), \quad m \rightarrow \infty.
\end{equation*}
Finally, we notice that the remainder in \Cref{proposition:expectation_temp_fixed_space} even decays with rate $m^{-1}$ so that upon multiplying with $\sqrt{m}$ it still converges to zero. The resulting  convergence
\begin{equation*}
\sqrt{m}\left(\frac{\delta^{\beta-3}}{m^{2}}\mathbf{V}_{\mathrm{te},m}  -  \frac{\tempconstexpt}{\vartheta^{\beta/2}}\right) \xrightarrow{d} N\left(0, \frac{\tempconstvart}{\vartheta^{\beta}}\right), \quad m \rightarrow \infty,
\end{equation*}
follows by replacing the expectation with its limit. 
\end{proof}

\begin{proof}[Proof of \Cref{corollary:clt_method_of_moments_estimator_temporal_space_fixed}]
Consider the mapping $f(x)=(\tempconstexpt/x)^{2/\beta}$. By applying the delta method to \Cref{theorem:clt_temporal_space_fixed}, we observe
\begin{equation*}
\sqrt{m}(\hat{\vartheta}_{\mathrm{te}, m} - \vartheta) \xrightarrow{d} N\left(0, f'\left( \frac{\tempconstexpt}{\vartheta^{\beta/2}}\right)^{2}\frac{\tempconstvart}{\vartheta^{\beta}} \right), \quad m \rightarrow \infty.
\end{equation*}
The derivative of $f$ is given by $f'(x)=-(2/\beta)f(x)/x$. Clearly $f(\tempconstexpt/\vartheta^{\beta/2})=\vartheta$ and we have
\begin{equation*}
f'\left( \frac{\tempconstexpt}{\vartheta^{\beta/2}}\right)^{2}\frac{\tempconstvart}{\vartheta^{\beta}}  =  \left(\frac{4 \vartheta^{2} \vartheta^{\beta}}{\beta^{2}\tempconstexpt^{2}}\right) \frac{\tempconstvart}{\vartheta^{\beta}} = \vartheta^{2}\frac{4 \tempconstvart}{\beta^{2} \tempconstexpt^{2}},
\end{equation*}
which gives precisely the desired asymptotic variance. \qedhere
\end{proof}
\subsection{Proofs for \Cref{section:space_time_variation}}
\begin{proof}[Proof of \Cref{lemma:covariance_space_time_increments}]
By \Cref{proposition:covariance_representation} the covariance function is given by
\begin{align*}
\E[u(t_i,x_k)u(t_j,x_l)]
&= \frac{1}{ (2\pi)^{d}\vartheta^{\beta/2}} \int_{\mathbb{R}^{d}} \cos\left(\frac{\lambda(k-l)\rho}{\sqrt{\vartheta}}\cdot \fv \right) \Phi_{\mathrm{te}, |\fv|}(\delta i,\delta j) |\fv|^{\beta-d-2}\mathrm{d}\fv\\
&=\frac{\lambda^{2-\beta}}{ (2\pi)^{d}\vartheta} \int_{\mathbb{R}^{d}} \cos\left({(k-l)} \rho\cdot \fv \right) \Phi_{\mathrm{te}, \sqrt{\vartheta}\lambda^{-1}|\fv|}(\delta i,\delta j) |\fv|^{\beta-d-2}\mathrm{d}\fv\\
&=\frac{\delta\lambda^{2-\beta}}{ (2\pi)^{d}\vartheta} \int_{\mathbb{R}^{d}} \cos\left({(k-l)} \rho\cdot \fv \right) \Phi_{\mathrm{te}, \sqrt{\vartheta}\ratio|\fv|}(i, j) |\fv|^{\beta-d-2}\mathrm{d}\fv\\
&= \frac{\delta\lambda^{2-\beta}}{ (2\pi)^{d}\vartheta} \int_{\mathbb{R}^{d}} \Phi_{\mathrm{sp}, \rho \cdot \fv}(k,l) \Phi_{\mathrm{te}, \sqrt{\vartheta}\ratio|\fv|}(i, j) |\fv|^{\beta-d-2}\mathrm{d}\fv
\end{align*}
with $\Phi_{\mathrm{sp}, \rho \cdot \fv}(k,l)= \cos\left( {(k-l)} \rho \cdot \fv \right)$. The covariance between two second-order space-time increments \eqref{eq:space_time_increments} factorises and can be represented as
\begin{equation*}
\mathrm{Cov}(\mathbf{I}_{\mathrm{sp},\mathrm{te},i,k}, \mathbf{I}_{\mathrm{sp},\mathrm{te},j,l})=\frac{\delta\lambda^{2-\beta}}{ (2\pi)^{d}\vartheta} \int_{\mathbb{R}^{d}} \mathcal{I}^{2}[\Phi_{\mathrm{sp}, \rho \cdot \fv}](k,l) \mathcal{I}^{2}[\Phi_{\mathrm{te}, \sqrt{\vartheta}\ratio|\fv|}](i,j) |\fv|^{\beta-d-2}\mathrm{d}\fv.
\end{equation*}
Both of the second-order increment $\mathcal{I}^{2}[\Phi_{\mathrm{sp}, \rho \cdot \fv}](k,l)$ and $\mathcal{I}^{2}[\Phi_{\mathrm{te}, \sqrt{\vartheta}\ratio|\fv|}](i,j)$ have already been computed in the proofs of \Cref{lemma:covariance_second_order_spatial_increments}  and \Cref{lemma:covariance_second_order_temporal_increments}, respectively. In particular, they are given by
\begin{equation*}
\begin{aligned}
\mathcal{I}^{2}[\Phi_{\mathrm{sp}, \rho \cdot \fv}](k,l) &= \mathfrak{I}^{(4)}[\cos(\cdot (\rho^{\top} \fv))](k-l) = 16\sin^{4}(\rho \cdot \fv/2)\cos\left((k-l) {\rho \cdot \fv}\right)\\
\end{aligned}
\end{equation*}
and
\begin{equation*}
\begin{aligned}
\mathcal{I}^{2}[\Phi_{\mathrm{te}, \sqrt{\vartheta}\ratio|\fv|}](i,j)&= \frac{1}{2} \Big(2^{4}(i \land j)\sin^{4}(\sqrt{\vartheta}\ratio |\fv|/2)\cos((i-j)\sqrt{\vartheta}\ratio |\fv|) \\
&\qquad + 2^{3}\sin^{4}(\sqrt{\vartheta}\ratio |\fv|/2)r_{i,j}^{(1)}(\sqrt{\vartheta}\ratio |\fv|)|\fv|^{-1}+r_{i,j}^{(2)}(\sqrt{\vartheta}\ratio |\fv|)+ r_{i,j}^{(3)}(\sqrt{\vartheta}\ratio |\fv|) \Big),
\end{aligned}
\end{equation*}
with $r_{i,j}^{(z)}$, $z \in \{1, 2, 3\}$, defined through \eqref{eq:temporal_remainder_terms}. Thus, in total the covariance satisfies
\begin{equation*}
\begin{aligned}
&\mathrm{Cov}(\mathbf{I}_{\mathrm{sp},\mathrm{te},i,k}, \mathbf{I}_{\mathrm{sp},\mathrm{te},j,l}) \\
&= \delta \lambda^{2-\beta} \Bigg(\frac{128 (i \land j) }{(2\pi)^{d}\vartheta} \int_{\mathbb{R}^{d}} \cos\left((k-l) {\rho \cdot \fv}\right)\sin^{4}(\ratio\sqrt{\vartheta} |\fv|/2)\cos((i-j)\ratio\sqrt{\vartheta} |\fv|)\gspzero(\fv)\mathrm{d}\fv+R_{\mathrm{sp}}(i,j,k,l)\Bigg),
\end{aligned}
\end{equation*}
with $\gspzero$ from \Cref{lemma:covariance_second_order_spatial_increments} and 
\begin{equation}
\label{eq:spatial_remainder_box_inc}
\begin{aligned}
R_{\mathrm{sp}}(i,j,k,l)&\coloneqq R_{\mathrm{sp}}^{(1)}(i,j,k,l)+R_{\mathrm{sp}}^{(2)}(i,j,k,l)+R_{\mathrm{sp}}^{(3)}(i,j,k,l),\\
R_{\mathrm{sp}}^{(1)}(i,j,k,l) &\coloneqq \frac{64}{(2\pi)^{d}\vartheta}\int_{\mathbb{R}^{d}} \cos\left((k-l) {\rho \cdot \fv}\right)\sin^{4}(\sqrt{\vartheta}\ratio |\fv|/2)r_{i,j}^{(1)}(\sqrt{\vartheta}\ratio |\fv|)|\fv|^{-1}g_{\mathrm{sp}}(\fv)\mathrm{d}\fv,\\
R_{\mathrm{sp}}^{(2)}(i,j,k,l) &\coloneqq\frac{8}{(2\pi)^{d}\vartheta}\int_{\mathbb{R}^{d}}\cos\left((k-l) {\rho \cdot \fv}\right)r_{i,j}^{(2)}(\sqrt{\vartheta}\ratio |\fv|)\gspzero(\fv)\mathrm{d}\fv,\\
R_{\mathrm{sp}}^{(3)}(i,j,k,l) &\coloneqq\frac{8}{(2\pi)^{d}\vartheta}\int_{\mathbb{R}^{d}} \cos\left((k-l) {\rho \cdot \fv}\right)r_{i,j}^{(3)}(\sqrt{\vartheta}\ratio |\fv|)\gspzero(\fv)\mathrm{d}\fv.
\end{aligned}
\end{equation}
By applying the change of variables $\fv \mapsto \alpha \sqrt{\vartheta} \fv$ in all of the integrals, we further obtain the second representation for the covariance
\begin{equation*}
\begin{aligned}
&\mathrm{Cov}(\mathbf{I}_{\mathrm{sp},\mathrm{te},i,k}, \mathbf{I}_{\mathrm{sp},\mathrm{te},j,l}) \\
&= \delta^{3-\beta} \Bigg(\frac{128 (i \land j) }{(2\pi)^{d}\vartheta^{\beta/2}} \int_{\mathbb{R}^{d}} \cos\left((\ratio\sqrt{\vartheta})^{-1}(k-l) {\rho \cdot \fv}\right)\sin^{4}((\ratio\sqrt{\vartheta})^{-1} \rho \cdot \fv/2)\cos((i-j) |\fv|)\gtemp(\fv)\mathrm{d}\fv\\
&\qquad+R_{\mathrm{te}}(i,j,k,l)\Bigg),
\end{aligned}
\end{equation*}
with 
\begin{align}
\label{eq:temp_remainder_box_inc}
R_{\mathrm{te}}(i,j,k,l)&\coloneqq R_{\mathrm{te}}^{(1)}(i,j,k,l)+R_{\mathrm{te}}^{(2)}(i,j,k,l)+R_{\mathrm{te}}^{(3)}(i,j,k,l),\\
R_{\mathrm{te}}^{(1)}(i,j,k,l)&\coloneqq \frac{64}{(2\pi)^{d}\vartheta^{\beta/2}} \left(\sqrt{\vartheta} \alpha\right)\int_{\mathbb{R}^{d}} \cos\left((\sqrt{\vartheta}\ratio)^{-1}(k-l) {\rho \cdot \fv}\right)\sin^{4}((\sqrt{\vartheta}\ratio)^{-1}\rho \cdot \fv/2),\nonumber\\
&\hspace{35mm}\cdot r_{i,j}^{(1)}(|\fv|) (|\fv|^{-1}g_{\mathrm{te}}(\fv))\mathrm{d}\fv,\nonumber\\
R_{\mathrm{te}}^{(2)}(i,j,k,l)&\coloneqq\frac{8}{(2\pi)^{d}\vartheta^{\beta/2}}\int_{\mathbb{R}^{d}}\cos\left((\ratio\sqrt{\vartheta})^{-1}(k-l) {\rho \cdot \fv}\right)\sin^{4}((\sqrt{\vartheta}\ratio)^{-1}\rho \cdot \fv/2)r_{i,j}^{(2)}( |\fv|)|\fv|^{\beta-d-2}\mathrm{d}\fv,\nonumber\\
R_{\mathrm{te}}^{(3)}(i,j,k,l)&\coloneqq\frac{8}{(2\pi)^{d}\vartheta^{\beta/2}}\int_{\mathbb{R}^{d}} \cos\left((\ratio\sqrt{\vartheta})^{-1}(k-l) {\rho \cdot \fv}\right)\sin^{4}((\sqrt{\vartheta}\ratio)^{-1}\rho \cdot \fv/2)r_{i,j}^{(3)}(|\fv|)|\fv|^{\beta-d-2}\mathrm{d}\fv.\nonumber
\end{align}
Notice that $r_{i,j}^{(z)}(\xi)$, $\xi\in\mathbb{R},z\in\{1,2,3\},$ is uniformly bounded in $\xi$ independently of $i,j,z$. Thus, the uniform boundedness of $R_{\mathrm{sp}}(i,j,k,l)$, i.e.
$$\sup_{i,j\leq m}\sup_{k,l\leq n}|R_{\mathrm{sp}}(i,j,k,l)|\lesssim c$$
for some universal constant $c$, follows immediately since $\gspzero \in L^{1}(\mathbb{R}^{d})$ and $|\cdot|^{-1}\gspzero \in L^{1}(\mathbb{R}^{d})$. The boundedness argument for the remainders $R_{\mathrm{te}}(i,j,k,l)$ immediately reduces to \Cref{lemma:remainder_bounds_temp}, if we additionally assume that $\alpha$ remains bounded, which holds in the temporal regime $\alpha\rightarrow0$.
\end{proof}

\begin{proof}[Proof of \Cref{proposition:box_expectation}]
Throughout this proof, we are going to use the following Riemann-Lebesgue argument serveral times. Observe that $\sin^4(x)=3/8-\cos(2x)/2+\cos(4x)/8$. Thus, since $\partial_{i}\partial_j\gspzero,\partial_{i}\partial_j\gtemp \in L^{1}(\mathbb{R}^{d})$ by \Cref{result:regularity_g} for all $i,j=1, \dots, d$, we obtain with the Riemann-Lebesgue lemma the asymptotic expansions
\begin{align}
    \int_{\mathbb{R}^{d}} \sin^{4}(\ratio\sqrt{\vartheta} |\fv|/2)\gspzero(\fv)\mathrm{d}\fv &= \frac{3}{8}\int_{\mathbb{R}^{d}}\gspzero(\fv) \mathrm{d}\fv + \smallo(\ratio^{-2}), \quad \alpha \rightarrow \infty \label{eq:rl_alpha_inf}\\
    \int_{\mathbb{R}^{d}}\sin^{4}((\ratio\sqrt{\vartheta})^{-1} \rho \cdot \fv/2)\gtemp(\fv)\mathrm{d}\fv&=\frac{3}{8}\int_{\mathbb{R}^d}\gtemp(\fv)\mathrm{d}\fv+\smallo (\ratio^2), \quad \alpha \rightarrow 0,
    \label{eq:rl_alpha_zero}
\end{align}
see \citet[Table 3.1]{grafakosClassicalFourierAnalysis2014}. Notice further that as described in  \Cref{lemma:covariance_space_time_increments} the remainder terms $R_{\mathrm{sp}}(i,j,k,l)$ and $R_{\mathrm{te}}(i,j,k,l)$ defined in \eqref{eq:spatial_remainder_box_inc} and \eqref{eq:temp_remainder_box_inc} satisfy the rough upper bounds
\begin{align}
    &\sup_{i,j\leq m}\sup_{k,l\leq n}|R_{\mathrm{sp}}(i,j,k,l)| \lesssim c;\label{eq:uniform_remainder_bound_sp}\\
    &\sup_{i,j\leq m}\sup_{k,l\leq n}|R_{\mathrm{te}}(i,j,k,l)| 
    \lesssim c, \quad \alpha \lesssim1,\label{eq:uniform_remainder_bound_temp}
\end{align}
for some universal constant $c$. Recall the constants $\boxconstexpsp,\boxconstexptemp$ from \Cref{tab:asymptotic_constants}.
\begin{step}[$\ratio \rightarrow \infty$] 
By \Cref{lemma:covariance_space_time_increments}, we obtain the representation
    \begin{equation*}
    \begin{aligned}
        \frac{\delta^{-1}\lambda^{\beta-2}}{nm^{2}}\E[\mathbf{V}_{\mathrm{sp}, \mathrm{te}}]&=  \frac{1}{nm^{2}}\sum_{k=1}^{n}\sum_{i=1}^{m}  \Bigg(\frac{128 i }{(2\pi)^{d}\vartheta} \int_{\mathbb{R}^{d}} \sin^{4}(\ratio\sqrt{\vartheta} |\fv|/2)\gspzero(\fv)\mathrm{d}\fv+R_{\mathrm{sp}}(i,i,k,k)\Bigg)\\
        &=\frac{64 }{(2\pi)^{d}\vartheta}  \Bigg( \int_{\mathbb{R}^{d}} \sin^{4}(\ratio\sqrt{\vartheta} |\fv|/2)\gspzero(\fv)\mathrm{d}\fv \Bigg)\\ &\qquad+ \frac{64}{m(2\pi)^{d}\vartheta}  \Bigg( \int_{\mathbb{R}^{d}} \sin^{4}(\ratio\sqrt{\vartheta} |\fv|/2)\gspzero(\fv)\mathrm{d}\fv \Bigg) + \frac{1}{nm^{2}}\sum_{k=1}^{n}\sum_{i=1}^{m} R_{\mathrm{sp}}(i,i,k,k)\\
        &=\vartheta^{-1}\boxconstexpsp + \smallo(\ratio^{-2}) + \mathcal{O}(m^{-1}),
    \end{aligned}
    \end{equation*}
where we have used that $m^{-2}\sum_{i=1}^{m}i=1/2 -1/(2m)$, \eqref{eq:rl_alpha_inf} and  \eqref{eq:uniform_remainder_bound_sp}. 
\end{step}
\begin{step}[$\ratio \rightarrow 0$]
For the regime $\ratio \rightarrow 0$, we will exploit the second representation from \Cref{lemma:covariance_space_time_increments} yielding
\begin{equation*}
\begin{aligned}
\frac{\delta^{\beta-3}}{nm^{2}}\E[\mathbf{V}_{\mathrm{sp}, \mathrm{te}}]&= \frac{1}{nm^{2}}\sum_{k=1}^{n}\sum_{i=1}^{m} \Bigg(\frac{128 i}{(2\pi)^{d}\vartheta^{\beta/2}} \int_{\mathbb{R}^{d}}\sin^{4}((\ratio\sqrt{\vartheta})^{-1} \rho \cdot \fv/2)\gtemp(\fv)\mathrm{d}\fv+R_{\mathrm{te}}(i,i,k,k)\Bigg)\\
&= \vartheta^{-\beta/2}\boxconstexptemp + \smallo(\ratio^{2})+\mathcal{O}(m^{-1})
\end{aligned}
\end{equation*}
with \eqref{eq:rl_alpha_zero} and \eqref{eq:uniform_remainder_bound_temp}.\qedhere
\end{step}
\end{proof}

\begin{proof}[Proof of \Cref{proposition:variance_space_time_variation}]
Recall both $\boxconstvarsp,\boxconstvartemp$ from \Cref{tab:asymptotic_constants}.
\begin{case}[$\ratio\rightarrow\infty$]
Observe by Wicks' theorem and \Cref{lemma:covariance_space_time_increments} that
\begin{align}
&\frac{\delta^{-2}\lambda^{2\beta-4}}{nm^{3}}\V(\mathbf{V}_{\mathrm{sp}, \mathrm{te}}) \nonumber\\
&= \frac{2\delta^{-2}\lambda^{2\beta-4}}{nm^{3}}\sum_{i,j=1}^{m}\sum_{l,k=1}^{n} \mathrm{Cov}(\mathbf{I}_{\mathrm{sp},\mathrm{te},i,k}, \mathbf{I}_{\mathrm{sp},\mathrm{te},j,l})^{2}\nonumber\\
&=\frac{2}{nm^{3}}\sum_{i,j=1}^{m}\sum_{l,k=1}^{n}\Bigg(\frac{128 (i \land j)}{ (2\pi)^{d}\vartheta} \int_{\mathbb{R}^{d}}  \cos\left((k-l) {\rho \cdot \fv}\right) \sin^{4}(\ratio \sqrt{\vartheta}|\fv|/2)\cos(\ratio \sqrt{\vartheta} (i-j)|\fv|) \gspzero(\fv)\mathrm{d}\fv\nonumber\\
&\qquad+ R_{\mathrm{sp}}(i,j,k,l)\Bigg)^{2}.\label{eq:main_boxincr_var_1}
\end{align}
\begin{step}[Determining the limit]
    Using our representations \eqref{eq:introducing_the_fejér_kernel} and \eqref{eq:fejer_time_kernel} and utilizing the factorisations in \eqref{eq:fejer_factorisation} and \eqref{eq:time_fejer_factorisation}, we can rewrite the main expression in \eqref{eq:main_boxincr_var_1} as
    \begin{align}
        &2 \frac{1}{m^{3}} \sum_{i,j=1}^{m} \frac{1}{n}\sum_{l,k=1}^{n}\Bigg(\frac{128 (i \land j)}{ (2\pi)^{d}\vartheta} \int_{\mathbb{R}^{d}}  \cos\left((k-l) {\rho \cdot \fv}\right) \cos(\ratio \sqrt{\vartheta} (i-j)|\fv|) \sin^{4}(\ratio \sqrt{\vartheta}|\fv|/2)\gspzero(\fv)\mathrm{d}\fv \Bigg)^{2}\label{eq:varlimit_sp_main}\\
        &=\frac{2^{15}}{(2\pi)^{2d}\vartheta^{2}} \sum_{|j|\leq m-1}\sum_{|k|\leq n-1}w_{\mathrm{te},j}w_{\mathrm{sp},k}\left(\int_{\mathbb{R}^{d}}\cos(k\rho\cdot\fv)\cos(j\ratio\sqrt{\vartheta}|\fv|)\sin^{4}(\ratio \sqrt{\vartheta}|\fv|/2)\gspzero(\fv)\mathrm{d}\fv\right)^2\nonumber
    \end{align}
As $\sin^4(x)=3/8-\cos(2x)/2+\cos(4x)/8$, we obtain 
\begin{equation*}
    \begin{aligned}
        &\int_{\mathbb{R}^{d}}\cos(k\rho\cdot\fv)\cos(j\ratio\sqrt{\vartheta}|\fv|)\sin^{4}(\ratio \sqrt{\vartheta}|\fv|/2)\gspzero(\fv)\mathrm{d}\fv\\
        &=\sum_{z\in\{0,\pm 1,\pm 2\}}\frac{(3-|z|)(-1)^{|z|}}{|z|!8}\int_{\mathbb{R}^{d}}\cos(k\rho\cdot\fv)\cos((j+z)\ratio\sqrt{\vartheta}|\fv|)\gspzero(\fv)\mathrm{d}\fv\\
        &=\sum_{z\in\{0,\pm 1,\pm 2\}}c_z C_{\mathrm{sp},k,j,z}
    \end{aligned}
\end{equation*}
with 
\begin{equation}
    \label{eq:coeff_gamma}
    \begin{aligned}
        C_{\mathrm{sp},k,j,z}&\coloneqq\int_{\mathbb{R}^{d}}\cos(k\rho\cdot\fv)\cos((j+z)\ratio\sqrt{\vartheta}|\fv|)\sin^4(\rho\cdot\fv/2)|\fv|^{\beta-d-2}\mathrm{d}\fv,\\
        c_z&\coloneqq \frac{(3-|z|)(-1)^{|z|}}{|z|!8}.
    \end{aligned}
\end{equation}
Note that $C_{\mathrm{sp},k,j,z}=\mathcal{F}_c(\gspzero)(k\rho)$ if and only if $j+z=0$, i.e.\ $z=-j$, which is only possible for $|j|\leq2$ and exactly one $z\in\{0,\pm1,\pm2\}$. 
Hence, we can rewrite \eqref{eq:varlimit_sp_main} as 
\begin{align*}
        &\frac{2^{15}}{(2\pi)^{2d}\vartheta^{2}} \sum_{|j|\leq m-1}\sum_{|k|\leq n-1}w_{\mathrm{te},j}w_{\mathrm{sp},k}\left(\sum_{z\in\{0,\pm 1,\pm 2\}}c_z C_{\mathrm{sp},k,j,z}\right)^2\\
        &=\frac{2^{15}}{(2\pi)^{2d}\vartheta^{2}} \sum_{|j|\leq m-1}\sum_{z_1,z_2\in\{0,\pm 1,\pm 2\}}c_{z_1}c_{z_2}w_{\mathrm{te},j}\sum_{|k|\leq n-1}w_{\mathrm{sp},k} C_{\mathrm{sp},k,j,z_1}C_{\mathrm{sp},k,j,z_2}\\
        &=\frac{2^{15}}{(2\pi)^{2d}\vartheta^{2}} \sum_{|j|\leq m-1}\sum_{\substack{z_1,z_2\in\{0,\pm 1,\pm 2\}\\
        z_1=z_2=-j}}c_{z_1}c_{z_2}w_{\mathrm{te},j}\sum_{|k|\leq n-1}w_{\mathrm{sp},k} C_{\mathrm{sp},k,j,z_1}C_{\mathrm{sp},k,j,z_2}\\
        &\qquad+\frac{2^{15}}{(2\pi)^{2d}\vartheta^{2}} \sum_{|j|\leq m-1}\sum_{\substack{z_1,z_2\in\{0,\pm 1,\pm 2\}\\
        z_1\neq-j\text{ or }z_2\neq-j}}c_{z_1}c_{z_2}w_{\mathrm{te},j}\sum_{|k|\leq n-1}w_{\mathrm{sp},k} C_{\mathrm{sp},k,j,z_1}C_{\mathrm{sp},k,j,z_2}\\
        &=\frac{2^{15}}{(2\pi)^{2d}\vartheta^{2}} \sum_{z\in\{0,\pm 1,\pm 2\}}c_{z}^2w_{\mathrm{te},j}\sum_{|k|\leq n-1}w_{\mathrm{sp},k} \left(\mathcal{F}_c(\gspzero)(k\rho)\right)^2\\
        &\qquad+\frac{2^{15}}{(2\pi)^{2d}\vartheta^{2}} \sum_{|j|\leq m-1}\sum_{\substack{z_1,z_2\in\{0,\pm 1,\pm 2\}\\
        z_1\neq-j\text{ or }z_2\neq-j}}c_{z_1}c_{z_2}w_{\mathrm{te},j}\sum_{|k|\leq n-1}w_{\mathrm{sp},k} C_{\mathrm{sp},k,j,z_1}C_{\mathrm{sp},k,j,z_2}\\
        &\rightarrow\frac{2^{15}}{3(2\pi)^{2d}\vartheta^{2}} \sum_{z\in\{0,\pm 1,\pm 2\}}c_{z}^2\sum_{k\in\mathbb{Z}}\left(\mathcal{F}_c(\gspzero)(k\rho)\right)^2\\
        &=\frac{2^{8}35}{3(2\pi)^{2d}\vartheta^{2}}\sum_{k\in\mathbb{Z}}\left(\mathcal{F}_c(\gspzero)(k\rho)\right)^2=\frac{C_{\mathrm{box,sp,\V}}}{\vartheta^2},
\end{align*}
where we have used $ \sum_{z\in\{0,\pm 1,\pm 2\}}c_{z}^2=35/128$ in the last line. The convergence in the penultimate line follows from \Cref{lemma:fourier_representation_time_fejer_kernel}, the proof of \Cref{proposition:asymptotic_scaling_for_the_variance}, i.e.\ $w_{\mathrm{te},j}\rightarrow1/3$ and $\sum_{|k|\leq n-1}w_{\mathrm{sp},k}\left(\mathcal{F}_c(\gspzero)(k\rho)\right)^2\rightarrow\sum_{k\in\mathbb{Z}}\left(\mathcal{F}_c(\gspzero)(k\rho)\right)^2$, and the following bound
\begin{align*}
        &\left|\frac{2^{15}}{(2\pi)^{2d}\vartheta^{2}} \sum_{|j|\leq m-1}\sum_{\substack{z_1,z_2\in\{0,\pm 1,\pm 2\}\\
        z_1\neq-j\text{ or }z_2\neq-j}}c_{z_1}c_{z_2}w_{\mathrm{te},j}\sum_{|k|\leq n-1}w_{\mathrm{sp},k} C_{\mathrm{sp},k,j,z_1}C_{\mathrm{sp},k,j,z_2}\right|\\
        &\lesssim \sum_{|j|\leq m-1}\sum_{\substack{z_1,z_2\in\{0,\pm 1,\pm 2\}\\
        z_1\neq-j\text{ or }z_2\neq-j}}\sum_{|k|\leq n-1}\left| C_{\mathrm{sp},k,j,z_1}C_{\mathrm{sp},k,j,z_2}\right|\\
        &\leq\sum_{|j|\leq m-1}\sum_{\substack{z_1,z_2\in\{0,\pm 1,\pm 2\}\\
        z_1\neq-j\text{ or }z_2\neq-j}}\left(\sum_{|k|\leq n-1}C_{\mathrm{sp},k,j,z_1}^2\right)^{1/2}\left(\sum_{|k|\leq n-1}C_{\mathrm{sp},k,j,z_2}^2\right)^{1/2}\\
        &=\sum_{|j|\leq m-1}\sum_{\substack{z_1,z_2\in\{0,\pm 1,\pm 2\}\\
        z_1,z_2\neq-j}}\left(\sum_{|k|\leq n-1}C_{\mathrm{sp},k,j,z_1}^2\right)^{1/2}\left(\sum_{|k|\leq n-1}C_{\mathrm{sp},k,j,z_2}^2\right)^{1/2}\\
        &\qquad +2\sum_{|j|\leq 2}\sum_{\substack{z\in\{0,\pm 1,\pm 2\}\\
        z\neq-j}}\left(\sum_{|k|\leq n-1}C_{\mathrm{sp},k,j,z}^2\right)^{1/2}\left(\sum_{|k|\leq n-1}\left(\mathcal{F}_c(\gspzero)(k\rho)\right)^2\right)^{1/2}\\
        &\lesssim \begin{cases}
            mn\ratio^{-4}+n\ratio^{-4},\quad&d=1,\\
            \alpha^{-1}\log(m)+\alpha^{-1/2},\quad &d\geq2,
        \end{cases}\\
        &\rightarrow0
\end{align*}
due to \Cref{lemma:C_k_j_f_a_j}, \Cref{lemma:d_1_bound_C_j_k} together with \Cref{assumption:bias_convergence} and \Cref{assumption:helper}.
\end{step}
\begin{step}[Controlling the remainder]
For the remainder term in \eqref{eq:main_boxincr_var_1}, we observe by \Cref{lemma:covariance_space_time_increments}
\begin{equation*}
\begin{aligned}
&\frac{2}{nm^{3}}\sum_{i,j=1}^{m}\sum_{l,k=1}^{n}R_{\mathrm{sp}}(i,j,k,l)^{2}= \frac{1}{nm^{3}}\mathcal{O}(m^{2}n^{2})=\mathcal{O}\left(\frac{n}{m}\right) \rightarrow 0,
\end{aligned}
\end{equation*}
where the convergence follows immediately from \eqref{eq:bias_assumption_n_frac_m} in \Cref{assumption:bias_convergence}. As discussed in \Cref{proposition:variance_temporal_second_order_variation_fixed_space} the cross term remains of lower order by the Cauchy-Schwarz inequality. 
\end{step}
\end{case}
\begin{case}[$\ratio\rightarrow0$]
    The proof for $\ratio\rightarrow0$ is identical. We obtain by \Cref{lemma:covariance_space_time_increments}, \Cref{lemma:C_k_j_f_a_j} and \Cref{lemma:d_1_bound_C_j_k} the convergence
    \begin{equation*}
        \begin{aligned}
        \frac{\delta^{2\beta-6}}{nm^{3}}\V(\mathbf{V}_{\mathrm{sp}, \mathrm{te}})&=\frac{2\delta^{2\beta-6}}{nm^{3}}\sum_{i,j=1}^m\sum_{l,k=1}^n\mathrm{Cov}(\mathbf{I}_{\mathrm{sp},\mathrm{te},i,k}, \mathbf{I}_{\mathrm{sp},\mathrm{te},j,l})^2\\
        &= \frac{2^{15}}{(2\pi)^{2d}\vartheta^{\beta}} \sum_{|k|\leq n-1}w_{\mathrm{sp},k}\sum_{|j|\leq m-1}w_{\mathrm{te},j}\left(\sum_{z\in\{0,\pm1,\pm2\}}c_z C_{\mathrm{te},k,j,z}\right)^2+\smallo(1)\\
        &\rightarrow \frac{2^{8}35}{3(2\pi)^{2d}\vartheta^{\beta}}\sum_{j\in\mathbb{Z}}\left(\mathcal{F}_c^+(\gtemp)(j)\right)^2=\frac{C_{\mathrm{box,temp,\V}}}{\vartheta^\beta}
        \end{aligned}
    \end{equation*}
    with $c_z$ from \eqref{eq:coeff_gamma} and
    \begin{equation*}
        C_{\mathrm{te},k,j,z}\coloneqq\int_{\mathbb{R}^{d}} \cos(j  |\fv|)\cos((k+z)(\ratio \sqrt{\vartheta})^{-1} \rho \cdot \fv)\sin^{4}(|\fv|/2)|\fv|^{\beta-d-2}\mathrm{d}\fv.\tag*{\qedhere}
    \end{equation*}
\end{case}
\end{proof}

\begin{lemma}
  \label[lemma]{lemma:C_k_j_f_a_j}
  Let $k,j \in \mathbb{Z}$,  $\ratio>0$ and $z\in\{-2,-1,0,1,2\}$. Consider the coefficients
  \begin{equation*}
  \begin{aligned}
  C_{\mathrm{sp},k,j,z} &= \int_{\mathbb{R}^{d}} \cos(k \rho \cdot \fv)\cos((j+z) \ratio \sqrt{\vartheta} |\fv|)\sin^{4}(\rho \cdot \fv/2)|\fv|^{\beta-d-2}\mathrm{d}\fv,\\
  C_{\mathrm{te},k,j,z} &=\int_{\mathbb{R}^{d}} \cos(j  |\fv|)\cos((k+z)(\ratio \sqrt{\vartheta})^{-1} \rho \cdot \fv)\sin^{4}(|\fv|/2)|\fv|^{\beta-d-2}\mathrm{d}\fv.\\
  \end{aligned}
  \end{equation*}
  Assume that $d\geq 2$. If $|j+z|\geq 1,$ then
  \begin{equation}
  \label{eq:spatial_sum_decay}
  \sum_{k \in \mathbb{Z} } C_{\mathrm{sp},k,j,z}^{2}
  \lesssim \frac{1}{\ratio|j+z|},
  \end{equation}
  and if $|k+z| \geq 1$, then
  \begin{equation}
  \label{eq:temporal_sum_decay}
  \sum_{j \in \mathbb{Z} } C_{\mathrm{te},k,j,z}^{2}
  \lesssim \frac{\ratio}{|k+z|}.
  \end{equation}
  \end{lemma}
\begin{proof}
    \begin{step}[Proving \eqref{eq:spatial_sum_decay}]
        $C_{\mathrm{sp},k,j,z}$ depends on $\fv$ only through $|\fv|$ and $\rho \cdot \fv$. By rotational invariance, we can rotate the coordinate system and assume that $\rho=e_{1}$, yielding
          \begin{equation}
          \label{eq:periodisation_C_k_j}
          \begin{aligned}
          C_{\mathrm{sp}, k,j,z}&= \int_{\mathbb{R}^{d}} \cos(k \rho \cdot \fv)\cos((j+z) \ratio \sqrt{\vartheta} |\fv|)\sin^{4}(\rho \cdot \fv/2)|\fv|^{\beta-d-2}\mathrm{d}\fv,\\
          &= \int_{\mathbb{R}^{d}} \cos(k \fv_1)\cos((j+z) \ratio \sqrt{\vartheta} |\fv|)\sin^{4}(\fv_1/2)|\fv|^{\beta-d-2}\mathrm{d}\fv\\
          &= \int_{\mathbb{R}} \cos(k s)\sin^{4}(s/2) H_{\mathrm{sp},\ratio,j}(s) \mathrm{d}s
          \end{aligned}
          \end{equation}
          with
          \begin{equation*}
          \begin{aligned}
          H_{\mathrm{sp},\ratio,j}(s) &\coloneqq\int_{\mathbb{R}^{d-1}}\cos\left((j+z)\ratio \sqrt{\vartheta}\sqrt{s^{2}+|\fv|^{2}}\right)\left(s^{2}+|\fv|^{2}\right)^{\frac{\beta-d-2}{2}}\mathrm{d}\fv\\
          &=\sarea[d-2]\int_{0}^{\infty} \cos\left((j+z)\ratio \sqrt{\vartheta}\sqrt{s^{2}+r^{2}}\right)\left(s^{2}+r^{2}\right)^{\frac{\beta-d-2}{2}}r^{d-2}\mathrm{d}r\\
          &=\sarea[d-2]\int_{|s|}^{\infty} \cos((j+z)\ratio \sqrt{\vartheta} v)v^{\beta-d-1}(v^{2}-s^{2})^{\frac{d-3}{2}}\mathrm{d}v,
          \end{aligned}
          \end{equation*}
          where we have rewritten the integral in polar coordinates and then applied the transformation $v(r)=\sqrt{s^{2}+r^{2}}$. Next, notice that $\cos(ks)$ and $\sin^{4}(s/2)$ are $2\pi$-periodic. Thus, splitting $\mathbb{R}$ into intervals of length $2\pi$ we can rewrite
          \begin{equation*}
          \begin{aligned}
          C_{\mathrm{sp},k,j,z}&=\sum_{l \in \mathbb{Z}} \int_{-\pi}^{\pi}\cos(k(s+2\pi l))\sin^{4}((s+2\pi l)/2)H_{\mathrm{sp},\ratio,j}(s+2\pi l) \mathrm{d}s\\
          &=\int_{-\pi}^{\pi}\cos(ks)\sin^{4}(s/2)\sum_{l \in \mathbb{Z}}H_{\mathrm{sp},\ratio,j}(s+2\pi l)\mathrm{d}s\\
          &=\int_{-\pi}^{\pi}\cos(ks)f_{\mathrm{sp},\ratio,j}(s) \mathrm{d}s,
          \end{aligned}
          \end{equation*}
          with $f_{\mathrm{sp},\ratio,j}(s)=\sin^{4}(s/2)\sum_{l \in \mathbb{Z}}H_{\mathrm{sp},\ratio,j}(s+2\pi l)$. Now provided that $f_{\mathrm{sp},\ratio,j}$ is well-defined and an element of $L^{2}((-\pi,\pi))$, we obtain
          \begin{equation}
          \label{eq:plancherel_estimate}
          \sum_{k \in \mathbb{Z}} C_{\mathrm{sp},k,j,z}^{2} = \sum_{k \in \mathbb{Z}}\left(\int_{-\pi}^{\pi}\cos(ks)f_{\mathrm{sp},\ratio,j}(s) \mathrm{d}s\right)^{2}=2\pi\Vert f_{\mathrm{sp},\ratio,j} \Vert^{2}_{L^{2}((-\pi,\pi))}.
          \end{equation}
          Next, we investigate the mapping $H_{\mathrm{sp},\ratio,j}$. Splitting the integral at $\varepsilon>0$ we can rewrite $H_{\mathrm{sp},\ratio,j}$ as
          \begin{equation}
          \label{eq: integral_splitting_H}
          \begin{aligned}
             H_{\mathrm{sp},\ratio,j}(s)&=\int_{|s|}^{|s|+\varepsilon} \cos((j + z)\ratio \sqrt{\vartheta} v)v^{\beta-d-1}(v^{2}-s^{2})^{\frac{d-3}{2}}\mathrm{d}v\\
             &\qquad+\int_{|s|+\varepsilon}^{\infty} \cos((j + z)\ratio \sqrt{\vartheta} v)v^{\beta-d-1}(v^{2}-s^{2})^{\frac{d-3}{2}}\mathrm{d}v.
          \end{aligned}
          \end{equation}
          Note that for $v \geq |s|$, we have
          \begin{equation*}
              \begin{aligned}
                  v^{\beta-d-1}(v^2-s^2)^{(d-3)/2}&=v^{\beta-d-1}(v^2-s^2)^{-1/2}(v^2-s^2)^{(d-2)/2}\\
                  &\leq v^{\beta-d-1}v^{d-2}((v-|s|)(v+|s|))^{-1/2}\\
                  &\leq v^{\beta-3}|s|^{-1/2}(v-|s|)^{-1/2}\\
                  &\leq |s|^{\beta-7/2}(v-|s|)^{-1/2}\mathbbm{1}_{[|s|,|s|+\varepsilon]}(v)+v^{\beta-7/2}\varepsilon^{-1/2}\mathbbm{1}_{(|s|+\varepsilon,\infty)}(v).
              \end{aligned}
          \end{equation*}          
          We will bound the first integral directly and use integration by parts for the second one. In the special case of $d=2$ we obtain
          \begin{equation*}
              \begin{aligned}
                  &\int_{|s|+\varepsilon}^{\infty} \cos((j + z)\ratio \sqrt{\vartheta} v)v^{\beta-3}(v^{2}-s^{2})^{-1/2}\mathrm{d}v\\
                  &=\left[\frac{\sin((j+z)\ratio\sqrt{\vartheta}v)}{(j+z)\ratio\sqrt{\vartheta}}v^{\beta-3}(v^{2}-s^{2})^{-1/2}\right]_{|s|+\varepsilon}^\infty-\int_{|s|+\varepsilon}^\infty \frac{\sin((j+z)\ratio\sqrt{\vartheta}v)}{(j+z)\ratio\sqrt{\vartheta}}\frac{\mathrm{d}}{\mathrm{d}v}\left(v^{\beta-3}(v^{2}-s^{2})^{-1/2}\right)\mathrm{d}v\\
                  &\leq\frac{|s|^{\beta-7/2}\varepsilon^{-1/2}}{|j+z|\ratio\sqrt{\vartheta}}+\int_{|s|+\varepsilon}^\infty \frac{1}{|j+z|\ratio\sqrt{\vartheta}}\left|\frac{\mathrm{d}}{\mathrm{d}v}\left(v^{\beta-3}(v^{2}-s^{2})^{-1/2}\right)\right|\mathrm{d}v\\
                  &=\frac{|s|^{\beta-7/2}\varepsilon^{-1/2}}{|j+z|\ratio\sqrt{\vartheta}}-\frac{1}{|j+z|\ratio\sqrt{\vartheta}}\int_{|s|+\varepsilon}^\infty \frac{\mathrm{d}}{\mathrm{d}v}\left(v^{\beta-3}(v^{2}-s^{2})^{-1/2}\right)\mathrm{d}v\\
                  &\leq\frac{2|s|^{\beta-7/2}\varepsilon^{-1/2}}{|j+z|\ratio\sqrt{\vartheta}},
              \end{aligned}
          \end{equation*}
          whereas the penultimate equality fails in $d\geq3$. In that case we bound the derivative of $v^{\beta-d-1}(v^2-s^2)^{(d-3)/2}$ by the triangle inequality and consider each term separately.
          In total,
          \begin{equation*}
              \begin{aligned}
                  H_{\mathrm{sp},\ratio,j}(s)&\lesssim|s|^{\beta-7/2}\left(\varepsilon^{1/2}+\frac{1}{\ratio|j+z|\varepsilon^{1/2}}\right)\\
                  &\leq |s|^{\beta-7/2}\varepsilon^{1/2}+\frac{1}{\ratio|j+z|\varepsilon^{1/2}}\\
                  &=|s|^{\beta-7/2}\frac{1}{\ratio^{1/2}|j+z|^{1/2}}
              \end{aligned}
          \end{equation*}
          for the choice $\varepsilon=\ratio^{-1}|j+z|^{-1}$. Plugging the last display into \eqref{eq:plancherel_estimate}, it follows for the norm
          \begin{equation*}
          \begin{aligned}
              \Vert f_{\mathrm{sp},\ratio,j} \Vert_{L^{2}((-\pi, \pi))}^{2} \lesssim &\frac{1}{\ratio|j+z|}\int_{-\pi}^{\pi} \sin^{8}(s/2) \left(\sum_{l \in \mathbb{Z}} |s + 2\pi l|^{\beta-7/2}\right)^{2}\mathrm{d}s.\\
              &\lesssim\frac{1}{\ratio|j+z|}\int_{-\pi}^{\pi} \sin^{8}(s/2) \left(|s|^{\beta-7/2}+\sum_{l \in \mathbb{Z},|l|\geq1} |l|^{\beta-7/2}\right)^2\mathrm{d}s\\
                &\lesssim\frac{1}{\ratio|j+z|}\int_{-\pi}^{\pi} \sin^{8}(s/2) (|s|^{2\beta-7}+1)\mathrm{d}s\\
                &\lesssim \frac{1}{\ratio|j+z|},
          \end{aligned}
          \end{equation*}
          proving the assertion.
    \end{step}
    \begin{step}[Proving \eqref{eq:temporal_sum_decay}]
      The structure of the analysis of $C_{\mathrm{te},k,j,z}$ is essentially similar. In the analysis of $C_{\mathrm{sp},k,j,z}$, we isolated the direction $\rho \cdot \fv= \fv_1$ and captured all dependencies on $j$ as well as $\ratio$ within the mapping $f_{\mathrm{sp},a,j}$. Likewise for $C_{\mathrm{te},k,j,z}$, we first transform the integral to polar coordinates, i.e.\
      \begin{equation*}
      \begin{aligned}
      C_{\mathrm{te},k,j,z} &=\int_{\mathbb{R}^{d}} \cos(j  |\fv|)\cos((k+z)(\ratio \sqrt{\vartheta})^{-1} \rho \cdot \fv)\sin^{4}(|\fv|/2)|\fv|^{\beta-d-2}\mathrm{d}\fv\\
      &=\int_{0}^{\infty}\cos(jr)\sin^{4}(r/2)r^{\beta-3}\left(\int_{\mathbb{S}^{d-1}}\cos((k+z)(\ratio \sqrt{\vartheta})^{-1} \rho \cdot (r\fv)) \mathrm{d}\sigma(\fv) \right)\mathrm{d}r\\
      &=
      \int_{0}^{\infty}\cos(jr)\sin^{4}(r/2)r^{\beta-3} H_{\mathrm{te},\ratio,k}(r)\mathrm{d}r
      \end{aligned}
      \end{equation*}
      with 
      \begin{equation}
      \label{eq:tilde_H_a_k}
      H_{\mathrm{te}, \ratio, k}(r) = \int_{\mathbb{S}^{d-1}}\cos((k+z)(\ratio \sqrt{\vartheta})^{-1} \rho \cdot (r\fv))\mathrm{d}\sigma(\fv) .
      \end{equation}
      Now periodisation yields 
      \begin{equation*}
      \sum_{j \in \mathbb{Z}}C_{\mathrm{te},k,j,z}^{2} \lesssim \Vert {f}_{\mathrm{te}, \ratio, k} \Vert_{L^{2}([-\pi, \pi])}^{2},
      \end{equation*}
      with
      \begin{equation*}
      {f}_{\mathrm{te},\ratio,k}(r)= \mathbbm{1}_{(0,\infty)}(r) \sin^{4}(r/2)\sum_{l \in \mathbb{Z}} {H}_{\mathrm{te},\ratio,k}(r+2\pi l)|r + 2\pi l|^{\beta-3}.
      \end{equation*}
      By the classical Funk-Hecke formula and substitution, we can rewrite \eqref{eq:tilde_H_a_k} as
      \begin{equation*}
      \begin{aligned}
          H_{\mathrm{te},\ratio,k}(r) &=2\sarea[d-2]r^{-1}\int_{0}^{r}\cos((\ratio \sqrt{\vartheta})^{-1}(k+z)s)((r^2-s^{2})/r^2)^{(d-3)/2}\mathrm{d}s\\
          &=2\sarea[d-2] r^{-d+2}\int_{0}^{r}\cos((\ratio \sqrt{\vartheta})^{-1}(k+z)s)(r^2-s^{2})^{(d-3)/2}\mathrm{d}s.
      \end{aligned}
      \end{equation*}
      As in the situation of $H_{\mathrm{sp},\ratio,j}$, we can continue by splitting the integral at $r-\varepsilon$. Using that 
      \begin{equation*}
          (r^2-s^2)^{(d-3)/2} \leq r^{d-5/2}\left(\varepsilon^{-1/2}\mathbbm{1}_{(0,r-\varepsilon)}(s)+(r-s)^{-1/2}\mathbbm{1}_{[r-\varepsilon,r]}\right),
      \end{equation*}
      it follows analogously with the choice $\varepsilon=\ratio \sqrt{\vartheta}/|k+z|$ that 
      \begin{equation*}
           H_{\mathrm{te},\ratio,k}(r)\lesssim r^{-1/2}\frac{\ratio^{1/2}}{|k+z|^{1/2}}.
      \end{equation*}
      Thus, 
      \begin{equation*}
      \begin{aligned}
          \Vert f_{\mathrm{te},\ratio,k}\Vert^2_{L^2((-\pi,\pi))}&\lesssim\frac{\ratio}{|k+z|}\int_{-\pi}^{\pi}\sin^8(r/2)\left(\sum_{l \in \mathbb{Z}}^{n} |r + 2\pi l|^{\beta-7/2}\right)^2\mathrm{d}r\lesssim\frac{\ratio}{|k+z|}.
      \end{aligned}\qedhere
      \end{equation*}
    \end{step}
\end{proof}
Summing \eqref{eq:spatial_sum_decay} over $|j|\geq2$ yields by \Cref{assumption:bias_convergence}
\begin{equation*}
    \sum_{|j|\geq2}\sum_{k\in\mathbb{Z}}C_{\mathrm{sp},k,j,z}^2\lesssim\log(m)\ratio^{-1}\rightarrow0.
\end{equation*}
The convergence in the last display is required to establish the convergence in \Cref{proposition:variance_space_time_variation}. As mentioned in \Cref{rmk:assumption_weakening}, the claim of \Cref{proposition:variance_space_time_variation} also holds in $d\geq3$ when \Cref{assumption:bias_convergence} is not fulfilled. This can be achieved by improving \Cref{lemma:C_k_j_f_a_j} significantly, i.e.\ instead of \eqref{eq:spatial_sum_decay} we can show the stronger statement
\begin{equation}
\label{eq: improvement C_k_j_gamma}
    \sum_{k \in \mathbb{Z} } C_{\mathrm{sp},k,j,z}^{2}\lesssim \frac{1}{\ratio^2|j+z|^2},\quad d\geq3,
\end{equation}
which entails
\begin{equation*}
    \sum_{|j|\geq2}\sum_{k\in\mathbb{Z}}C_{\mathrm{sp},k,j,z}\lesssim\ratio^{-2}\rightarrow0,\quad \ratio\rightarrow\infty.
\end{equation*}
Indeed, instead of the splitting-argument in \eqref{eq: integral_splitting_H}, we can apply integration by parts directly, using that 
\begin{equation*}
\begin{aligned}
    \left|\frac{\mathrm{d}}{\mathrm{d}v}\left(v^{\beta-d-1}(v^2-s^2)^{(d-3)/2}\right)\right|&\lesssim 
    \begin{cases}
        v^{\beta-5},\quad &d=3,\\
        v^{\beta-d-2}(v^2-s^2)^{(d-3)/2}+v^{\beta-d}(v^2-s^2)^{(d-5)/2},\quad &d>3,
    \end{cases}\\
    &\lesssim\begin{cases}
        v^{\beta-5},\quad &d=3,\\
        v^{\beta-5}+v^{\beta-4}|s|^{1/2}(v-|s|)^{-1/2},\quad &d>3.
    \end{cases}
\end{aligned}
\end{equation*}
Consequently, 
\begin{equation*}
    \begin{aligned}
        H_{\mathrm{sp},\ratio,j}(s)\lesssim \frac{1}{\ratio|j+z|}|s|^{\beta-4}
    \end{aligned}
\end{equation*}
and 
\begin{equation*}
    \Vert f_{\mathrm{sp},\ratio,j}\Vert^2_{L^2((-\pi,\pi))}\lesssim\frac{1}{\ratio^2|j+z|^2},
\end{equation*}
proving \eqref{eq: improvement C_k_j_gamma}. In principle, \eqref{eq: improvement C_k_j_gamma} can be refined even further, utilizing multiple applications of integration by parts in growing dimensions. Intuitively, this stems from the idea of separating the directions in which the temporal and spatial indices $j$ and $k$, respectively, act. The contribution of $\rho\cdot\fv$ to $|\fv|$ becomes significantly more important in lower dimensions, making a separation more difficult. The case $d=2$ constitutes an edge case, resulting in a slower Fourier decay, whereas separation becomes impossible in $d=1$. There, \Cref{assumption:helper} imposes a tool to control the cosine-frequencies $k$ and $j\ratio$, which is used to bound the coefficients $C_{\mathrm{sp},k,j,z}$ directly. 

\begin{lemma}
  \label[lemma]{lemma:d_1_bound_C_j_k} Suppose that $d=1$ and $j,k\in\mathbb{Z}$ and let \Cref{assumption:helper} hold true. Let $z\in\{0,\pm 1,\pm2\}$. If $|j+z|\geq1$, then  $C_{\mathrm{sp},k,j,z}$ defined in \Cref{lemma:C_k_j_f_a_j} satisfies
  \begin{equation*}
  C_{\mathrm{sp},k,j,z}^2 \lesssim \ratio^{-4}.
  \end{equation*}
  Similarly, if $|k+z| \geq 1$, then $C_{\mathrm{te},k,j,z}$ defined in \Cref{lemma:C_k_j_f_a_j} satisfies
  \begin{equation*}
  C_{\mathrm{te},k,j,z}^2 \lesssim \ratio^4.
  \end{equation*}
  \end{lemma}
  \begin{proof}
  \begin{step}[Arguments for $C_{\mathrm{sp},k,j,z}$]
  If $d=1$, we can simply assume that $\rho=1$. By direct computation
  \begin{equation*}
  \begin{aligned}
  C_{\mathrm{sp},k,j,z}&=\int_{\mathbb{R}} \cos(k \fv)\cos((j+z) \ratio \sqrt{\vartheta} \fv)\sin^{4}( \fv/2)|\fv|^{\beta-3}\mathrm{d}\fv\\
  &=2\int_{0}^{\infty} \cos(k \fv)\cos((j+z) \ratio \sqrt{\vartheta} \fv) \fv/2)\sin^{4}( \fv/2)\fv^{\beta-3}\mathrm{d}\fv\\
  &=\int_{0}^{\infty}\left(\cos((k + (j + z)\ratio \sqrt{\vartheta}) \fv)+\cos((k  - (j + z)\ratio \sqrt{\vartheta}) \fv)\right)\gspzero(\fv) \mathrm{d}\fv.
  \end{aligned}
  \end{equation*}
  Note that $\gspzero(0)=\gspzero^\prime(0)=0$. Hence, applying integration by parts twice, we obtain
  \begin{equation*}
  \begin{aligned}
  |C_{\mathrm{sp},k,j,z}| &\lesssim \sum_{z \in \{0, \pm1, \pm 2\}} \left(\frac{1}{|k + (j+z)\ratio \sqrt{\vartheta}|^{2}}+ \frac{1}{|k - (j+z)\ratio \sqrt{\vartheta}|^{2}}\right) \Vert \gspzero'' \Vert_{L^{1}((0,\infty))}\\
  &\lesssim \ratio^{-2},
  \end{aligned}
  \end{equation*}
  where we have used $|k \pm (j+z)\ratio \sqrt{\vartheta}| \geq ||j+z|\ratio \sqrt{\vartheta} - |k|| \geq \ratio \sqrt{\vartheta} - n \geq \ratio \sqrt{\vartheta}/2$ due to \Cref{assumption:helper}.
  \end{step}
  \begin{step}[Argument for $C_{\mathrm{te},k,j,z}$]
  If $d=1$, we have 
  \begin{equation*}
  C_{\mathrm{te},k,j,z} =\int_{\mathbb{R}} \cos(j  \fv)\cos((k+z)(\ratio \sqrt{\vartheta})^{-1}\fv)\sin^{4}(\fv/2)|\fv|^{\beta-3}\mathrm{d}\fv.
  \end{equation*}
  Analogously, 
  \begin{equation*}
  |C_{\mathrm{te},k,j,z}| \lesssim \sum_{z \in \{0, \pm1, \pm 2\}} \left(\frac{1}{|j + (k+z)(\ratio \sqrt{\vartheta})^{-1}|^{2}}+ \frac{1}{|j - (k+z)(\ratio \sqrt{\vartheta})^{-1}|^{2}}\right) \Vert \gtemp'' \Vert_{L^{1}((0,\infty))} \lesssim \ratio^{-2}. \qedhere
  \end{equation*}
  \end{step}
  \end{proof}

  \begin{lemma}
      \label[lemma]{lemma: numerator_space_time_increments}
      Grant \Cref{assumption:bias_convergence} and \Cref{assumption:helper}. Then, in the asymptotic regime $\alpha \rightarrow \infty$, the numerator in \eqref{eq:lyapunov_condition_box_increments} satisfies
      \begin{equation*}
          \begin{aligned}
              &\left(\max_{j=1,\dots,m}\max_{l=1,\dots,n} \sum_{k=1}^{n}\sum_{i=1}^{m}|\mathrm{Cov}(\mathbf{I}_{\mathrm{sp},\mathrm{te},i,k},\mathbf{I}_{\mathrm{sp},\mathrm{te},j,l})|\right)^2=\begin{cases}
                  \mathcal{O}(\delta^2\lambda^{4-2\beta}(n^2m^2+n^2m^4\ratio^{-4})),\quad &d=1,\\
                  \mathcal{O}(\delta^2\lambda^{4-2\beta}(n^2m^2+nm^3\ratio^{-1})),\quad &d\geq2,
              \end{cases}
        \end{aligned}
    \end{equation*}
    and for $\alpha \rightarrow 0$
    \begin{equation*}
            \begin{aligned}
              &\left(\max_{j=1,\dots,m}\max_{l=1,\dots,n} \sum_{k=1}^{n}\sum_{i=1}^{m}|\mathrm{Cov}(\mathbf{I}_{\mathrm{sp},\mathrm{te},i,k},\mathbf{I}_{\mathrm{sp},\mathrm{te},j,l})|\right)^2=\begin{cases}
                  \mathcal{O}(\delta^2\lambda^{4-2\beta}(n^2m^2+n^2m^4\ratio^{4})),\quad &d=1,\\
                  \mathcal{O}(\delta^2\lambda^{4-2\beta}(n^2m^2+nm^3\ratio)),\quad &d\geq2.
              \end{cases}
          \end{aligned}
      \end{equation*}
  \end{lemma}
  \begin{proof}
  We prove the claim for the case $\ratio\rightarrow\infty$. The arguments for $\ratio\rightarrow0$ are identical. According to \Cref{lemma:covariance_space_time_increments}, we can bound 
      \begin{equation*}
          \begin{aligned}
              &\max_{j=1,\dots,m}\max_{l=1,\dots,n} \sum_{k=1}^{n}\sum_{i=1}^{m}|\mathrm{Cov}(\mathbf{I}_{\mathrm{sp},\mathrm{te},i,k},\mathbf{I}_{\mathrm{sp},\mathrm{te},j,l})|\\
              &\lesssim \delta\lambda^{2-\beta}\left(nm+m\sum_{|j|\leq m-1}\sum_{|k|\leq n-1}\left|\int_{\mathbb{R}^d}\cos(k\rho\cdot w)\cos(j\ratio\sqrt{\vartheta}|\fv|)\sin^4(\ratio\sqrt{\vartheta}|\fv|/2)\gspzero(\fv)\mathrm{d}\fv\right|\right)\\
              &\lesssim\delta\lambda^{2-\beta}\left(nm+m\sum_{|j|\leq m-1}\sum_{z\in\{0,\pm1\pm2\}}\sum_{|k|\leq n-1}\left|C_{\mathrm{sp},k,j,z}\right|\right)
          \end{aligned}
      \end{equation*}
      with $C_{\mathrm{sp},k,j,z}$ from \eqref{eq:coeff_gamma}. Following the proof structure of \Cref{proposition:variance_space_time_variation}, we further obtain with the Cauchy-Schwarz inequality, \Cref{lemma:C_k_j_f_a_j} and \Cref{lemma:d_1_bound_C_j_k}
      \begin{equation*}
          \begin{aligned}
              \sum_{|j|\leq m-1}\sum_{z\in\{0,\pm1\pm2\}}\sum_{|k|\leq n-1}\left|C_{\mathrm{sp},k,j,z}\right|&\lesssim \sqrt{n}\sum_{|j|\leq m-1}\left(\sum_{\substack{z\in\{0,\pm1\pm2\}\\
              z=-j}}+\sum_{\substack{z\in\{0,\pm1\pm2\}\\
              z\neq-j}}\right)\left(\sum_{|k|\leq n-1}C_{\mathrm{sp},k,j,z}^2\right)^{1/2}\\
              &\lesssim\sqrt{n}\begin{cases}
                  1+m\sqrt{n}\ratio^{-2},\quad &d=1,\\
                  1+\sqrt{m}\ratio^{-1/2},\quad &d\geq2.
              \end{cases}
          \end{aligned}
      \end{equation*}
      In total,
      \begin{equation*}
          \begin{aligned}
              &\left(\max_{j=1,\dots,m}\max_{l=1,\dots,n} \sum_{k=1}^{n}\sum_{i=1}^{m}|\mathrm{Cov}(\mathbf{I}_{\mathrm{sp},\mathrm{te},i,k},\mathbf{I}_{\mathrm{sp},\mathrm{te},j,l})|\right)^2=\begin{cases}
                  \mathcal{O}(\delta^2\lambda^{4-2\beta}(n^2m^2+n^2m^4\ratio^{-4})),\quad &d=1,\\
                  \mathcal{O}(\delta^2\lambda^{4-2\beta}(n^2m^2+nm^3\ratio^{-1})),\quad &d\geq2.
              \end{cases}
          \end{aligned}
      \end{equation*}
  \end{proof}

\begin{proof}[Proof of \Cref{theorem:clt_box_increments}]
By \Cref{lemma: numerator_space_time_increments} and \Cref{proposition:variance_space_time_variation}, it holds
\begin{equation*}
\frac{\left( \max_{j=1,\dots,m}\max_{l=1,\dots,n} \sum_{k=1}^{n}\sum_{i=1}^{m}|\mathrm{Cov}(\mathbf{I}_{\mathrm{sp},\mathrm{te},i,k},\mathbf{I}_{\mathrm{sp},\mathrm{te},j,l})|\right)^{2}}{\V(\mathbf{V}_{\mathrm{sp}, \mathrm{te}})} =  \begin{cases}
\mathcal{O}(nm^{-1}+nm\ratio^{-4}), \quad &d=1,\\
\mathcal{O}(nm^{-1}+\ratio^{-1}),\quad &d\geq2,
\end{cases}
\end{equation*}
and
\begin{equation*}
\frac{\left( \max_{j=1,\dots,m}\max_{l=1,\dots,n} \sum_{k=1}^{n}\sum_{i=1}^{m}|\mathrm{Cov}(\mathbf{I}_{\mathrm{sp},\mathrm{te},i,k},\mathbf{I}_{\mathrm{sp},\mathrm{te},j,l})|\right)^{2}}{\V(\mathbf{V}_{\mathrm{sp}, \mathrm{te}})} =  \begin{cases}
\mathcal{O}(nm^{-1}+nm\ratio^{4}), \quad &d=1,\\
\mathcal{O}(nm^{-1}+\ratio),\quad &d\geq2,
\end{cases}
\end{equation*}
given the regime $\ratio \rightarrow \infty$ and $\ratio \rightarrow 0$, respectively. By \Cref{assumption:bias_convergence} these expressions converge to zero.
Overall, we have
\begin{equation*}
\frac{\mathbf{V}_{\mathrm{sp}, \mathrm{te}}-\E[\mathbf{V}_{\mathrm{sp}, \mathrm{te}}]}{\sqrt{\V(\mathbf{V}_{\mathrm{sp}, \mathrm{te}})}} \xrightarrow{d} N(0,1),
\end{equation*}
which yields the convergences
\begin{equation*}
{\sqrt{nm}\left(\frac{\delta^{-1}\lambda^{\beta-2}}{nm^{2}}\mathbf{V}_{\mathrm{sp}, \mathrm{te}}-\E\left[\frac{\delta^{-1}\lambda^{\beta-2}}{nm^{2}}\mathbf{V}_{\mathrm{sp}, \mathrm{te}} \right]\right)} \xrightarrow{d} N\left(0, \frac{\boxconstvarsp}{\vartheta^{2}}\right), \quad \ratio \rightarrow \infty,
\end{equation*}
and
\begin{equation*}
{\sqrt{nm}\left(\frac{\delta^{\beta-3}}{nm^{2}}\mathbf{V}_{\mathrm{sp}, \mathrm{te}}-\E\left[\frac{\delta^{\beta-3}}{nm^{2}}\mathbf{V}_{\mathrm{sp}, \mathrm{te}} \right]\right)} \xrightarrow{d} N\left(0, \frac{\boxconstvartemp}{\vartheta^{\beta}}\right), \quad \ratio \rightarrow 0.
\end{equation*}
By \Cref{proposition:box_expectation} and \Cref{assumption:bias_convergence}, we can replace the expectations in the last display by their limits, concluding
\begin{equation*}
{\sqrt{nm}\left(\frac{\delta^{-1}\lambda^{\beta-2}}{nm^{2}}\mathbf{V}_{\mathrm{sp}, \mathrm{te}}-\frac{C_{\mathrm{box,sp,\E}}}{\vartheta}\right)} \xrightarrow{d} N\left(0, \frac{\boxconstvarsp}{\vartheta^{2}}\right), \quad \ratio \rightarrow \infty,
\end{equation*}
and
\begin{equation*}
{\sqrt{nm}\left(\frac{\delta^{\beta-3}}{nm^{2}}\mathbf{V}_{\mathrm{sp}, \mathrm{te}}-\frac{C_{\mathrm{box,temp,\E}}}{\vartheta^{\beta/2}}\right)} \xrightarrow{d} N\left(0, \frac{\boxconstvartemp}{\vartheta^{\beta}}\right), \quad \ratio \rightarrow 0. \qedhere
\end{equation*}
\end{proof}
\subsection{General results on increments}
\label{section:general_results_on_increments}
Let $f:\mathbb{R}\rightarrow\mathbb{R}$ and $h:\mathbb{R}^{2} \rightarrow \mathbb{R}$.
Let $q,z\in\mathbb{Z}$. Recall the second-order discrete increment in two variables from \eqref{eq:second_order_incremental_operator_d_2}, i.e.\
\begin{equation*}
\begin{aligned}
\mathcal{I}^{(2)}[h](q, z) &= 4h(q, z) + h(q+1,z+1) + h(q+1, z-1) + h(q-1,z+1) + h(q-1, z-1) \\
& \qquad- 2\left(h(q,z+1) + h(q, z-1) + h(q-1,z)+h(q+1,z)\right).
\end{aligned}
\end{equation*}
Let $p\in\mathbb{N}$ denote the number of increments. The first, second, third and fourth-order increments of $f$ and index $z$ are given by
\begin{equation}
\label{eq:increments_1_2_3_4}
\begin{aligned}
\mathfrak{I}^{(1)}[f](z) &= f(z+1)-f(z-1),\\
	\mathfrak{I}^{(2)}[f](z) &= f(z+1) + f(z-1) -2f(z),\\
\mathfrak{I}^{(3)}[f](z) &= f(z+2)-2f(z+1)+2f(z-1)-f(z-2),\\
\mathfrak{I}^{(4)}[f](z) &= f(z+2)-4f(z+1) + 6f(z)-4f(z-1)+f(z-2).
\end{aligned}
\end{equation}
The pattern described by \eqref{eq:increments_1_2_3_4} can be formalised as follows. If $p$ is even, we define the incremental operator applied to the real function $f$ as
\begin{equation*}
	\mathfrak{I}^{(p)}[f](z)\coloneqq\underline{\mathfrak{I}}^{(p)}[f](z - p/2),
\end{equation*}
with the forward incremental operator
\begin{equation*}
	\underline{\mathfrak{I}}^{(p)}[f](z)=\sum_{q=0}^p (-1)^{p-q} \binom{p}{q}f(z+q).
\end{equation*}
 If $p$ is odd, we have
\begin{equation*}
	\mathfrak{I}^{(p)}[f](z)\coloneqq\sum_{q=0}^{p-1} (-1)^{q} \binom{p-1}{q} \left(f(z+q+1-(p-1)/2)-f(z+q-1-(p-1)/2)\right).
\end{equation*}
\begin{lemma}
\label[lemma]{result:higher_order_increments_trig}
The increments of the cosine and the sine admit the representations
\begin{equation*}
\mathfrak{I}^{(p)}[\cos(\cdot r)](z)= \begin{cases}
2^{p}\sin^{p}(r/2)\cos\left(zr + \frac{p\pi}{2}\right), \quad p \text{ even },\\
-2^{p}\sin(r)\sin^{p-1}(r/2)\sin(zr + (n-1)\pi/2), \quad p \text{ odd },
\end{cases}
\end{equation*}
and
\begin{equation*}
\mathfrak{I}^{(p)}[\sin(\cdot r)](z)= \begin{cases}
2^{p}\sin^{p}(r/2)\sin\left(zr + \frac{p\pi}{2}\right), \quad p \text{ even },\\
 2^{p}\sin(r) \sin^{p-1}(r/2)\cos(zr + (p-1)\pi/2), \quad p \text{ odd }.
\end{cases}
\end{equation*}
\end{lemma}
\begin{proof}
\begin{case}[$p$ is even]
Let us start with the simpler case that the incremental order $p$ is even. 
Now, notice that in the case of the cosine, we have
\begin{equation*}
	\underline{\mathfrak{I}}^{(p)}[\cos(\cdot r)](z)= \mathrm{Re}\left( \sum_{q=0}^p (-1)^{p-q} \binom{p}{q}\exp(\ii r(z+q))\right) =\mathrm{Re}\left( e^{\ii rz}\sum_{q=0}^p (-1)^{p-q} \binom{p}{q}\exp(\ii rq)\right).
\end{equation*}
Note that, as the coefficients of the increments stem from Pascal's triangle, we quite naturally observe
\begin{equation*}
	\underline{\mathfrak{I}}^{(p)}[\cos(\cdot r)](z)=\mathrm{Re}\left( e^{\ii rz}(e^{\ii r}-1)^{p}\right), \quad \underline{\mathfrak{I}}^{(p)}[\sin(\cdot r)](z)=\mathrm{Im}\left( e^{\ii rz}(e^{\ii r}-1)^{p}\right).
\end{equation*}
Next, notice that
\begin{equation*}
	e^{\ii r}-1=e^{\ii r}-e^{\ii r/2}e^{-\ii r/2} = e^{\ii r/2}(e^{\ii r/2}-e^{-\ii r/2}).
\end{equation*}
Thus, plugging in Euler's formula, we have
\begin{equation*}
	e^{\ii r/2}-e^{-\ii r/2} = 2\ii \sin(r/2).
\end{equation*}
Consequently, we obtain
\begin{equation*}
	\underline{\mathfrak{I}}^{(p)}[\cos(\cdot r)](z) = \mathrm{Re}\left(  \sin^{p}(r/2)(2\ii)^{p}e^{\ii r(z + p/2)})\right)=2^{p}\sin^{p}(r/2)\mathrm{Re}\left( \ii^{p}e^{\ii r(z + p/2)}\right).
\end{equation*}
We can write the latter object into one exponential, by noticing that $\ii^{p}=e^{p\ii\pi/2}$ such that
\begin{equation*}
	\begin{aligned}
		\underline{\mathfrak{I}}^{(p)}[\cos(\cdot r)](z) & = 2^{p}\sin^{p}(r/2)\mathrm{Re}\left( e^{\ii(rz + rp/2 + p\pi/2)}\right)= 2^{p}\sin^{p}(r/2)\cos\left(zr + \frac{pr}{2} + \frac{p\pi}{2}\right), \\
		\underline{\mathfrak{I}}^{(p)}[\sin(\cdot r)](z) & =2^{p}\sin^{p}(r/2)\mathrm{Im}\left( e^{\ii(rz + rp/2 + p\pi/2)}\right)= 2^{p}\sin^{p}(r/2)\sin\left(zr + \frac{pr}{2} + \frac{p\pi}{2}\right).
	\end{aligned}
\end{equation*}
Consequently, we have for the symmetric differences
\begin{equation*}
	\begin{aligned}
		{\mathfrak{I}}^{(p)}[\cos(\cdot r)](z) & = \underline{\mathfrak{I}}^{(p)}[\cos(\cdot r)](z-p/2)= 2^{p}\sin^{p}(r/2)\cos\left((z-\frac{p}{2})r + \frac{pr}{2} + \frac{p\pi}{2}\right) \\
		                                                & =2^{p}\sin^{p}(r/2)\cos\left(zr + \frac{p\pi}{2}\right),                                                                                                                         \\
		{\mathfrak{I}}^{(p)}[\sin(\cdot r)](z) & =\underline{\mathfrak{I}}^{(p)}[\sin(\cdot r)](z-p/2)                                                                                                                                    \\
		                                                & = 2^{p}\sin^{p}(r/2)\sin\left(zr + \frac{p\pi}{2}\right)
	\end{aligned}
\end{equation*}
Note that since we have assumed that $p$ was even, this is simply a shift of the order $\pi$.    
\end{case}
\begin{case}[$p$ is odd] We can reduce the odd case to the even case using the fact that $\mathfrak{I}^{(p)}=\mathfrak{I}^{(p-1)}\mathfrak{I}^{(1)}$. In particular, we compute for the first-order increment $\mathfrak{I}^{(1)}[\sin(\cdot r)](z)=2\sin(r)\cos(zr)$ and $\mathfrak{I}^{(1)}[\cos(\cdot r)](z)=-2\sin(r)\sin(zr)$. Thus, we have
\begin{align*}
\mathfrak{I}^{(p)}[\cos(r \cdot)](z) &=-2\sin(r) 2^{p-1}\sin^{p-1}(r/2)\sin\left(zr + \frac{(p-1)\pi}{2}\right),\\
\mathfrak{I}^{(p)}[\sin(r \cdot)](z) &= 2 \sin(r) 2^{p-1}\sin^{p-1}(r/2)\cos\left(zr + \frac{(p-1)\pi}{2}\right).\tag*{\qedhere}
\end{align*}
\end{case}
\end{proof}

Next, we will present a couple of results to relate $\mathcal{I}^{(2)}$ from \eqref{eq:second_order_incremental_operator_d_2} with the increments \eqref{eq:increments_1_2_3_4}.
\begin{lemma}
\label[lemma]{lemma:f_+_-_increments}
It holds
\begin{equation*}
\begin{aligned}
    \mathcal{I}^{(2)}[h](q,z) &= \mathfrak{I}^{(4)}[f](q+z), \quad h(q,z)=f(q+z),\\
        \mathcal{I}^{(2)}[h](q,z) &= \mathfrak{I}^{(4)}[f](q-z),\quad h(q,z)=f(q-z).
\end{aligned}
\end{equation*}
\end{lemma}
\begin{proof}
When $h(q,z)=f(q+z),$ the proof follows immediately by observing
\begin{equation*}
\begin{aligned}
\mathcal{I}^{(2)}[h](q, z) &= 4f(q+z)+f(q+z+2)+f(q+z)+f(q+z)+f(q+z-2)\\
&\qquad -2(f(q+z+1)+f(q+z-1) +f(q+z-1)+f(q+z+1))\\
&=6f(q+z)+f(q+z+2)+f(q+z-2)-4f(q+z+1)-4f(q+z-1)\\
&=\mathfrak{I}^{(4)}[f](q+z)
\end{aligned}
\end{equation*}
and similarly for the case $h(q,z)=f(q-z)$.
\end{proof}

\begin{lemma}
\label[lemma]{result:discrete_product_rule}
Suppose we have a function $h(x,y)=xf(x-y)$. Then, we have
\begin{equation*}
\mathcal{I}^{(2)}[h](q,z) = i \mathfrak{I}^{(4)}[f](q-z) + \mathfrak{I}^{(3)}[f](q-z).
\end{equation*}
\end{lemma}
\begin{proof}
From the definition, we have
\begin{align*}
&\mathcal{I}^{(2)}[h](q,z) \\
&= 4q f(q-z) + (q+1)f(q-z) + (q+1)f(q-z+2) + (q-1)f(q-z-2) + (q-1)f(q-z) \\
&\quad -2\Big(qf(q-z-1) + qf(q-z+1) + (q-1)f(q-z-1)+(q+1)f(q-z+1)\Big)\\
&= q\Big(6f(q-z) + f(q-z+2) + f(q-z-2) - 4f(q-z-1) - 4f(q-z+1)\Big)\\
&\quad + \Big(f(q-z+2) - f(q-z-2) + 2f(q-z-1) - 2f(q-z+1)\Big)\\
&= q \mathfrak{I}^{(4)}[f](q-z) + \mathfrak{I}^{(3)}[f](q-z).\tag*{\qedhere}
\end{align*}
\end{proof}
\begin{lemma}
\label[lemma]{lemma:minimum_increment_lemma}
Suppose that $h(x,y)=(x \land y) f(x-y)$ and that $f$ is symmetric. Then, we have 
\begin{equation*}
\mathcal{I}^{(2)}[h](q,z)=\begin{cases}
(q\land z)\mathfrak{I}^{(4)}[f](|q-z|) -\mathfrak{I}^{(3)}[f](|q-z|)+ 4f(1)-2f(2), &q=z,\\
(q \land z)\mathfrak{I}^{(4)}[f](|q-z|)-\mathfrak{I}^{(3)}[f](|q-z|)-f(1), &|q-z|=1,\\
(q\land z) \mathfrak{I}^{(4)}[f](|q-z|) - \mathfrak{I}^{(3)}[f](|q-z|), &|q-z|\geq2.
\end{cases}
\end{equation*}
\end{lemma}
\begin{proof}
From the definition, we have
\begin{equation}
\label{eq:second_order_general_mini_rep}
\begin{aligned}
    \mathcal{I}^{(2)}[h](q,z) 
    &= 4(q \land z)f(q, z) 
    \\&\qquad+ ((q+1) \land (z+1))f(q+1,z+1) + ((q+1) \land (z-1))f(q+1, z-1) \\
    &\qquad + ((q-1) \land (z+1))f(q-1,z+1) + ((q-1) \land (z-1))f(q-1, z-1) \\
    &\qquad - 2\Big((q \land (z+1))f(q,z+1) + (q \land (z-1))f(q, z-1) \\
    &\qquad\qquad+ ((q-1)\land z)f(q-1,z)+ ((q+1)\land z)f(q+1,z)\Big)\\
    &= 4(q \land z)f(q-z) 
    \\&\qquad+ ((q+1) \land (z+1))f(q-z) + ((q+1) \land (z-1))f(q-z+2) \\
    &\qquad + ((q-1) \land (z+1))f(q-z-2) + ((q-1) \land (z-1))f(q-z) \\
    &\qquad - 2\Big((q \land (z+1))f(q-z-1) + (q \land (z-1))f(q-z+1) \\
    &\qquad\qquad+ ((q-1)\land z)f(q-z-1)+ ((q+1)\land z)f(q-z+1)\Big).
\end{aligned}
\end{equation}
\begin{case}[$|q-z| \geq 2$.]
    Say that $q+1 \leq z-1$. Then, all the $z$ in the minima in \eqref{eq:second_order_general_mini_rep} will simply vanish and the result will reduce to the case $x \leq y$ with $h(x,y)=x f (x-y)$. Similarly, since we have assumed $f$ to be symmetric, the case $z+1 \leq q-1$ reduces to the analysis for $y \leq x$ and $h(x,y)=yf(y-x)$. Thus, whenever $|q-z|\geq 2$, \Cref{result:discrete_product_rule} yields the representation of the second-order increments:
    \begin{equation}
    \label{eq:two_cases_representation_second_order_increment}
    \mathcal{I}^{(2)}[h](q,z)= 
        \begin{cases}
        q \mathfrak{I}^{(4)}[f](q-z) + \mathfrak{I}^{(3)}[f](q-z), &q+2 \leq z,\\
        z \mathfrak{I}^{(4)}[f](z-q) + \mathfrak{I}^{(3)}[f](z-q), &z+2 \leq q.
        \end{cases}
    \end{equation}
    Note that since the third order increment is an odd increment and the fourth order increment is even, we have the property
    \begin{equation}
    \label{eq:oddness_third_order_increment}
    \begin{aligned}
        \mathfrak{I}^{(3)}[f](q-z)&=f(q-z+2)-2f(q-z+1)+2f(q-z-1)-f(q-z-2)\\
        &=f(z-q-2)-2f(z-q-1)+2f(z-q+1)-f(z-q+2)\\
        &=-\mathfrak{I}^{(3)}[h](z-q),\\
        \mathfrak{I}^{(4)}[f](q-z)&= f(q-z+2)-4f(q-z+1) + 6f(q-z)-4f(q-z-1)+f(q-z-2)\\
        &=f(z-q-2)-4f(z-q-1) + 6f(z-q)-4f(z-q+1)+f(z-q+2)\\
        &= f(z-q+2)-4f(z-q+1) + 6f(z-q)-4f(z-q-1)+f(z-q-2)\\
        &= \mathfrak{I}^{(4)}[f](z-q).
    \end{aligned}
    \end{equation}
    where we have used the symmetry of $f$.
    Consequently, \eqref{eq:oddness_third_order_increment} allows us derive the following more compact representation for \eqref{eq:two_cases_representation_second_order_increment}:
    \begin{equation}
    \begin{aligned}
        \mathcal{I}^{(2)}[h](q,z)&= 
            \begin{cases}
            q \mathfrak{I}^{(4)}[f](q-z) + \mathfrak{I}^{(3)}[f](q-z), &q+2 \leq z,\\
            z \mathfrak{I}^{(4)}[f](z-q) + \mathfrak{I}^{(3)}[f](z-q), &z+2 \leq q,
            \end{cases}\\
            &= (q \land z ) \mathfrak{I}^{(4)}[f](|q-z|)-\mathfrak{I}^{(3)}[f](|q-z|).
    \end{aligned}
    \end{equation}
    
\end{case}
\begin{case}[$z=q$] We have that
\begin{equation*}
\begin{aligned}
    \mathcal{I}^{(2)}[h](q, q) 
    &= 4q f(0) + (q+1)f(0) + (q-1)f(2) + (q-1)f(-2) + (q-1)f(0) \\
    & \qquad - 2\Big(q f(-1) + (q-1)f(1) + (q-1)f(-1)+ qf(1)\Big)\\
    &= q \left( 6f(0) + f(2) + f(-2) -4f(-1)-4f(1)\right) - f(2) - f(-2) +2f(1)+2f(-1)\\
    &=q (6 f(0)+ 2f(2) -8f(1)) + 4f(1)-2f(2).
\end{aligned}
\end{equation*}
We also have
\begin{equation*}
\begin{aligned}
\mathfrak{I}^{(4)}[f](0) &= f(2)-4f(1)+6f(0) -4f(-1) +f(-2)=6f(0)+2f(2)-8f(1)\\
\mathfrak{I}^{(3)}[f](0) &=0,
\end{aligned}
\end{equation*}
yielding 
\begin{equation*}
\begin{aligned}
 \mathcal{I}^{(2)}[h](q, q) &= q\mathfrak{I}^{(4)}[f](0) + 4f(1)-2f(2)\\
 &=(q\land z)\mathfrak{I}^{(4)}[f](|q-z|) -\mathfrak{I}^{(3)}[f](|q-z|)+ 4f(1)-2f(2), \quad q=z.
 \end{aligned}
\end{equation*}
\end{case}
\begin{case}[$|q-z|=1$] Assume first that $z=q+1$. We observe
\begin{equation*}
\begin{aligned}
    \mathcal{I}^{(2)}[h](q, q+1) &= 4qf(-1) + (q+1)f(-1) + qf(1)+ (q-1)f(-3) + (q-1)f(-1)\\
    &\qquad- 2\Big(qf(-2) + qf(0) + (q-1)f(-2)+ (q+1)f(0)\Big)\\
    &= q(7 f(1) +f(3)-4f(2)-4f(0)) + 2f(2)-2f(0) -f(3).
\end{aligned}
\end{equation*}
Note that
\begin{equation*}
\begin{aligned}
\mathfrak{I}^{(4)}[f](1)&=7f(1) + f(3) - 4f(2)-4f(0),\\
\mathfrak{I}^{(3)}[f](1)&=f(3)-f(1) +2f(0)-2f(2),
\end{aligned}
\end{equation*}
leading to
\begin{equation*}
 \mathcal{I}^{(2)}[h](q, q+1)=q\mathfrak{I}^{(4)}[f](1)-\mathfrak{I}^{(3)}[f](1)-f(1).
\end{equation*}
Similarly, we can derive the case $z=q-1$ leading in total to 
\begin{equation*}
     \mathcal{I}^{(2)}[h](q,z)=(q \land z)\mathfrak{I}^{(4)}[f](|q-z|)-\mathfrak{I}^{(3)}[f](|q-z|)-f(1), \quad |q-z|=1. \qedhere
\end{equation*}
\end{case}
\end{proof}
\begin{lemma}
\label[lemma]{result:increment_sine_absolute}
Let $h(x,y)=f(|x-y|)$. Then, we have
\begin{equation*}
\mathcal{I}^{(2)}[h](q,z)=\begin{cases}
6f(0)-8f(1)+2f(2), & |q-z|=0,\\
7f(1)+f(3)-4f(2)-4f(0), & |q-z|=1,\\
\mathfrak{I}^{(4)}[f](|q-z|), & |q-z| \geq 2.
\end{cases}
\end{equation*}
\end{lemma}
\begin{proof}
\begin{case}[$|q-z|\geq 2$]
Assume first that $q>z$. By definition \eqref{eq:second_order_incremental_operator_d_2} we obtain
    \begin{equation*}
        \begin{aligned}
            \mathcal{I}^{(2)}[h](q,z)&=f(|q-z+2|)-4f(|q-z+1|)+6f(|q-z|)-4f(|q-z-1|)+f(|q-z-2|)\\
            &=f(q-z+2)-4f(q-z+1)+6f(q-z)-4f(q-z-1) + f(q-z-2) \\
            &=\mathfrak{I}^{(4)}[f](q-z).
        \end{aligned}
    \end{equation*}
    If $z>q$, we similarly obtain by the symmetry of $|\cdot|$ that $\mathcal{I}^{(2)}[h](q-z)=\mathfrak{I}^{(4)}[f](z-q)$. Thus, in total this amounts to
    \begin{equation*}
        \mathcal{I}^{(2)}[h](q,z)=\mathfrak{I}^{(4)}[f](|q-z|). 
   \end{equation*}
\end{case}
\begin{case}[$q=z$]
We have
\begin{equation*}
\begin{aligned}
\mathcal{I}^{(2)}[h](q,z)=6h(0)-8h(1)+2h(2).
\end{aligned}
\end{equation*}
\end{case}
\begin{case}[$z=q+1$]
It holds directly by \eqref{eq:second_order_incremental_operator_d_2} that
\begin{equation*}
\begin{aligned}
\mathcal{I}^{(2)}[h](q,z)=7f(1)+f(3)-4f(2)-4f(0).
\end{aligned}\qedhere
\end{equation*} 
\end{case}
\end{proof}

\subsection{Auxiliary results}
\label{sec: ax_results_g}

\begin{lemma}
\label[lemma]{result:regularity_g}
Consider
\begin{equation*}
\begin{aligned}
\gtemp(\fv) &=\sin^{4}(|\fv|/2)|\fv|^{\beta-d-2}, \quad \fv \in \mathbb{R}^{d},\\
\gspzero(\fv) &= \sin^{4}(\rho \cdot \fv/2)|\fv|^{\beta-d-2}, \quad \fv \in \mathbb{R}^{d}.
\end{aligned}
\end{equation*}
We have $\gtemp, \gspzero \in L^{1}(\mathbb{R}^{d})$, $\nabla \gtemp, \nabla \gspzero \in L^{1}(\mathbb{R}^{d})$ and $\partial_i \partial_j \gtemp,  \partial_i \partial_j \gspzero \in L^{1}(\mathbb{R}^{d})$ for all $i,j=1,\dots,d$.
\end{lemma}
\begin{proof}
We prove the result for $\gtemp$. The proof for $\gspzero$ is analogous. 
\begin{step}[$\gtemp \in L^{1}(\mathbb{R}^{d})$]
We observe that with $\sarea$ being the surface measure of $\mathbb{S}^{d-1}$: 
\begin{equation*}
\begin{aligned}
\int_{\mathbb{R}^{d}}|\gtemp(\fv)|\mathrm{d}\fv
&=\sarea \int_{0}^{\infty}\sin^{4}(r/2)r^{\beta-d-2}r^{d-1}\mathrm{d}r\\
&= \sarea \int_{0}^{\infty}\sin^{4}(r/2)r^{\beta-3}\mathrm{d}r\\
&\leq \sarea \frac{1}{2^{4}}\int_{0}^{1}r^{\beta+1}\mathrm{d}r + \sarea \int_{1}^{\infty}r^{\beta-3}\mathrm{d}r <\infty,
\end{aligned}
\end{equation*}
since $\beta-3<-1$ and $\beta +1 > -1$ as $\beta \in (0,2 \land d)$.
\end{step}
\begin{step}[$\nabla \gtemp \in L^{1}(\mathbb{R}^{d})$] 
We may rewrite $\gtemp(\fv)=\widetilde{\gtemp}(|\fv|)$ with $\widetilde{\gtemp}(r)=\sin^{4}(r/2)r^{\beta-2-d}$. By the chain rule, we have $\partial_{\fv_i} \gtemp(\fv)=\widetilde{\gtemp}'(|\fv|)\partial_{\fv_i} |\fv|$. In particular, we have $\partial_{\fv_i} |\fv|=\frac{\fv_i}{|\fv|}$ such that $\partial_{\fv_i} \gtemp(\fv) = \widetilde{\gtemp}'(|\fv|)\frac{\fv_i}{|\fv|}$. As a consequence $\nabla \gtemp (\fv)=\widetilde{\gtemp}'(|\fv|) \frac{\fv}{|\fv|}$ and $|\nabla \gtemp(\fv)|=|\widetilde{\gtemp}'(|\fv|)|$. Next, we observe using the product rule that 
\begin{equation*}
\widetilde{\gtemp}'(r)=2 \cos\left(\frac{r}{2}\right) \sin^{3}\left(\frac{r}{2}\right) \, r^{{\beta} - d - 2} + \left({\beta} - d - 2\right) \sin^{4}\left(\frac{r}{2}\right) \, r^{{\beta} - d - 3}. 
\end{equation*}
Next, we would like to show that $\nabla \gtemp \in L^{1}(\mathbb{R}^{d})$. As a first step, we obtain the upper bound
\begin{equation*}
\begin{aligned}
\int_{\mathbb{R}^{d}} |\nabla \gtemp (\fv)|\mathrm{d}\fv&=\int_{\mathbb{R}^{d}}|\widetilde{\gtemp}'(|\fv|)|\mathrm{d}\fv= \sarea \int_{0}^{\infty} |\widetilde{\gtemp}'(r)|r^{d-1} \mathrm{d}r\\
&=\sarea \int_{0}^{\infty} \left|2 \cos\left(\frac{r}{2}\right) \sin^{3}\left(\frac{r}{2}\right) \, r^{{\beta} - d - 2} + \left({\beta} - d - 2\right) \sin^{4}\left(\frac{r}{2}\right) \, r^{{\beta} - d - 3} \right|r^{d-1}\mathrm{d}r\\
&\leq 2\sarea \int_{0}^{\infty} |\sin^{3}(r/2)|r^{\beta-3}\mathrm{d}r + |\beta-d-2| \sarea\int_{0}^{\infty}|\sin^{4}(r/2)|r^{\beta-4}\mathrm{d}r.
\end{aligned}
\end{equation*}
Next, we can split each of the integrals into two parts, corresponding to the region near the origin and the tail as in the first step and conclude that $\nabla \gtemp \in L^{1}(\mathbb{R}^{d})$.
\end{step}
\begin{step}[$\partial_{i} \partial_{j} g \in L^{1}(\mathbb{R}^{d})$]
We have $\gtemp(\fv)=\widetilde{\gtemp}(|\fv|)$. Thus, the second derivatives satisfy
\begin{equation*}
\partial_{i,j}^{2}\gtemp(\fv) = \widetilde{\gtemp}''(|\fv|)\frac{\fv_i \fv_j}{|\fv|^{2}} + \widetilde{\gtemp}'(|\fv|)\left(\frac{\mathbbm{1}(i=j)}{|\fv|}-\frac{\fv_i \fv_j}{|\fv|^{3}}\right),
\end{equation*}
and consequently
\begin{equation*}
|\partial_{i,j}^{2}\gtemp(\fv)| \lesssim |\widetilde{\gtemp}''(|\fv|)| + \frac{|\widetilde{\gtemp}'(|\fv|)|}{|\fv|}.
\end{equation*}
Clearly the mapping $\fv \mapsto {|\widetilde{\gtemp}'(|\fv|)|}{|\fv|^{-1}}$ is already in $L^{1}(\mathbb{R}^{d})$, since $\beta-1 > -1$ and $\beta-3<-1$. The same can be observed for the second-derivative once we explicitly compute
\begin{equation*}\begin{aligned}
\widetilde{\gtemp}(r)&=-\sin^{4}\left(\frac{r}{2}\right) \, r^{{\beta} - d - 2} + 3 \cos^{2}\left(\frac{r}{2}\right) \sin^{2}\left(\frac{r}{2}\right) \, r^{{\beta} - d - 2} \\
&\quad+ 4 \left({\beta} - d - 2\right) \cos\left(\frac{r}{2}\right) \sin^{3}\left(\frac{r}{2}\right) \, r^{{\beta} - d - 3} + \left({\beta} - d - 3\right) \left({\beta} - d - 2\right) \sin^{4}\left(\frac{r}{2}\right) \, r^{{\beta} - d - 4},
\end{aligned}
\end{equation*}
and passing to polar coordinates. \qedhere
\end{step}
\end{proof}

\begin{lemma}
\label[lemma]{result:uniform_integrability_f_lambda}
Consider the function $g_{\mathrm{sp, \lambda}}$ defined in \eqref{eq:g_sp_lambda}:
\begin{equation*}
\gsplambda(\fv)= (1-\mathrm{sinc}(2 t\sqrt{\vartheta}\lambda^{-1} |\fv|)\sin^{4}(\rho \cdot \fv / 2) |\fv|^{\beta-d-2}.
\end{equation*}
Then, the sequences $(g_{\mathrm{sp}, \lambda_n})_{n \in \mathbb{N}}$ and $(\nabla g_{\mathrm{sp}, \lambda_n})_{n \in \mathbb{N}}$ are uniformly integrable.
\end{lemma}
\begin{proof}
We begin by rewriting $\gsplambda$ as $\gsplambda(\fv)=h_{\mathrm{sp}, \lambda}(\fv)\gspzero(\fv)$ with  $h_{\mathrm{sp}, \lambda}(\fv)=(1-\mathrm{sinc}(2 t\sqrt{\vartheta}\lambda^{-1} |\fv|)$ and $\gspzero(\fv)=\sin^{4}(\rho \cdot \fv/2 )|\fv|^{\beta-d-2}$. The uniform integrability of $(g_{\mathrm{sp}, \lambda_n})_{n \in \mathbb{N}}$ follows analogously to \Cref{result:regularity_g} by additionally noting that $h_{\mathrm{sp}, \lambda}(\fv)\leq 2$ for any $\fv \in \mathbb{R}^{d}$ and $\lambda>0$. Next, we would like to uniformly bound the derivative
\begin{equation*}
\partial_{j} [h_{\mathrm{sp}, \lambda}(\fv)\gspzero(\fv)]=\partial_{j}\gspzero(\fv)h_{\mathrm{sp}, \lambda}(\fv) + \gspzero(\fv)\partial_{j}h_{\mathrm{sp}, \lambda}(\fv).
\end{equation*}
Clearly, we have
\begin{equation*}
\Vert \partial_{j}(\gspzero h_{\mathrm{sp}, \lambda}) \Vert_{L^{1}(\mathbb{R}^{d})} \leq \Vert (\partial_{j}\gspzero)h_{\mathrm{sp}, \lambda}  \Vert_{L^{1}(\mathbb{R}^{d})} + \Vert \gspzero (\partial_{j}h_{\mathrm{sp}, \lambda}) \Vert_{L^{1}(\mathbb{R}^{d})}.
\end{equation*}
The function $h_{\mathrm{sp}, \lambda}$ remains uniformly bounded in $\lambda$ so that once $\partial_{j}\gspzero \in L^{1}(\mathbb{R}^{d})$, we also have
\begin{equation*}
\Vert (\partial_{j}\gspzero)h_{\mathrm{sp}, \lambda}  \Vert_{L^{1}(\mathbb{R}^{d})} \leq \int_{\mathbb{R}^{d}} |\partial_{j}\gspzero(\fv)||h_{\mathrm{sp}, \lambda}(\fv)|\mathrm{d}\fv \leq 2 \Vert \partial_{}\gspzero \Vert_{L^{1}(\mathbb{R}^{d})}. 
\end{equation*}
For the other term, we compute with $a=2t\sqrt{\vartheta}\lambda_n^{-1}$
\begin{equation*}
\partial_{j}h_{\mathrm{sp}, \lambda}(\fv)=- \partial_{j} \mathrm{sinc}(a |\fv|)=-\frac{\fv_j}{|\fv|} \frac{1}{|\fv|}\left(\cos(a|\fv|)-\mathrm{sinc}(a|\fv|)\right).
\end{equation*}
In particular, the difference of the cosine and the $\mathrm{sinc}$-kernel always remains bounded by one, yielding the uniform upper bound
\begin{equation*}
\begin{aligned}
\Vert \gspzero (\partial_{j} h_{\mathrm{sp}, \lambda}) \Vert_{L^{1}(\mathbb{R}^{d})}&\leq \int_{\mathbb{R}^{d}}|\gspzero(\fv)| \left|-\frac{\fv_j}{|\fv|} \frac{1}{|\fv|}\left(\cos(a|\fv|)-\mathrm{sinc}(a|\fv|)\right)\right| \mathrm{d}\fv\\
&\leq \int_{\mathbb{R}^{d}}|\gspzero (\fv)||\fv|^{-1} \frac{|\fv_j|}{|\fv|} \mathrm{d}\fv \leq \int_{\mathbb{R}^{d}} |\gspzero(\fv)||\fv|^{-1}\mathrm{d}\fv < \infty.
\end{aligned}
\end{equation*}
Note that the additional singularity $|\fv|^{-1}$ is still counteracted by the regularity of $\sin^{4}$ around zero.
\end{proof}

\paragraph*{Acknowledgements}
The research of EZ has been partially funded by Deutsche Forschungsgemeinschaft \\(DFG)—SFB1294/1-318763901.

\printbibliography
\end{document}